\documentclass[12pt]{amsart}

\usepackage[T1]{fontenc}
\usepackage{fullpage}
\usepackage{xspace}
\usepackage{xcolor}
\usepackage{mathtools}
\usepackage{amssymb}
\usepackage{tikz-cd}
\usepackage{booktabs}
\usepackage{cite}

\theoremstyle{plain}
\newtheorem{theorem}{Theorem}[section]
\newtheorem{proposition}[theorem]{Proposition}
\newtheorem{lemma}[theorem]{Lemma}
\newtheorem{corollary}[theorem]{Corollary}

\theoremstyle{definition}
\newtheorem{definition}[theorem]{Definition}
\newtheorem{hypothesis}[theorem]{Hypothesis}

\newtheorem{example}[theorem]{Example}
\theoremstyle{remark}
\newtheorem{remark}[theorem]{Remark}
\usepackage{thmtools}

\usepackage[colorlinks, unicode, pdfencoding=auto, hypertexnames=false]{hyperref}
\usepackage{cleveref}
\crefname{hypothesis}{hypothesis}{hypotheses}
\Crefname{hypothesis}{Hypothesis}{Hypotheses}

\newenvironment{SC}[1][1]
{%
    \addtolength{\leftmargini}{1.25em}%
    \begin{enumerate}
        \setcounter{enumi}{#1}%
        \addtocounter{enumi}{-1}%
}
{%
    \end{enumerate}
}

\DeclareMathAlphabet{\mathds}{U}{bbold}{m}{n}
\DeclareMathOperator{\Hom}{Hom}             
\DeclareMathOperator{\Ext}{Ext}             
\DeclareMathOperator{\Eq}{Eq}               
\DeclareMathOperator{\Ind}{Ind}             
\DeclareMathOperator{\Pro}{Pro}             
\DeclareMathOperator{\AbObj}{Ab}            
\DeclareMathOperator*{\colim}{colim}        
\DeclareMathOperator{\Fun}{Fun}             
\DeclareMathOperator{\Map}{Map}             
\DeclareMathOperator{\id}{id}               
\DeclareMathOperator{\Aut}{Aut}             
\DeclareMathOperator{\Res}{Res}             
\DeclareMathOperator{\Coind}{Coind}         
\DeclareMathOperator{\PSh}{PSh}             
\DeclareMathOperator{\HSh}{HSh}             
\DeclareMathOperator{\Sh}{Sh}               
\DeclareMathOperator{\Spec}{Spec}           
\DeclareMathOperator{\Sym}{Sym}             
\DeclareMathOperator{\charac}{char}         
\DeclareMathOperator{\Frac}{Frac}           
\DeclareMathOperator{\diag}{Diag}           
\DeclareMathOperator{\vect}{Vect}           
\newcommand{\dual}{\vee}                    
\newcommand{\pr}{\mathrm{pr}}               

\newcommand{\zar}{\mathrm{zar}}             
\newcommand{\an}{\mathrm{an}}               
\newcommand{\et}{\mathrm{\Acute{e}t}}       
\newcommand{\fet}{\mathrm{f\Acute{e}t}}     
\newcommand{\bt}{\mathrm{bt}}               
\newcommand{\dR}{\mathrm{dR}}               
\newcommand{\alg}{\mathrm{alg}}             
\newcommand{\uni}{\mathrm{uni}}             
\newcommand{\N}{\mathrm{N}}                 
\newcommand{\rat}{\mathrm{rat}}             
\newcommand{\grp}{\mathrm{grp}}             
\newcommand{\cnt}{\mathrm{cnt}}             
\newcommand{\fm}{\mathrm{fm}}               
\newcommand{\op}{\mathrm{op}}               
\newcommand{\fin}{\mathrm{fin}}             
\newcommand{\rs}{\mathrm{rs}}               
\newcommand{\un}{\mathrm{un}}               
\newcommand{\diff}{\mathrm{diff}}           
\newcommand{\crys}{\mathrm{crys}}           
\newcommand{\mt}{\mathrm{mt}}               

\newcommand{\Set}{\mathbf{Set}}             
\newcommand{\Ab}{\mathbf{Ab}}               
\newcommand{\Vect}{\mathbf{Vec}}            
\newcommand{\Mod}{\mathbf{Mod}}             
\newcommand{\Rep}{\mathbf{Rep}}             
\newcommand{\QCoh}{\mathbf{QCoh}}           
\newcommand{\FSet}{\mathbf{FSet}}           
\newcommand{\FRep}{\mathbf{FRep}}           
\newcommand{\FC}{\mathbf{FC}}               
\newcommand{\AC}{\mathbf{AC}}               
\newcommand{\ASch}{\mathbf{ASch}}           
\newcommand{\RAlg}{\mathbf{RAlg}}           
\newcommand{\Gpd}{\mathbf{Gpd}}             
\newcommand{\PfGpd}{\mathbf{PfGpd}}         
\newcommand{\AfGpd}{\mathbf{AfGpd}}         
\newcommand{\FtGpd}{\mathbf{FtGpd}}         
\newcommand{\AMod}{\mathbf{AMod}}           
\newcommand{\UMod}{\mathbf{UMod}}           
\newcommand{\MMod}{\mathbf{MMod}}           
\newcommand{\FMod}{\mathbf{FMod}}           
\newcommand{\MIC}{\mathbf{MIC}}             
\newcommand{\FConn}{\mathbf{FConn}}         
\newcommand{\Dcat}{\mathbf{D}}              
\newcommand{\Rcat}{\mathbf{R}}              
\newcommand{\Open}{\mathbf{Open}}           
\newcommand{\Cov}{\mathbf{Cov}}             
\newcommand{\EtCov}{\mathbf{\Acute{E}tCov}} 
\newcommand{\FTor}{\mathbf{FTor}}           
\newcommand{\E}{\mathrm{E}}                 
\newcommand{\W}{\mathrm{W}}                 

\makeatletter
\newcommand{\colim@}[2]{%
    \vtop{\m@th\ialign{##\cr
        \hfil$#1\operator@font colim$\hfil\cr
        \noalign{\nointerlineskip\kern1.5\ex@}#2\cr
        \noalign{\nointerlineskip\kern-\ex@}\cr}}%
}
\renewcommand{\varinjlim}{%
    \mathop{\mathpalette\colim@{\rightarrowfill@\textstyle}}\nmlimits@
}
\makeatother

\begin{document}

\title{A Uniform Construction of Cohomology Theories of Varieties via Fundamental Groupoids}

\author{Hyuk Jun Kweon}
\address{Department of Mathematical Sciences, Seoul National University, South Korea}
\email{kweon7182@snu.ac.kr}


\begin{abstract}
    We prove that the cohomology of a variety is fully recovered from the fundamental groupoids of its Zariski open subsets in various settings. This follows from a new sheaf theory that replaces the target category of Zariski sheaves with a fibered category while retaining the Zariski site as the source. Given a fundamental groupoid functor and a suitable abelian coefficient category, this theory produces the corresponding cohomology theory on varieties, providing a uniform construction of several cohomology theories.

    The topological fundamental groupoid with abelian groups recovers singular cohomology. The \'etale fundamental groupoid with discrete abelian groups recovers \'etale cohomology. The pro-algebraic fundamental groupoid with vector spaces recovers algebraic de Rham cohomology. The pro-algebraic fundamental groupoid with commutative formal groups unifies \'etale and algebraic de Rham cohomology. In each case, we prove a comparison theorem with the corresponding classical theory or theories. We conjecture that the Nori fundamental groupoid with commutative formal groups unifies \'etale and $p$-adic cohomology.
\end{abstract}

\maketitle

\section{Introduction}

Let $X$ be a variety over a field $k$. The goal of this paper is to introduce a new sheaf theory that gives a uniform construction of several cohomology theories of $X$. For a Zariski open subset $U \subseteq X$, let $\mathfrak{G}(U)$ denote a certain fundamental groupoid of $U$. For example, $\mathfrak{G}(U)$ may be the topological, \'etale, pro-algebraic, or Nori fundamental groupoid. A sheaf $\mathcal{M}$ in our theory sends each Zariski open subset $U \subseteq X$ to a $\mathfrak{G}(U)$-module $\mathcal{M}(U)$. The associated cohomology functors are defined as the right derived functors of
\[ \mathcal{M} \mapsto \Gamma(X, \mathcal{M})^{\mathfrak{G}(X)}, \]
the composite of the global section functor and the $\mathfrak{G}(X)$-invariants functor. By varying the type of fundamental groupoid and the coefficient category for its groupoid modules, we obtain the corresponding cohomology theories of $X$, as summarized below.

\begin{table}[htpb]
\centering
\begin{tabular}{ccc}
\toprule
\textbf{Fundamental Groupoid} & \textbf{Coefficients} & \textbf{Cohomology Theory} \\
\midrule
topological & abelian groups & singular \\
\'etale & discrete abelian groups & \'etale \\
pro-algebraic & vector spaces over $k$ & de Rham \\
pro-algebraic & commutative formal groups & \'etale $+$ de Rham \\
Nori & commutative formal groups & \'etale $+$ $p$-adic \\
\bottomrule
\end{tabular}
\end{table}

The precise hypotheses on $X$ and $k$ are stated in the comparison theorems below. For the first four rows, we prove comparison theorems showing that the resulting cohomology theory recovers the indicated classical cohomology theory or theories. In particular, the fourth row shows that \'etale and algebraic de Rham cohomology can be recovered from a single algebraic cohomology theory with coefficients in commutative formal groups. This is worth noting because these two theories are traditionally constructed in rather different ways. Here, the pro-algebraic fundamental groupoid means the Tannaka groupoid of the category of regular singular integrable connections. The last row, however, is conjectural. We recall the definitions of the fundamental groupoids used in the table in \Cref{sec:appendix_groupoids}.

Classically, a sheaf on a topological space $X$ is a contravariant functor
\[ \mathcal{F} \colon \Open(X)^\op \to \Ab, \]
subject to a gluing condition. In arithmetic geometry, however, the Zariski topology is often too coarse for cohomological purposes. Grothendieck's remedy was to replace the category of open subsets by a site. For example, the \'etale site $X_\et$, whose objects are \'etale morphisms to $X$, gives sheaves of the form
\[ \mathcal{F} \colon X_\et^\op \to \Ab. \]
This construction led to \'etale cohomology \cite{SGA4_1,Milne}, culminating in the proof of the Weil conjectures \cite{Weil,Grothendieck2,SGA5,Deligne2,Deligne3}.

However, this source-replacing approach leads to heterogeneous constructions for the other cohomology theories of arithmetic geometry. Algebraic de Rham cohomology is traditionally constructed from the de Rham complex \cite{Grothendieck1,Hartshorne}, while crystalline cohomology is constructed from the crystalline site \cite{Berthelot}. These constructions are philosophically quite different from \'etale cohomology. In particular, they are not obtained by further refining the \'etale site.

The goal of this paper is to show that this heterogeneity can be resolved by replacing the target category instead. We retain the Zariski site as the source and replace the fixed target category by a fibered category built from fundamental groupoids and a coefficient category. With sufficiently large fundamental groupoids and a sufficiently rich coefficient category, this construction even produces a single unified algebraic cohomology theory that recovers both \'etale and algebraic de Rham cohomology. We conjecture that the same framework also unifies \'etale and $p$-adic cohomology in characteristic $p > 0$. This leads to the following notion of sheaf.

\begin{definition}
    Given a topological space $X$, let $p \colon \mathcal{T}^\op \to \Open(X)$ be a Grothendieck fibration equipped with a chosen cleavage. A \emph{presheaf} on $X$ with values in $\mathcal{T}$ is a functor
    \[ \mathcal{M} \colon \Open(X)^\op \to \mathcal{T} \]
    satisfying $p \circ \mathcal{M}^\op = \id_{\Open(X)}$. Such a presheaf $\mathcal{M}$ is a \emph{sheaf} if, for every $U \in \Open(X)$ and every open cover $\mathcal{U}$ of $U$, the diagram
    \[ \mathcal{M}(U) \to \prod_{V \in \mathcal{U}} \mathcal{M}(V) \rightrightarrows \prod_{V, W \in \mathcal{U}} \mathcal{M}(V \cap W) \]
    is an equalizer diagram in $\mathcal{T}$. We denote the category of presheaves and its full subcategory of sheaves by $\PSh(X, \mathcal{T})$ and $\Sh(X, \mathcal{T})$, respectively.
\end{definition}

Requiring $\mathcal{M}^\op$ to be a section of $p$ is our only modification to the classical definitions of presheaves and sheaves. A satisfactory sheaf theory should admit a sheafification functor $(-)^\# \colon \PSh(X, \mathcal{T}) \to \Sh(X, \mathcal{T})$, left adjoint to the forgetful functor $\Sh(X, \mathcal{T}) \to \PSh(X, \mathcal{T})$. To define cohomology by right derived functors, the category $\Sh(X, \mathcal{T})$ should also have enough injectives. We prove both properties under the following four mild conditions on the fibers and restriction functors. For every inclusion $j \colon V \to U$ in $\Open(X)$, write $j^* \colon \mathcal{T}(U) \to \mathcal{T}(V)$ for the corresponding restriction functor.

\begin{SC}
    \item The functor $j^*$ admits a right adjoint $j_*$. 
    \item The category $\mathcal{T}(U)$ is locally finitely presentable.
    \item The functor $j^*$ is left exact.
    \item The category $\mathcal{T}(U)$ is an abelian category.
\end{SC}

An important example arises from a fundamental groupoid functor $\mathfrak{G}$. Let $\mathcal{A}$ be a suitable abelian category, called the \emph{coefficient category}, such that for every $U \in \Open(X)$, we can define $\mathfrak{G}(U)$-module structures on objects of $\mathcal{A}$. As $U$ varies, the resulting categories of $\mathfrak{G}(U)$-modules assemble into a fibration $p \colon \mathcal{T}^\op \to \Open(X)$, whose fiber $\mathcal{T}(U)$ is the category of $\mathfrak{G}(U)$-modules.

For example, let $X$ be a complex variety, and let $\Pi_1$ denote the topological fundamental groupoid functor. Taking $\mathfrak{G} = \Pi_1$ and $\mathcal{A} = \Ab$ gives a fibered category $\mathcal{T}$ over $\Open(X)$. This fibered category can equivalently be defined without reference to $\Pi_1$. For every $U \in \Open(X)$, define the category by
\[ \Cov(U) \coloneqq \text{\{topological coverings of $U^\an$\}}. \]
Then $\mathcal{T}$ is the fibered category with fibers $\mathcal{T}(U) \simeq \AbObj(\Cov(U))$ and restriction functors induced by pullback of coverings.
The cohomology functor $H^i_\bt(X, -) \colon \Sh(X, \mathcal{T}) \to \Ab$ is the $i$-th right derived functor of $\mathcal{M} \mapsto \Gamma(X, \mathcal{M})^{\Pi_1(X)}$. For an abelian group $M$, the constant sheaf $\underline{M}$ sends each Zariski open subset $U$ to the trivial $\Pi_1(U)$-module associated to $M$.

\begin{theorem}[see \Cref{thm:betti_comparison}]\label{thm:intro_betti_comparison}
    Let $X$ be a smooth complex variety, and let $M$ be an abelian group. Then the natural map
    \[ H^i_\bt(X, \underline{M}) \to H^i(X^\an, M) \]
    is an isomorphism for all $i \ge 0$.
\end{theorem}

Similarly, let $X$ be a scheme, and let $\Pi_1^\et$ denote the \'etale fundamental groupoid functor. Taking $\mathfrak{G} = \Pi_1^\et$ and discrete abelian groups as coefficients, we obtain a fibered category $\mathcal{T}$. Equivalently, for $U \in \Open(X)$, define $\mathcal{T}(U) \coloneqq \AbObj(\Ind(\EtCov(U)))$, where
\[ \EtCov(U) \coloneqq \text{\{finite \'etale covers of $U$\}}. \]
We obtain a cohomology functor $H^\bullet_\fet(X, -)$. For a discrete abelian group $M$, the constant sheaf $\underline{M}$ sends each Zariski open subset $U$ to the trivial $\Pi_1^\et(U)$-module associated to $M$.

\begin{theorem}[see \Cref{thm:profinite_etale_comparison}]\label{thm:intro_etale_comparison}
    Let $X$ be a smooth variety over a field of characteristic $0$, or any Noetherian scheme over $\mathbb{F}_p$. Let $M$ be a torsion abelian group. The natural map
    \[ H^i_\fet(X, \underline{M}) \to H^i(X_\et, M) \]
    is an isomorphism for all $i \ge 0$.
\end{theorem}

Let $X$ be a smooth variety over an algebraically closed field $k$ of characteristic $0$, and let $\Pi_1^\alg$ denote the pro-algebraic fundamental groupoid functor. Taking $\mathfrak{G} = \Pi_1^\alg$ and $k$-vector spaces as coefficients, we obtain a fibered category $\mathcal{T}$. Equivalently, for $U \in \Open(X)$, define $\mathcal{T}(U) \coloneqq \Ind(\MIC^\rs(U))$, where
\[ \MIC^\rs(U) \coloneqq \text{\{vector bundles with regular singular integrable connections on $U$\}}. \]
We obtain a cohomology functor $H^\bullet_\alg(X, -)$. The constant sheaf $\underline{k}$ sends each Zariski open subset $U$ to the trivial one-dimensional $\Pi_1^\alg(U)$-representation.

\begin{theorem}[see \Cref{thm:alg_dR_comparison}]\label{thm:intro_dR_comparison}
    Let $X$ be a smooth variety over an algebraically closed field $k$ of characteristic $0$. Then the natural map
    \[ H^i_\alg(X, \underline{k}) \to H^i_\dR(X) \]
    is an isomorphism for all $i \ge 0$.
\end{theorem}

Let $\FC$ denote the category of commutative formal groups over a field $k$ \cite[Ch.~II, \S4]{Demazure}. We first assume that $k$ is algebraically closed of characteristic $0$. For an abelian group $D$, let $\underline{D} \in \FC$ denote the constant formal group associated to $D$. Let $\mathbb{G}_a^\vee \in \FC$ denote the formal additive group. For $M \in \FC$, let $\mathrm{D}(M) \coloneqq \Hom_\FC(\underline{\mathbb{Z}}, M)$ denote its \emph{discrete part}, and let $\mathrm{V}(M) \coloneqq \Hom_\FC(\mathbb{G}_a^\vee, M)$ denote its \emph{vector space part}.

Now let $X$ be a smooth variety over $k$. Taking $\mathfrak{G} = \Pi_1^\alg$ and $\FC$ as the coefficient category, we obtain a cohomology functor $H^\bullet_\uni(X, -)$. For $M \in \FC$, the constant sheaf $\underline{M}$ sends each Zariski open subset $U$ to the trivial $\Pi_1^\alg(U)$-module associated to $M$.

\begin{theorem}[see \Cref{thm:uni_comparison}]\label{thm:intro_unified_comparison}
    Let $X$ be a smooth variety over an algebraically closed field $k$ of characteristic $0$. Let $D$ be a torsion abelian group. Then the natural maps
    \[ \mathrm{D}(H^i_\uni(X, \underline{\underline{D}})) \to H^i(X_\et, D) \quad\text{and}\quad \mathrm{V}(H^i_\uni(X, \underline{\mathbb{G}_a^\vee})) \to H^i_\dR(X) \]
    are isomorphisms for all $i \ge 0$.
\end{theorem}

For simplicity, the four comparison theorems above are stated only for constant sheaves. Each is a special case of a comparison theorem for locally constant sheaves proved in the corresponding section.

Finally, let $X$ be a reduced variety over a perfect field $k$ of characteristic $p > 0$, and let $\Pi_1^\N$ denote the Nori fundamental groupoid functor. Taking $\mathfrak{G} = \Pi_1^\N$ and $\FC$ as the coefficient category, we obtain a fibered category $\mathcal{T}$. Equivalently, for $U \in \Open(X)$, define $\mathcal{T}(U) \coloneqq \AbObj(\Ind(\FTor(U)))$, where
\[ \FTor(U) \coloneqq \text{\{torsors over $U$ under finite $k$-group schemes\}}. \]
We obtain a cohomology functor $H^\bullet_\N(X, -)$. Let $\mathrm{CW} \in \FC$ denote the formal $p$-group representing Fontaine's Witt covectors \cite[Ch.~II, \S1, and Ch.~III, \S I.1]{Fontaine}. Let $\W(k)$ denote the ring of $p$-typical Witt vectors of $k$.

\begin{restatable*}{conjecture}{padicCohomologyConjecture}\label{conj:padic_cohomology}
    Let $X$ be a smooth projective variety over a perfect field $k$ of characteristic $p > 0$. For every $i \ge 0$, let
    \[ H^i(X) \coloneqq \left( \varprojlim_n H^i_\N(X, \underline{\mathrm{CW}})[p^n](k) \right) \otimes_{\W(k)} \Frac \W(k). \]
    Then $H^\bullet(X)$ is a $p$-adic Weil cohomology theory.
\end{restatable*}

One caveat is that $H^i_\N(X, \underline{\mathrm{CW}})$ may contain summands with no counterpart in classical crystalline cohomology. Besides $\mathrm{CW}$ and $\W_n$, possible factors include $\mathbb{Z}/p\mathbb{Z}$, $\boldsymbol{\alpha}_p$, and $\boldsymbol{\mu}_p$. Here, $\W_n \in \FC$ denotes the formal group representing the truncated $p$-typical Witt vector group functor of length $n$. It is unclear whether these additional factors would remedy some of the pathologies caused by torsion in existing integral $p$-adic cohomology theories or instead introduce further complications.

The cohomology theories constructed here share some features. For example, their defining functors admit two factorizations, each yielding a Grothendieck spectral sequence; see \Cref{rem:spectral_sequences}. Moreover, each cohomology theory has a corresponding groupoid cohomology theory; see \Cref{def:sheaf_cohomology}. For $H^\bullet_\bt$, $H^\bullet_\fet$, and $H^\bullet_\alg$, the corresponding cohomology theories for groups are the well-studied ordinary cohomology of abstract groups, continuous cohomology of profinite groups \cite[Ch.~I, \S2]{Serre}, and rational cohomology of affine group schemes \cite[Ch.~4]{Jantzen}, respectively. Those corresponding to $H^\bullet_\uni$ and $H^\bullet_\N$ instead use commutative formal groups as coefficients and appear not to have been studied. In particular, $H^\bullet_\N$ suggests a notion of $p$-adic group cohomology for profinite group schemes over a perfect field of characteristic $p > 0$.

In \Cref{sec:notation}, we fix the notation used throughout the paper. \Cref{sec:analytic} introduces the Betti site $X_\bt$, describes its sheaves in terms of the topological fundamental groupoid, and proves that the resulting cohomology theory $H^\bullet_\bt$ recovers analytic sheaf cohomology. \Cref{sec:formalization} develops the general formalism of sheaves valued in a fibered category, including sheafification, enough injectives, sheaf cohomology, and a comparison theorem. \Cref{sec:profinite} applies this formalism to the \'etale fundamental groupoid, defines $H^\bullet_\fet$, and proves comparison theorems with analytic sheaf cohomology and \'etale cohomology. \Cref{sec:algebraic} applies this formalism to the pro-algebraic fundamental groupoid, defines $H^\bullet_\alg$, and proves comparison theorems with analytic sheaf cohomology and algebraic de Rham cohomology. \Cref{sec:unified} applies this formalism to the pro-algebraic fundamental groupoid with coefficients in commutative formal group schemes, defines $H^\bullet_\uni$, and proves that this theory recovers both \'etale and algebraic de Rham cohomology. \Cref{sec:further_remarks} defines $H^\bullet_\N$ using the Nori fundamental groupoid with coefficients in commutative formal groups and discusses its conjectural relation to $p$-adic cohomology. It also discusses an algebraic Poincar\'e lemma and the universal cohomology theory. Finally, \Cref{sec:appendix_groupoids} reviews background material on groupoids, profinite groupoids, and affine groupoid schemes.

\section{Notation}\label{sec:notation}

Throughout this paper, $k$ denotes the base field. We denote by $\Set$, $\Ab$, and $\Vect(k)$ the categories of sets, abelian groups, and $k$-vector spaces, respectively. For categories $\mathcal{C}$ and $\mathcal{A}$, we denote by $\Fun(\mathcal{C}, \mathcal{A})$ the category of functors from $\mathcal{C}$ to $\mathcal{A}$. We denote by $\Ind(\mathcal{C})$ and $\Pro(\mathcal{C})$ the ind- and pro-completions of $\mathcal{C}$, respectively, and by $\AbObj(\mathcal{C})$ the category of abelian group objects in $\mathcal{C}$.

Throughout, by a \emph{fibration} $p \colon \mathcal{T}^\op \to \mathcal{B}$, we mean a Grothendieck fibration equipped with a chosen cleavage. We refer to $\mathcal{T}$ as a \emph{fibered category} over $\mathcal{B}$. For a morphism $f \colon V \to U$ in $\mathcal{B}$, we call the opposite $f^* \colon \mathcal{T}(U) \to \mathcal{T}(V)$ of the pullback functor determined by the cleavage the \emph{restriction functor}. A \emph{bifibration} is a fibration for which $p^\op$ is also a fibration. All fibrations below are understood in this sense.

All limits and colimits are small unless otherwise specified. We say that a diagram
\[ A \to \prod_{i \in I} B_i \rightrightarrows \prod_{i, j \in I} C_{i,j} \]
is an \emph{equalizer diagram} if the induced morphism from $A$ to the limit of the diagram formed by the morphisms $B_i \to C_{i,j}$ and $B_j \to C_{i,j}$ for all $i, j \in I$ is an isomorphism. The products are merely notational shorthand and are not assumed to exist in the ambient category.

We denote by $\Gpd$ the category of groupoids. For $G \in \Gpd$, we denote by $\Set(G)$, $\Mod(G)$, and $\Rep(G)$ the categories of $G$-sets, $G$-modules, and complex $G$-representations, respectively. All actions are left unless otherwise specified. We denote the full subcategories of finite $G$-sets and finite-dimensional complex $G$-representations by $\FSet(G)$ and $\FRep(G)$, respectively.

Given $M$ in any of these categories and an object $x$ of $G$, we regard $M$ as a functor and write $M_x$ for $M(x)$. For a morphism of groupoids $H \to G$, restriction and coinduction form an adjoint pair
\[ \Res_H^G : \Set(G) \rightleftarrows \Set(H) : \Coind_H^G. \]
Let $\mathds{1}$ denote the terminal groupoid, which has a single object and a single morphism, so that $\Set(\mathds{1}) \cong \Set$. The \emph{trivial $G$-set functor} $\Delta_G \coloneqq \Res_G^{\mathds{1}}$ and the \emph{$G$-invariants functor} $(-)^G \coloneqq \Coind_G^{\mathds{1}}$ form an adjoint pair
\[ \Delta_G : \Set \rightleftarrows \Set(G) : (-)^G. \]
The same constructions and notation apply to $\Mod(G)$ and $\Rep(G)$.

For an abstract group $G$, we denote its profinite completion by $\widehat{G}$ and its pro-algebraic completion over $k$ by $G^\alg$. The latter is defined as the Tannaka group of $\FRep(G)$ with respect to the forgetful fiber functor. We write $H^\bullet_\grp(G, -)$ for the group cohomology of $G$. For a profinite group $G$, we write $H^\bullet_\cnt(G, -)$ for its continuous cohomology \cite[Ch.~I, \S2]{Serre}. For an affine group scheme $G$ over $k$, we write $H^\bullet_\rat(G, -)$ for its rational cohomology \cite[Ch.~4]{Jantzen}. We use the same subscripts for the corresponding groupoid cohomology.

Let $X$ be a topological space. We denote by $\Open(X)$ the category of open subsets of $X$ and generally regard $X$ as a site via $\Open(X)$. When $X$ is a complex variety, we denote by $X^\an$ the space $X$ equipped with the analytic topology. For a scheme $X$ over $k$, we denote its Zariski and \'etale sites by $X^\zar$ and $X_\et$, respectively. We often omit the superscript $\zar$. Unless otherwise specified, all open subsets of a scheme are assumed to be Zariski open.

For a site $\mathcal{C}$ and a category $\mathcal{A}$, we denote the categories of presheaves and sheaves on $\mathcal{C}$ with values in $\mathcal{A}$ by $\PSh(\mathcal{C}, \mathcal{A})$ and $\Sh(\mathcal{C}, \mathcal{A})$, respectively. If $\mathcal{A}$ is a suitable abelian category, we write $H^i(\mathcal{C}, -)$ for the $i$-th sheaf cohomology functor. In particular, we write analytic and \'etale sheaf cohomology as $H^i(X^\an, -)$ and $H^i(X_\et, -)$, respectively.

For appropriate $X$, we denote the topological fundamental groupoid, the \'etale fundamental groupoid \cite[Exp.~V, \S7]{SGA1}, the pro-algebraic fundamental groupoid \cite[Proposition~10.32(a)]{Deligne1}, and the Nori fundamental groupoid of $X$ by $\Pi_1(X)$, $\Pi_1^\et(X)$, $\Pi_1^\alg(X)$, and $\Pi_1^\N(X)$, respectively.

The name and notation for the pro-algebraic fundamental groupoid $\Pi_1^\alg(X)$ vary considerably in the literature. In this paper, $\Pi_1^\alg(X)$ is the Tannaka groupoid of the category
\[ \{ \text{vector bundles on $X$ with regular singular integrable connections} \} \]
with respect to the fiber functor over $X$ that forgets the connection.
Our superscript follows \cite[Proposition~10.32(a)]{Deligne1}. Over $\mathbb{C}$, our terminology and notation agree with the conventions of algebraic topology and complex geometry. Some references denote our $\Pi_1^\alg(X)$ by $\Pi_1^{\alg,\rs}(X)$ and reserve $\Pi_1^\alg(X)$ for the Tannaka groupoid of all integrable connections. We refer to \Cref{sec:appendix_groupoids} for background on the groupoids and fundamental groupoids used in this paper, including the existence of coinduction functors.

\section{Analytic Theory}\label{sec:analytic}

Throughout this section, let $X$ be a complex variety. We first define the Betti site $X_\bt$, whose objects are topological coverings over Zariski open subsets of $X$. We then describe $\Sh(X_\bt, \Set)$ using only the Zariski topology and the topological fundamental groupoid. Finally, we compare sheaf cohomology on $X_\bt$ with analytic sheaf cohomology on $X^\an$.

The alternative description of $\Sh(X_\bt, \Set)$ is particularly significant because it admits natural generalizations obtained by substituting other types of fundamental groupoids, such as the \'etale, pro-algebraic, or Nori fundamental groupoid, in place of the topological one. These substitutions yield cohomology theories that are, in general, difficult to obtain from Grothendieck topologies.

Before proceeding, we recall the well-known equivalences used throughout this section. For $U \in \Open(X)$, let $\Cov(U)$ denote the site of topological covering spaces over $U^\an$, equipped with the jointly surjective topology. There are natural equivalences
\begin{equation}\label{eq:topological_covering_groupoid_equivalence}
    \Phi_U : \Sh(\Cov(U), \Set) \rightleftarrows \Set(\Pi_1(U)) : \Psi_U.
\end{equation}
To give an explicit description, for $x \in U^\an$, let $P_x$ denote the universal cover of the connected component of $U^\an$ containing $x$, based at $x$. We view $P_x$ as the set of homotopy classes of paths in $U^\an$ starting at $x$. Then
\begin{align*}
    \Phi_U \colon \Sh(\Cov(U), \Set) &\xrightarrow{\sim} \Set(\Pi_1(U)), \\
    \mathcal{F} &\mapsto \left(x \mapsto \mathcal{F}(P_x)\right).
\end{align*}
Moreover, for $E \in \Cov(U)$ and $x \in U^\an$, let $E_x$ denote the fiber of $E$ over $x$, and write $E_{(-)}$ for the corresponding $\Pi_1(U)$-set. Then
\begin{align*}
    \Psi_U \colon \Set(\Pi_1(U)) &\xrightarrow{\sim} \Sh(\Cov(U), \Set), \\
    M &\mapsto \left(E \mapsto \Hom_{\Set(\Pi_1(U))}\left(E_{(-)}, M\right)\right).
\end{align*}
Under these equivalences, groupoid cohomology of $\Pi_1(U)$ coincides with sheaf cohomology on $\Cov(U)$. When $U$ is connected, groupoid cohomology coincides with group cohomology of $\pi_1(U, x)$ with locally constant coefficients.

\subsection{Betti Site}\label{sec:betti_site}

To motivate the construction, assume temporarily that $X$ is smooth and connected. Zariski cohomology with constant coefficients does not yield a Weil cohomology theory. For example, no nonempty Zariski open subset of $\mathbb{A}^1_{\mathbb{C}} \setminus \{0\}$ is contractible in the analytic topology. Nevertheless, as Artin observed, every smooth complex variety is Zariski locally a $K(\pi, 1)$ space \cite[Exp.~XI, Proposition~3.3 and Variante~4.6]{SGA4_3}. A connected manifold is a \emph{$K(\pi, 1)$ space} if its universal cover is contractible, or equivalently, if its higher homotopy groups $\pi_i$ vanish for all $i > 1$.

For locally constant coefficients, the cohomology of a $K(\pi, 1)$ space coincides with the group cohomology of its topological fundamental group \cite{Hurewicz,Eilenberg}. This suggests that the global cohomology of $X$ can be recovered by gluing the cohomological data of the topological fundamental groups of its Zariski open subsets. Thus, combining $X^\zar$ and $\Cov(U)$ for every $U \in \Open(X)$ into a single site yields a cohomology theory that glues these local group-cohomological data. From another perspective, the resulting site is locally contractible. This principle explains why the site of local homeomorphisms as well as its variant, the \'etale site, yields the correct cohomology.

Among the sites suggested by this principle, we deliberately choose the coarsest one. This choice not only clarifies the essential role of the site but also provides a framework that extends beyond Grothendieck topologies, as we explain in the next subsection. We now drop the temporary assumption and again let $X$ be an arbitrary complex variety.

\begin{definition}\label{def:category_Bt}
    The \emph{Betti site} $X_\bt$ of $X$ is defined as follows.
    \begin{enumerate}
        \item \textbf{Objects:} A continuous map $p_E\colon E \to X^\an$ such that the image $p_E(E) \subseteq X$ is a Zariski open subset and the induced map $E \to p_E(E)$ is an analytic covering.
        \item \textbf{Morphisms:} A continuous map $f\colon E \to F$ making the following diagram commute.
        \[
        \begin{tikzcd}[column sep=1em, row sep=2.5em]
            E \arrow[rr, "f"] \arrow[dr, "p_E"'] & & F \arrow[dl, "p_F"] \\
            & X^\an &
        \end{tikzcd}
        \]
    \end{enumerate}
    We endow $X_\bt$ with the Grothendieck topology generated by Zariski open covers in $X^\zar$ and jointly surjective families in $\Cov(U)$ for every $U \in \Open(X)$. Objects of $X^\zar$ are regarded as trivial coverings.
\end{definition}

The assignment $E \mapsto p_E(E)$ defines a fibration
\[ p\colon X_\bt \to \Open(X), \]
whose fiber over $U$ is $\Cov(U)$. For $V \subseteq U$ and $E \in \Cov(U)$, its cleavage is given by the projection $E \times_X V \to E$. In the sense of Abbes--Gros \cite[VI.5.1 and VI.5.3]{AbbesGros}, the topology above makes $X_\bt$ the covanishing site associated with the fibered site $p$.

For brevity, we denote an object $p_E\colon E \to X^\an$ simply by $E$, so $p(E) = p_E(E)$. We write $E_\bt$ for the site obtained by restricting $X_\bt$ to $E$. We use $E^\an$ when we want to emphasize the underlying topological space of $E$. A morphism $f\colon E \to F$ satisfies $p(E) \subseteq p(F)$ and induces a morphism of covering spaces $E \to F \times_X p(E)$ over $p(E)$.

In the remainder of this section, we describe the sheaf condition on $X_\bt$ in terms of the sheaf conditions on $X^\zar$ and $\Cov(U)$ for every $U \in \Open(X)$.

\begin{definition}\label{def:half_sheaf}
    For a presheaf $\mathcal{F} \in \PSh(X_\bt, \Set)$ and $U \in \Open(X)$, define the restricted presheaf $\mathcal{F}_U \in \PSh(\Cov(U), \Set)$ by
    \[ \mathcal{F}_U(E) \coloneqq \mathcal{F}(E) \quad \text{for all } E \in \Cov(U). \]
    We say that $\mathcal{F}$ is \emph{sheafified along coverings} if $\mathcal{F}_U \in \Sh(\Cov(U), \Set)$ for every $U \in \Open(X)$. We call such a presheaf a \emph{half-sheaf} and denote the full subcategory of half-sheaves by $\HSh(X_\bt, \Set)$.
\end{definition}

Clearly, $\Sh(X_\bt, \Set)$ is a full subcategory of $\HSh(X_\bt, \Set)$.

\begin{definition}\label{def:zariski_sheafified}
    For a presheaf $\mathcal{F} \in \PSh(X_\bt, \Set)$ and an object $E \in \Cov(U)$ with $U \in \Open(X)$, define the presheaf $\mathcal{F}^E \in \PSh(U, \Set)$ by
    \[ \mathcal{F}^E(V) \coloneqq \mathcal{F}(E \times_X V) \quad \text{for all } V \subseteq U. \]
    We say that $\mathcal{F}$ is \emph{sheafified along the Zariski topology} if $\mathcal{F}^E \in \Sh(U, \Set)$ for every $E \in X_\bt$.
\end{definition}

By \cite[Proposition~VI.5.10]{AbbesGros}, we obtain the following criterion.

\begin{proposition}\label{prop:sheaf_criterion}
    A presheaf $\mathcal{F} \in \PSh(X_\bt, \Set)$ is a sheaf if and only if it is sheafified along coverings and along the Zariski topology.
\end{proposition}

\subsection{Sheaf Theory via Fundamental Groupoids}\label{sec:sheaf_via_groupoid}

This subsection provides an alternative description of the category $\Sh(X_\bt, \Set)$. In this framework, sheaves are characterized solely in terms of the Zariski topology and the topological fundamental groupoid.

The key observation is that given $\mathcal{F} \in \HSh(X_\bt, \Set)$, the family $\{ \mathcal{F}_U \}_{U \in \Open(X)}$ determines all sections of $\mathcal{F}$. Thus, via the equivalences \eqref{eq:topological_covering_groupoid_equivalence}, a half-sheaf or a sheaf on $X_\bt$ corresponds to a family $\{ \mathcal{M}(U) \}_{U \in \Open(X)}$, where $\mathcal{M}(U) \in \Set(\Pi_1(U))$, subject to some compatibility conditions. This motivates the definitions below of the fibered categories $\Set_X$, $\Mod_X$, and $\Rep_X$ as target categories for sheaves on $X$.

\begin{definition}\label{def:set_X}
    The fibration $p \colon \Set_X^\op \to \Open(X)$ is defined by $\Set_X(U) \coloneqq \Set(\Pi_1(U))$ with the usual restriction functors.
    Concretely, for $V \subseteq U$ in $\Open(X)$, the fibered category $\Set_X$ is described as follows.
    \begin{enumerate}
        \item \textbf{Objects:} A $\Pi_1(U)$-set $M$.
        \item \textbf{Morphisms:} A groupoid-set morphism
        \[ \Res_{\Pi_1(V)}^{\Pi_1(U)} M \to N, \text{ or equivalently } M \to \Coind_{\Pi_1(V)}^{\Pi_1(U)} N. \]
        \item \textbf{Fibration:} The forgetful functor $M \mapsto U$.
        \item \textbf{Cleavage:} The choice $M \to \Res_{\Pi_1(V)}^{\Pi_1(U)} M$ induced by the identity morphism.
    \end{enumerate}
    In a similar way, the fibrations $\Mod_X^\op \to \Open(X)$ and $\Rep_X^\op \to \Open(X)$ are defined by $\Mod_X(U) \coloneqq \Mod(\Pi_1(U))$ and $\Rep_X(U) \coloneqq \Rep(\Pi_1(U))$, respectively, with the usual restriction functors.
\end{definition}

For simplicity, let $\Res_V^U \coloneqq \Res_{\Pi_1(V)}^{\Pi_1(U)}$ and $\Coind_V^U \coloneqq \Coind_{\Pi_1(V)}^{\Pi_1(U)}$ for $V \subseteq U$ in $\Open(X)$.

\begin{remark}
    The subsequent constructions and results apply, mutatis mutandis, to $\Mod_X$ and $\Rep_X$. Until cohomology is defined, we therefore treat only $\Set_X$.
\end{remark}

\begin{definition}\label{def:presheaf_set_X}
    A \emph{presheaf} on $X$ with values in $\Set_X$ is a functor
    \[ \mathcal{M} \colon \Open(X)^\op \to \Set_X \]
    satisfying $p \circ \mathcal{M}^\op = \id_{\Open(X)}$. We denote by $\PSh(X, \Set_X)$ the full subcategory of $\Fun(\Open(X)^\op, \Set_X)$ consisting of such presheaves. The categories $\PSh(X, \Mod_X)$ and $\PSh(X, \Rep_X)$ are defined analogously.
\end{definition}

We now show that $\PSh(X, \Set_X)$ is equivalent to $\HSh(X_\bt, \Set)$. For $V \subseteq U$ in $\Open(X)$, the continuous functor
\begin{align*}
    \jmath_V^U \colon \Cov(U) &\to \Cov(V), \\
    E &\mapsto E \times_X V
\end{align*}
induces the direct image functor
\begin{align*}
    \kappa_V^U \coloneqq (\jmath_V^U)_* \colon \Sh(\Cov(V), \Set) &\to \Sh(\Cov(U), \Set), \\
    \mathcal{F} &\mapsto (E \mapsto \mathcal{F}(E \times_X V)).
\end{align*}
Recall that a half-sheaf $\mathcal{F}$ on $X_\bt$ is a collection of sheaves $\mathcal{F}_U \in \Sh(\Cov(U), \Set)$ for every $U \in \Open(X)$, with compatibility conditions between them. The presheaf structure of $\mathcal{F}$ induces natural restriction maps relating the sheaves in the family $\{\mathcal{F}_U\}_{U \in \Open(X)}$. Specifically, for $V \subseteq U$ in $\Open(X)$ and $E \in \Cov(U)$, the restriction map $\mathcal{F}(E) \to \mathcal{F}(E \times_X V)$ corresponds to a morphism
\[ \rho_V^U \colon \mathcal{F}_U \to \kappa_V^U \mathcal{F}_V \]
in $\Sh(\Cov(U), \Set)$. These morphisms satisfy the conditions
\[ \rho_U^U = \id_{\mathcal{F}_U} \quad\text{and}\quad \rho_W^U = \kappa_V^U(\rho_W^V) \circ \rho_V^U \]
for all $W \subseteq V \subseteq U$ in $\Open(X)$. Conversely, these data determine the presheaf $\mathcal{F}$.

\begin{lemma}\label{lem:equiv_hsh}
    The assignment $\mathcal{F} \mapsto (\{ \mathcal{F}_U \}_U, \{ \rho_V^U \}_{V \subseteq U})$ establishes an equivalence of categories between $\HSh(X_\bt, \Set)$ and the category of data consisting of the following two pieces.
    \begin{enumerate}
        \item\label{item:hsh_sheaves} Sheaves $\mathcal{F}_U \in \Sh(\Cov(U), \Set)$ for every $U \in \Open(X)$.
        \item\label{item:hsh_cocycle} Morphisms $\rho_V^U \colon \mathcal{F}_U \to \kappa_V^U \mathcal{F}_V$ satisfying
        \[ \rho_U^U = \id_{\mathcal{F}_U} \quad\text{and}\quad \rho_W^U = \kappa_V^U(\rho_W^V) \circ \rho_V^U \]
        for all $W \subseteq V \subseteq U$ in $\Open(X)$.
    \end{enumerate}
\end{lemma}
\begin{proof}
    The only non-trivial part of the proof is to fully recover the restriction maps of $\mathcal{F}$ from the given data $(\{ \mathcal{F}_U \}, \{ \rho_V^U \})$. For an object $E \in X_\bt$ with $p(E) = U$, we define $\mathcal{F}(E) \coloneqq \mathcal{F}_U(E)$. Given a morphism $f \colon E \to F$ in $X_\bt$, we necessarily have $U \coloneqq p(E) \subseteq p(F) \eqqcolon V$. The map $f$ factors uniquely as
    \[ E \xrightarrow{g} F \times_X U \to F, \]
    where $g$ is a morphism in $\Cov(U)$. We define the restriction map $\mathcal{F}(f) \colon \mathcal{F}(F) \to \mathcal{F}(E)$ by the composition
    \[ \mathcal{F}_V(F) \xrightarrow{\rho_U^V(F)} \mathcal{F}_U(F \times_X U) \xrightarrow{\mathcal{F}_U(g)} \mathcal{F}_U(E). \]
    The cocycle condition, together with the naturality of the maps $\rho_V^U$, ensures that the assignment preserves composition, making $\mathcal{F}$ a well-defined presheaf. This construction provides the inverse to the original assignment.
\end{proof}

For all $V \subseteq U$ in $\Open(X)$, the restriction maps of $\mathcal{M} \in \PSh(X, \Set_X)$ factor uniquely as
\[ \mathcal{M}(U) \xrightarrow{r_V^U} \Coind_{\Pi_1(V)}^{\Pi_1(U)} \mathcal{M}(V) \to \mathcal{M}(V). \]
The presheaf axioms for $\mathcal{M}$ are then equivalent to
\[ r_U^U = \id_{\mathcal{M}(U)} \quad\text{and}\quad r_W^U = \Coind_{\Pi_1(V)}^{\Pi_1(U)}(r_W^V) \circ r_V^U \]
for all $W \subseteq V \subseteq U$ in $\Open(X)$.

\begin{lemma}\label{lem:kappa_coind}
    There is a natural isomorphism
    \[ \kappa_V^U \circ \Psi_V \cong \Psi_U \circ \Coind_{\Pi_1(V)}^{\Pi_1(U)}. \]
\end{lemma}
\begin{proof}
    Given $M \in \Set(\Pi_1(V))$ and $E \in \Cov(U)$, we have
    \begin{align*}
        \kappa_V^U (\Psi_V(M))(E)
        &= \Psi_V(M)(E \times_X V) \\
        &= \Hom_{\Set(\Pi_1(V))}((E \times_X V)_{(-)}, M)\\
        &\cong \Hom_{\Set(\Pi_1(V))}\left(\Res_{\Pi_1(V)}^{\Pi_1(U)}E_{(-)}, M\right)\\
        &\cong \Hom_{\Set(\Pi_1(U))}\left(E_{(-)}, \Coind_{\Pi_1(V)}^{\Pi_1(U)}(M)\right) \\
        &= \Psi_U\left(\Coind_{\Pi_1(V)}^{\Pi_1(U)}(M)\right)(E). \qedhere
    \end{align*}
\end{proof}

By \Cref{lem:kappa_coind}, the presheaf axioms for $\mathcal{M}$ are merely a reformulation of the condition in \Cref{lem:equiv_hsh}(\ref{item:hsh_cocycle}). Therefore, via the section-wise equivalences $\Phi_U$ and $\Psi_U$, we identify $\PSh(X, \Set_X)$ with the category described in \Cref{lem:equiv_hsh}. We thus obtain the following.

\begin{proposition}\label{prop:half_sheaf_equivalence}
    The construction above establishes an equivalence of categories
    \[ \Phi : \HSh(X_\bt, \Set) \rightleftarrows \PSh(X, \Set_X) : \Psi. \]
\end{proposition}

We call $\mathcal{M} \in \PSh(X, \Set_X)$ a \emph{sheaf} if it satisfies the classical sheaf condition.

\begin{definition}\label{def:sheaf}
    A presheaf $\mathcal{M} \in \PSh(X, \Set_X)$ is called a \emph{sheaf} if for every $U \in \Open(X)$ and every open cover $\mathcal{U}$ of $U$, the sequence
    \[ \mathcal{M}(U) \to \prod_{V \in \mathcal{U}} \mathcal{M}(V) \rightrightarrows \prod_{V, W \in \mathcal{U}} \mathcal{M}(V \cap W) \]
    is an equalizer diagram in $\Set_X$. We denote by $\Sh(X, \Set_X)$ the full subcategory of sheaves.
\end{definition}

\begin{remark}\label{rem:sheaf_categorical_products}
    As a warning, the products and the equalizer in the sheaf condition are category-theoretic limits. Namely,
    \begin{align*}
        &\Eq\left(\prod_{V \in \mathcal{U}} \mathcal{M}(V) \rightrightarrows \prod_{V, W \in \mathcal{U}} \mathcal{M}(V \cap W)\right) \\
        \cong{} &\Eq\left(\prod_{V \in \mathcal{U}} \Coind_V^U \mathcal{M}(V) \rightrightarrows \prod_{V, W \in \mathcal{U}} \Coind_{V \cap W}^U \mathcal{M}(V \cap W)\right).
    \end{align*}
\end{remark}

We now show that the equivalence in \Cref{prop:half_sheaf_equivalence} restricts to an equivalence between $\Sh(X_\bt, \Set)$ and $\Sh(X, \Set_X)$. By \Cref{prop:sheaf_criterion}, it suffices to prove that $\mathcal{M} \in \PSh(X, \Set_X)$ is a sheaf if and only if $\Psi(\mathcal{M}) \in \HSh(X_\bt, \Set)$ is sheafified along the Zariski topology.

\begin{lemma}\label{lem:local_equivalence}
    Let $\mathcal{M} \in \PSh(X, \Set_X)$ and set $\mathcal{F} \coloneqq \Psi(\mathcal{M}) \in \HSh(X_\bt, \Set)$. For an open cover $\mathcal{U}$ of $U \in \Open(X)$, the equalizer conditions for the following diagrams are equivalent.
    \begin{enumerate}
        \item\label{axiom:sheaf_set_X} $\displaystyle \mathcal{M}(U) \to \prod_{V \in \mathcal{U}} \mathcal{M}(V) \rightrightarrows \prod_{V, W \in \mathcal{U}} \mathcal{M}(V \cap W)$
        \item\label{axiom:sheaf_betti} $\displaystyle \mathcal{F}^E(U) \to \prod_{V \in \mathcal{U}} \mathcal{F}^E(V) \rightrightarrows \prod_{V, W \in \mathcal{U}} \mathcal{F}^E(V \cap W)$ \quad (for all $E \in \Cov(U)$)
    \end{enumerate}
\end{lemma}
\begin{proof}
    By \Cref{rem:sheaf_categorical_products}, the equalizer condition for~(\ref{axiom:sheaf_set_X}) is equivalent to that of
    \[ \mathcal{M}(U) \to \prod_{V \in \mathcal{U}} \Coind_V^U \mathcal{M}(V) \rightrightarrows \prod_{V, W \in \mathcal{U}} \Coind_{V \cap W}^U \mathcal{M}(V \cap W) \]
    in $\Set(\Pi_1(U))$. By \Cref{prop:half_sheaf_equivalence}, this is equivalent to the equalizer condition of
    \[ \mathcal{F}_U \to \prod_{V \in \mathcal{U}} \kappa_V^U \mathcal{F}_V \rightrightarrows \prod_{V, W \in \mathcal{U}} \kappa_{V \cap W}^U \mathcal{F}_{V \cap W} \]
    in $\Sh(\Cov(U), \Set)$. Since limits in $\Sh(\Cov(U), \Set)$ are computed sectionwise and
    \[ (\kappa_V^U \mathcal{F}_V)(E) = \mathcal{F}(E \times_X V) = \mathcal{F}^E(V), \]
    evaluation at every $E \in \Cov(U)$ identifies this with the equalizer condition for~(\ref{axiom:sheaf_betti}).
\end{proof}

\begin{theorem}\label{thm:sheaf_equivalence}
    The equivalence in \Cref{prop:half_sheaf_equivalence} restricts to
    \[ \Phi : \Sh(X_\bt, \Set) \rightleftarrows \Sh(X, \Set_X) : \Psi. \]
\end{theorem}

By \Cref{prop:half_sheaf_equivalence} and \Cref{thm:sheaf_equivalence}, the classical sheafification on the site $X_\bt$ yields the following \emph{sheafification functor} $(-)^{\#}$ \cite[\href{https://stacks.math.columbia.edu/tag/00WB}{Tag~00WB}, \href{https://stacks.math.columbia.edu/tag/00WJ}{Tag~00WJ}]{StacksProject}.

\begin{corollary}\label{cor:sheafification}
    There is an adjunction
    \[ (-)^{\#} : \PSh(X, \Set_X) \rightleftarrows \Sh(X, \Set_X) : (-) \]
    in which $(-)$ is the forgetful functor and $(-)^{\#}$ is exact.
\end{corollary}

The analogous statements for $\Mod_X$ and $\Rep_X$ follow from \cite[\href{https://stacks.math.columbia.edu/tag/00YR}{Tag~00YR}]{StacksProject}. For $\mathcal{M}$ in $\Sh(X, \Mod_X)$ or $\Sh(X, \Rep_X)$, there is a natural isomorphism
\begin{equation}\label{eq:betti_global_sections}
    \Gamma(X, \mathcal{M})^{\Pi_1(X)} \cong \Gamma(X, \Psi(\mathcal{M})).
\end{equation}
We therefore define their cohomology as follows.

\begin{definition}\label{def:betti_cohomology}
    We define the \emph{$i$-th Betti cohomology functor} by
    \[ H^i_\bt(X, -) \coloneqq R^i\!\left(\Gamma(X, -)^{\Pi_1(X)}\right) \colon \Sh(X, \Mod_X) \to \Ab, \]
    and the \emph{$i$-th Betti cohomology functor with $\mathbb{C}$-coefficients} by
    \[ H^i_{\bt,\mathbb{C}}(X, -) \coloneqq R^i\!\left(\Gamma(X, -)^{\Pi_1(X)}\right) \colon \Sh(X, \Rep_X) \to \Vect(\mathbb{C}). \]
\end{definition}

Taking right derived functors in \eqref{eq:betti_global_sections} gives the following corollary.

\begin{corollary}\label{cor:betti_site_identification}
    For all $i \ge 0$, there are natural isomorphisms of functors
    \[ H^i_\bt(X, -) \cong H^i(X_\bt, \Psi(-)) \quad\text{and}\quad H^i_{\bt,\mathbb{C}}(X, -) \cong H^i(X_\bt, \Psi(-)). \]
\end{corollary}

\subsection{Comparison Theorem}\label{sec:analytic_comparison}

In the remainder of this section, we assume in addition that $X$ is smooth. We prove the comparison theorem between $H^i_\bt(X, -)$ and analytic sheaf cohomology $H^i(X^\an, -)$. By \Cref{thm:sheaf_equivalence}, this reduces to comparing sheaf cohomology on $X_\bt$ with sheaf cohomology on $X^\an$. The proof is adapted from Artin's comparison theorem between \'etale and singular cohomology \cite[Exp.~XI, \S4]{SGA4_3}.

The \emph{analytic Betti site} $X_\bt^\an$ is the site obtained from \Cref{def:category_Bt} by replacing every Zariski open subset with an analytic open subset. In particular, $X_\bt^\an$ is finer than $X^\an$, and there is a natural continuous morphism
\[ \epsilon\colon X_\bt^\an \to X^\an. \]
For $E \in X_\bt^\an$, we write $E_\bt^\an$ for the site obtained by restricting $X_\bt^\an$ to $E$.

\begin{lemma}\label{prop:epsilon_equivalence}
    For every $E \in X_\bt^\an$, the direct image functor
    \[ \epsilon_*\colon \Sh(E_\bt^\an, \Set) \xrightarrow{\sim} \Sh(E^\an, \Set) \]
    is an equivalence of categories. In particular, the natural map
    \[ H^i(E_\bt^\an, -) \to H^i(E^\an, \epsilon_*(-)) \]
    is an isomorphism for all $i \ge 0$.
\end{lemma}
\begin{proof}
    Every object of $E_\bt^\an$ is locally isomorphic to an analytic open subset of $E^\an$, viewed as an object of $E_\bt^\an$ by the identity covering. Hence, $\epsilon_*$ is an equivalence by the Comparison Lemma \cite[Exp.~III, Theorem~4.1]{SGA4_1}.
\end{proof}

Since $X_\bt^\an$ is also finer than $X_\bt$, the natural continuous morphism
\[ \delta\colon X_\bt^\an \to X_\bt \]
induces the direct image functor
\[ \delta_*\colon \Sh(X_\bt^\an, \Set) \to \Sh(X_\bt, \Set). \]
This functor is generally not an equivalence. However, we will show that for every locally constant $\mathcal{F} \in \Sh(X_\bt^\an, \Ab)$ and every $E \in X_\bt$, the canonical map
\[ H^i(E_\bt, \delta_*\mathcal{F}) \to H^i(E_\bt^\an, \mathcal{F}) \]
is an isomorphism.

\begin{lemma}\label{lem:higher_direct_image}
    For $\mathcal{F} \in \Sh(X_\bt^\an, \Ab)$, the sheaf $R^i \delta_* \mathcal{F} \in \Sh(X_\bt, \Ab)$ is associated to the presheaf
    \[ E \mapsto H^i(E_\bt^\an, \mathcal{F}). \]
\end{lemma}
\begin{proof}
    Let $0 \to \mathcal{F} \to \mathcal{I}^\bullet$ be an injective resolution in $\Sh(X_\bt^\an, \Ab)$. Restriction to $E_\bt^\an$ preserves injective sheaves \cite[\href{https://stacks.math.columbia.edu/tag/03F3}{Tag~03F3}]{StacksProject}, so $R^i \delta_* \mathcal{F}$ is the sheaf associated to the presheaf
    \[ E \mapsto H^i(\delta_* \mathcal{I}^\bullet(E)) = H^i(\Gamma(E_\bt^\an, \mathcal{I}^\bullet)) \cong H^i(E_\bt^\an, \mathcal{F}). \qedhere \]
\end{proof}

We now require the fact that $X$ admits a basis consisting of \emph{Artin neighborhoods}, also called good neighborhoods in \cite[Exp.~XI, \S3]{SGA4_3}. For their definition and construction, we refer to the same reference and recall only the properties necessary for our discussion.

\begin{theorem}[Artin, {\cite[Exp.~XI, Proposition~3.3]{SGA4_3}}]\label{prop:artin_good_neighborhood}
    The variety $X$ admits a basis of open sets consisting of Artin neighborhoods.
\end{theorem}

\begin{theorem}[Serre, {\cite[Exp.~XI, Variante~4.6]{SGA4_3}}]\label{thm:top_kpi1}
    Let $X$ be an Artin neighborhood over $\mathbb{C}$. Then $X^\an$ is a $K(\pi, 1)$ space.
\end{theorem}

\begin{lemma}\label{lem:bt_delta_vanishing}
    If $\mathcal{F} \in \Sh(X_\bt^\an, \Ab)$ is locally constant, then for all $i > 0$,
    \[ R^i \delta_* \mathcal{F} = 0. \]
\end{lemma}
\begin{proof}
    By \Cref{lem:higher_direct_image}, it suffices to prove that for every $E \in X_\bt$, there exists a covering family $\{E_\lambda \to E\}$ such that $H^i((E_\lambda)_\bt^\an, \mathcal{F}) = 0$ for each $\lambda$. Since $X$ admits a basis of Artin neighborhoods by \Cref{prop:artin_good_neighborhood}, we may cover $p(E)$ by Artin neighborhoods $U_\lambda \subseteq X$. Taking the fiber products $E \times_X U_\lambda$ and passing to their connected components, it suffices to assume that $p(E)$ is an Artin neighborhood and that $E$ is connected. Replacing $E$ by its universal cover, it further suffices to assume that $E$ is contractible by \Cref{thm:top_kpi1}. Hence, $\epsilon_*\mathcal{F}$ is a constant sheaf on $E^\an$. By \Cref{prop:epsilon_equivalence}, for all $i > 0$, we have
    \[ H^i(E_\bt^\an, \mathcal{F}) \cong H^i(E^\an, \epsilon_*\mathcal{F}) = 0. \qedhere \]
\end{proof}

\begin{theorem}[Comparison Theorem]\label{thm:betti_comparison}
    Let $\mathcal{F} \in \Sh(X_\bt^\an, \Ab)$ be a locally constant sheaf. Let $E \in X_\bt$. Then the natural map
    \[ H^i(E_\bt, \delta_*\mathcal{F}) \to H^i(E_\bt^\an, \mathcal{F}) \]
    is an isomorphism for all $i \ge 0$. In particular, the natural map
    \[ H^i_\bt(X, \Phi(\delta_*\mathcal{F})) \to H^i(X^\an, \epsilon_*\mathcal{F}) \]
    is an isomorphism for all $i \ge 0$.
\end{theorem}
\begin{proof}
    The Leray spectral sequence for $\delta|_E\colon E_\bt^\an \to E_\bt$ is
    \[ E_2^{p,q} = H^p(E_\bt, R^q \delta_* \mathcal{F}) \Rightarrow H^{p+q}(E_\bt^\an, \mathcal{F}). \]
    By \Cref{lem:bt_delta_vanishing}, we have $R^q \delta_* \mathcal{F} = 0$ for all $q > 0$. Therefore, the edge map gives the isomorphism $H^i(E_\bt, \delta_*\mathcal{F}) \xrightarrow{\sim} H^i(E_\bt^\an, \mathcal{F})$. Applying \Cref{cor:betti_site_identification} and \Cref{prop:epsilon_equivalence} together with $E = X^\an$ gives the stated particular case.
\end{proof}

Before concluding this section, we give a lemma that will be needed later. We say $\mathcal{M} \in \Sh(X, \Mod_X)$ is \emph{locally constant} if $\Psi(\mathcal{M}) \in \Sh(X_\bt, \Ab)$ is locally constant.

\begin{lemma}\label{lem:betti_kpi1}
    Let $U \in \Open(X)$, and let $\mathcal{M} \in \Sh(X, \Mod_X)$ be a locally constant sheaf. If $U^\an$ is a $K(\pi, 1)$ space, then for all $i > 0$,
    \[ R^i \Gamma(U, \mathcal{M}) = 0. \]
\end{lemma}
\begin{proof}
    Without loss of generality, assume that $U = X$ is connected. Since $\mathcal{M}$ is locally constant, $\Psi(\mathcal{M})$ becomes constant on the universal cover of $X$. Hence, $\Psi(\mathcal{M}) = \delta_*\mathcal{F}$ for some locally constant sheaf $\mathcal{F} \in \Sh(X_\bt^\an, \Ab)$. Fix $x \in X$, and let $P_x \to X^\an$ be the universal cover of $X^\an$ based at $x$, regarded as an object of $X_\bt$. Note that $(-)_x\colon \Mod(\Pi_1(U)) \to \Ab$ is an exact and conservative functor. Since
    \[ \Gamma(U, \mathcal{M})_x = \Gamma((P_x)_\bt, \delta_*\mathcal{F}), \]
    it suffices to show that $H^i((P_x)_\bt, \delta_*\mathcal{F}) = 0$ for $i > 0$. Applying \Cref{thm:betti_comparison} with $E = P_x$, this follows from \Cref{prop:epsilon_equivalence}.
\end{proof}

\begin{remark}\label{rem:rep_comparison}
    The adjunction
    \[ \mathbb{C} \otimes_\mathbb{Z} (-) : \Ab \rightleftarrows \Vect(\mathbb{C}) : (-) \]
    induces adjunctions
    \[ \Mod(G) \rightleftarrows \Rep(G) \quad\text{and}\quad \Sh(X^\an, \Ab) \rightleftarrows \Sh(X^\an, \Vect(\mathbb{C})) \]
    for any groupoid $G$ and complex variety $X$, where all adjoints are exact. Therefore, by the Grothendieck spectral sequence, groupoid cohomology and analytic sheaf cohomology with coefficients in $\Vect(\mathbb{C})$ are compatible with their counterparts with coefficients in $\Ab$. In particular, all results in this section hold, mutatis mutandis, for $\Rep_X$ in place of $\Mod_X$.
\end{remark}

\section{Formalization}\label{sec:formalization}

This section generalizes the sheaf theory developed in \Cref{sec:analytic}. We gave two equivalent descriptions of sheaves. One uses the Grothendieck topology of the Betti site, while the other uses the topological fundamental groupoid functor $\Pi_1$ together with the Zariski topology. The formalism developed here generalizes the latter description by allowing $\Pi_1$ to be replaced by other fundamental groupoid functors, such as $\Pi_1^\et$, $\Pi_1^\alg$, and $\Pi_1^\N$.

As explained in the previous section, the role of the Grothendieck topology in this setting is to glue local group-cohomological data. The relationships between topological fundamental groups and analytic sheaf cohomology, and between \'etale fundamental groups and \'etale cohomology, are well established. Analogous relationships also arise between fundamental group schemes constructed through Tannaka duality and algebraic de Rham cohomology \cite{EsnaultHai, BaoHoDao, HBB, BaoNguyen}. These analogies suggest that the same gluing mechanism should recover algebraic de Rham cohomology.

The obstacle, however, already appears at the level of group cohomology. For an abstract group $G$, the category $\Set(G)$ carries the canonical Grothendieck topology whose sheaf cohomology recovers $H^\bullet_\grp(G, -)$. For a profinite group $G$, the category $\FSet(G)$ similarly carries one recovering $H^\bullet_\cnt(G, -)$. For an affine group scheme $G$, however, it seems unlikely that $\Rep(G)$ can be equipped with a Grothendieck topology in the same way, with the resulting sheaf cohomology recovering $H^\bullet_\rat(G, -)$. This obstacle motivates the new approach to de Rham cohomology developed in this section.

To this end, we formalize the sheaf theory of the previous section in a more general setting. We first define presheaves and sheaves and then identify conditions under which sheafification exists. Finally, we show that the resulting category of sheaves is a Grothendieck category, thereby allowing us to define sheaf cohomology.

Throughout this section, fix a fibration
\[ p \colon \mathcal{T}^\op \to \mathcal{B}. \]
Unless otherwise stated, the base category $\mathcal{B}$ is the category $\Open(X)$ of open subsets of a topological space $X$ or, more generally, a frame\footnote{A frame is a small poset with finite products $U \wedge V$ and arbitrary coproducts $\bigvee U_i$ satisfying the infinite distributive law $U \wedge \bigvee U_i = \bigvee (U \wedge U_i)$. It is a point-free categorical abstraction of topology.} \cite[Ch.~II]{PicadoPultr}. Even if $\mathcal{B}$ is an abstract frame, we use the notation and terminology of open subsets. For example, we write $\mathcal{B} = \Open(X)$, where $X$ is the terminal object of $\mathcal{B}$. We also interpret unions and intersections as coproducts and products, respectively. In \Cref{sec:SCn_categories}, however, we allow $\mathcal{B}$ to be an arbitrary category.

\subsection{Presheaves and Sheaves}\label{sec:sheaf_theory}

This subsection defines presheaves and sheaves with values in $\mathcal{T}$, studies their properties, and gives several examples.

\begin{definition}\label{def:presheaf_abstract}
    A \emph{presheaf} on $X$ with values in $\mathcal{T}$ is a functor
    \[ \mathcal{M} \colon \mathcal{B}^\op \to \mathcal{T} \]
    satisfying $p \circ \mathcal{M}^\op = \id_{\mathcal{B}}$. We denote by $\PSh(X, \mathcal{T})$ the full subcategory of $\Fun(\mathcal{B}^\op, \mathcal{T})$ consisting of such presheaves.
\end{definition}

For $V \subseteq U$ in $\mathcal{B}$, we denote the corresponding restriction functor by
\[ \Res_V^U \colon \mathcal{T}(U) \to \mathcal{T}(V) \]
We assume the following for the remainder of this section.

\begin{SC}
    \item\label{axiom:SC1_pre} Every restriction functor admits a right adjoint.
\end{SC}

Therefore, the smallness of $\mathcal{B}$ allows us to choose adjunctions
\[ \Res_V^U : \mathcal{T}(U) \rightleftarrows \mathcal{T}(V) : \Coind_V^U \]
simultaneously for all $V \subseteq U$ in $\mathcal{B}$. We call $\Coind_V^U$ the \emph{coinduction functor}. The chosen adjunctions equip $p^\op$ with a cleavage \cite[Lemma~9.1.2]{Jacobs}, so $p$ is a bifibration. When the choice of $\mathcal{T}$ is not clear from context, we write
\[ \Res_{\mathcal{T}(V)}^{\mathcal{T}(U)} = \Res_V^U \quad\text{and}\quad \Coind_{\mathcal{T}(V)}^{\mathcal{T}(U)} = \Coind_V^U. \]

For a morphism $f \colon M \to N$ in $\mathcal{T}$, set $U \coloneqq p(M)$ and $V \coloneqq p(N)$. The Cartesian and cocartesian factorization properties imply that $f$ factors uniquely as $M \to \Res_V^U M \to N$ and as $M \to \Coind_V^U N \to N$, respectively.

\begin{lemma}\label{lem:limit_in_T}
    Given $U \in \mathcal{B}$, let $M \colon J \to \mathcal{T}(U)$ be a diagram indexed by a non-empty small category $J$. If the limit (resp.\ colimit) of $M$ exists in $\mathcal{T}(U)$, then it is also the limit (resp.\ colimit) of $M$ in $\mathcal{T}$.
\end{lemma}
\begin{proof}
    Let $L$ be the limit of $M$ taken in $\mathcal{T}(U)$. Since $J$ is non-empty, for a morphism $N \to M_j$ to exist in $\mathcal{T}$ for all $j \in J$, we must have $U \subseteq p(N)$. It is then straightforward to verify that $L$ satisfies the universal property of the limit of $M$ in $\mathcal{T}$, using the Cartesian property of the restriction functors. The colimit statement follows in the same manner, using the cocartesian factorization property associated with coinduction.
\end{proof}

Note that \Cref{lem:limit_in_T} requires $J$ to be non-empty. Indeed, initial and terminal objects of $\mathcal{T}$ need not lie in $\mathcal{T}(U)$, so the analogous statement for empty $J$ may fail. We now show that the sectionwise limit (resp.\ colimit) of presheaves is the limit (resp.\ colimit) in $\PSh(X, \mathcal{T})$.

\begin{lemma}\label{lem:limits_in_psh}
    Let $\mathcal{M} \colon J \to \PSh(X, \mathcal{T})$ be a diagram indexed by a small category $J$. Suppose that for every $U \in \mathcal{B}$,
    \[ \mathcal{L}(U) \coloneqq \lim_{j \in J} \mathcal{M}_j(U) \quad \left(\text{resp.}\ \colim_{j \in J} \mathcal{M}_j(U)\right) \]
    taken in $\mathcal{T}(U)$ exists. Then $\mathcal{L}$ is the limit (resp.\ colimit) of $(\mathcal{M}_j)_{j \in J}$ in $\PSh(X, \mathcal{T})$.
\end{lemma}
\begin{proof}
    We treat only the limit case, since the colimit case is analogous. When $J$ is empty, the claim is straightforward, so assume that $J$ is non-empty. By \Cref{lem:limit_in_T} and the functoriality of limits, we may assume that $\mathcal{L}$ is the limit of $\mathcal{M}$ in $\Fun(\mathcal{B}^\op, \mathcal{T})$. Then $\mathcal{L}$ satisfies the universal property of the limit of $\mathcal{M}$ in the full subcategory $\PSh(X, \mathcal{T})$.
\end{proof}

\begin{lemma}\label{lem:limit_coind_cover}
    Let $M \colon J \to \mathcal{T}$ be a diagram indexed by a small category $J$. Let $U_j = p(M_j)$ for each $j \in J$, and let $U = \bigcup_{j \in J} U_j$. If the limit of $\Coind_{U_j}^U M_j$ exists in $\mathcal{T}(U)$, then it is also the limit of $M$ in $\mathcal{T}$, and there is a natural isomorphism
    \[ \lim_{j \in J} \Coind_{U_j}^U M_j \xrightarrow{\sim} \lim_{j \in J} M_j. \]
\end{lemma}
\begin{proof}
    By the cocartesian property of the coinduction functors, the limit on the left-hand side satisfies the universal property of the limit on the right-hand side in $\mathcal{T}$.
\end{proof}

\begin{definition}\label{def:sheaf_abstract}
    A presheaf $\mathcal{M} \in \PSh(X, \mathcal{T})$ is called a \emph{sheaf} if for every $U \in \mathcal{B}$ and every open cover $\mathcal{U}$ of $U$, the sequence
    \[ \mathcal{M}(U) \to \prod_{V \in \mathcal{U}} \mathcal{M}(V) \rightrightarrows \prod_{V, W \in \mathcal{U}} \mathcal{M}(V \cap W) \]
    is an equalizer diagram in $\mathcal{T}$. We denote by $\Sh(X, \mathcal{T})$ the full subcategory of such sheaves.
\end{definition}

By \Cref{lem:limit_coind_cover}, the sheaf condition is equivalent to requiring that
\[ \mathcal{M}(U) \to \prod_{V \in \mathcal{U}} \Coind_V^U \mathcal{M}(V) \rightrightarrows \prod_{V, W \in \mathcal{U}} \Coind_{V \cap W}^U \mathcal{M}(V \cap W) \]
is an equalizer diagram in $\mathcal{T}(U)$. We define the \emph{sheafification functor} to be the left adjoint of the forgetful functor $\Sh(X, \mathcal{T}) \to \PSh(X, \mathcal{T})$.

\begin{example}\label{ex:set_X}
    Let $X$ be a complex variety equipped with the Zariski topology. The fibered categories $\Set_X$, $\Mod_X$, and $\Rep_X$ introduced in \Cref{def:set_X} satisfy (\ref{axiom:SC1_pre}). Moreover, their sheafification functors exist by \Cref{cor:sheafification}.
\end{example}

\begin{example}\label{ex:dset_X}
    Let $X$ be a Noetherian scheme. For a profinite groupoid $G$, let $\Set(G)$ and $\Mod(G)$ denote the categories of discrete $G$-sets and discrete $G$-modules, respectively. The fibered categories $\Set^\et_X$ and $\Mod^\et_X$ over $\Open(X)$ are defined by
    \[ \Set^\et_X(U) \coloneqq \Set(\Pi_1^\et(U)) \quad\text{and}\quad \Mod^\et_X(U) \coloneqq \Mod(\Pi_1^\et(U)) \]
    respectively, with the usual restriction functors. They satisfy (\ref{axiom:SC1_pre}), as shown in \Cref{sec:profinite_groupoids}.
\end{example}

\begin{example}\label{ex:rep_X}
    Let $X$ be a smooth variety over a field $k$ of characteristic $0$. For an affine groupoid scheme $G$ over $k$, let $\Rep(G)$ denote the category of representations of $G$. The fibration $(\Rep^\alg_X)^\op \to \Open(X)$ is defined by $\Rep^\alg_X(U) \coloneqq \Rep(\Pi_1^\alg(U))$ with the usual restriction functors. It satisfies (\ref{axiom:SC1_pre}), as shown in \Cref{sec:affine_groupoids}.
\end{example}

\begin{example}\label{ex:qcoh_X}
    Let $X$ be a Noetherian scheme. For $U \in \Open(X)$, let $\QCoh(U)$ denote the category of quasi-coherent sheaves on $U$. The fibration $\QCoh_X^\op \to \Open(X)$ is defined by $\QCoh_X(U) \coloneqq \QCoh(U)$. For $V \subseteq U$, let $j_V^U \colon V \hookrightarrow U$ denote the inclusion. Its restriction and coinduction functors are given by the pullback and pushforward
    \[
    \begin{aligned}
        (j_V^U)^* \colon \QCoh(U) &\to \QCoh(V) \\
        \mathcal{F} &\mapsto \mathcal{F}|_V
    \end{aligned}
    \quad\text{and}\quad
    \begin{aligned}
        (j_V^U)_* \colon \QCoh(V) &\to \QCoh(U) \\
        \mathcal{G} &\mapsto (W \mapsto \mathcal{G}(W \cap V)),
    \end{aligned}
    \]
    respectively. Since $X$ is Noetherian, $j_V^U$ is quasi-compact and quasi-separated. Thus, $(j_V^U)_*$ preserves quasi-coherent sheaves by \cite[\href{https://stacks.math.columbia.edu/tag/01LC}{Tag~01LC}]{StacksProject}.
\end{example}

To develop sheaf theory in this setting, we need a general criterion for the existence of sheafification. A category of sheaves with values in a fibered category need not arise from a natural Grothendieck topology. In particular, $\Sh(X, \Rep^\alg_X)$ does not, so a different approach is required.

\subsection{Sheafification}\label{sec:sheafification}

We now construct the sheafification functor. For the remainder of this subsection, in addition to (\ref{axiom:SC1_pre}), we assume the following condition.

\begin{SC}[2]
    \item\label{axiom:SC2_pre} Each $\mathcal{T}(U)$ is locally finitely presentable.
\end{SC}

An object $A$ of a category $\mathcal{A}$ is \emph{finitely presentable} if $\Hom_{\mathcal{A}}(A, -)$ preserves filtered colimits \cite[Definition~1.1]{AdamekRosicky}. In particular, every morphism
\[ A \to \varinjlim_\alpha M_\alpha \]
factors through some $M_\alpha$. Moreover, two morphisms $f, g \colon A \to M_\alpha$ induce the same morphism to the colimit if their composites with some $M_\alpha \to M_\beta$ agree.

For a set $I$ of objects of $\mathcal{A}$, consider the functor
\begin{align*}
    \lambda \colon \mathcal{A} &\to \Set^I \\
    M &\mapsto (\Hom_{\mathcal{A}}(A, M))_{A \in I}.
\end{align*}
We call $I$ a \emph{strong generating set} if it is a generating set and $\lambda$ is conservative\footnote{A functor $\lambda$ is conservative if a morphism $f$ is an isomorphism whenever $\lambda(f)$ is an isomorphism.} \cite[1.9]{GabrielUlmer}; see \cite[Definition~0.6]{AdamekRosicky} for a reformulation. A category $\mathcal{A}$ is \emph{locally finitely presentable} if it admits all colimits and has a strong generating set consisting of finitely presentable objects \cite[Theorem~1.20 and Definition~1.9]{AdamekRosicky}. If $\mathcal{A}$ is locally finitely presentable, then the following hold.
\begin{enumerate}
    \item The category $\mathcal{A}$ has all limits \cite[Corollary~1.28]{AdamekRosicky}.
    \item\label{item:filtered_colimits_exact} Filtered colimits in $\mathcal{A}$ are exact\footnote{A functor is left exact, right exact, or exact if it preserves finite limits, finite colimits, or both, respectively.} \cite[Proposition~1.59]{AdamekRosicky}.
\end{enumerate}

The following construction is the analogue of the classical plus construction.

\begin{definition}\label{def:plus_construction_T}
    Let $\mathcal{M} \in \PSh(X, \mathcal{T})$. For an open cover $\mathcal{U}$ of $U$, set
    \[ \mathcal{M}(\mathcal{U}) \coloneqq \Eq \left( \prod_{V \in \mathcal{U}} \Coind_V^U \mathcal{M}(V) \rightrightarrows \prod_{V, W \in \mathcal{U}} \Coind_{V \cap W}^U \mathcal{M}(V \cap W) \right). \]
    The \emph{plus construction} $\mathcal{M}^+ \in \PSh(X, \mathcal{T})$ is defined by
    \[ \mathcal{M}^+(U) \coloneqq \varinjlim_{\mathcal{U}} \mathcal{M}(\mathcal{U}), \]
    where the colimit is taken over all open covers $\mathcal{U}$ of $U$ ordered by refinement. This defines a functor $(-)^+ \colon \PSh(X, \mathcal{T}) \to \PSh(X, \mathcal{T})$.
\end{definition}

We cannot use the standard proof that applying the plus construction twice yields a sheaf. The technical difficulty arises because coinduction functors need not commute with filtered colimits. Instead, we iterate the plus construction transfinitely.

\begin{definition}\label{def:transfinite_plus_construction}
    For every ordinal $\alpha$, define $\mathcal{M}^{(\alpha)}$ recursively by
    \[ \mathcal{M}^{(0)} \coloneqq \mathcal{M} \quad\text{and}\quad \mathcal{M}^{(\alpha+1)} \coloneqq \bigl(\mathcal{M}^{(\alpha)}\bigr)^+. \]
    For every limit ordinal $\beta$, define
    \[ \mathcal{M}^{(\beta)} \coloneqq \varinjlim_{\alpha < \beta} \mathcal{M}^{(\alpha)}. \]
    By transfinite induction, this defines a functor $(-)^{(\alpha)} \colon \PSh(X, \mathcal{T}) \to \PSh(X, \mathcal{T})$.
\end{definition}

\begin{lemma}\label{lem:uniform_regular_cardinal}
    There exists a regular cardinal $\kappa > |\mathcal{B}|$ such that, after identifying $\kappa$ with its initial ordinal, colimits indexed by $\kappa$ are preserved by (1) every $\Coind_V^U$ and (2) every $\kappa$-small limit functor in $\mathcal{T}(U)$.
\end{lemma}
\begin{proof}
    By \cite[Proposition~2.23]{AdamekRosicky}, every $\Coind_V^U$ is accessible. Since there are only set-many coinduction functors, we may choose a regular cardinal $\kappa > |\mathcal{B}|$ such that they are all $\kappa$-accessible, proving (1). Assertion (2) follows from \cite[Proposition~1.59]{AdamekRosicky}.
\end{proof}

Fix such a $\kappa$ and set
\[ (-)^\# \coloneqq (-)^{(\kappa)} \colon \PSh(X, \mathcal{T}) \to \PSh(X, \mathcal{T}). \]
We show that this functor is the sheafification functor.

\begin{corollary}\label{cor:kappa_matching_object}
    For every open cover $\mathcal{U}$ of $U$, we have the natural isomorphism
    \[ \varinjlim_{\alpha < \kappa} \mathcal{M}^{(\alpha)}(\mathcal{U}) \xrightarrow{\sim} \mathcal{M}^{(\kappa)}(\mathcal{U}). \]
\end{corollary}
\begin{proof}
    The limit defining $\mathcal{M}(\mathcal{U})$ in \Cref{def:plus_construction_T} is $\kappa$-small. The result now follows from \Cref{lem:uniform_regular_cardinal}.
\end{proof}

\begin{lemma}\label{lem:transfinite_plus_sheaf}
    For every $\mathcal{M} \in \PSh(X, \mathcal{T})$, the presheaf $\mathcal{M}^\#$ is a sheaf.
\end{lemma}
\begin{proof}
    Fix $U \in \mathcal{B}$, an open cover $\mathcal{U}$ of $U$, and a strong generating set $I$ for $\mathcal{T}(U)$ consisting of finitely presentable objects. Let
    \[ \sigma \colon \mathcal{M}^{(\kappa)}(U) \to \mathcal{M}^{(\kappa)}(\mathcal{U}) \]
    be the natural morphism. To prove that $\mathcal{M}^{(\kappa)}$ is a sheaf, it suffices to show that $\sigma$ is an isomorphism. Equivalently, for every $A \in I$, it suffices to prove the bijectivity of
    \[ \sigma \circ (-) \colon \Hom_{\mathcal{T}(U)}\!\left(A, \mathcal{M}^{(\kappa)}(U)\right) \to \Hom_{\mathcal{T}(U)}\!\left(A, \mathcal{M}^{(\kappa)}(\mathcal{U})\right). \]

    For surjectivity, let $f \colon A \to \mathcal{M}^{(\kappa)}(\mathcal{U})$. By \Cref{cor:kappa_matching_object} and the finite presentability of $A$, the morphism $f$ has a lift $f^{(\beta)} \colon A \to \mathcal{M}^{(\beta)}(\mathcal{U})$ for some $\beta < \kappa$. The composite
    \[ A \overset{f^{(\beta)}}{\longrightarrow} \mathcal{M}^{(\beta)}(\mathcal{U}) \longrightarrow \mathcal{M}^{(\beta+1)}(U) \longrightarrow \mathcal{M}^{(\kappa)}(U) \]
    maps to $f$ under $\sigma \circ (-)$.

    For injectivity, suppose that $f, g \colon A \to \mathcal{M}^{(\kappa)}(U)$ have the same image under $\sigma \circ (-)$. Since $A$ is finitely presentable, $f$ and $g$ have lifts $f^{(\alpha)}, g^{(\alpha)} \colon A \to \mathcal{M}^{(\alpha)}(U)$ for some $\alpha < \kappa$. By \Cref{cor:kappa_matching_object}, for some $\beta$ satisfying $\alpha \le \beta < \kappa$, the lifts induce the same morphism $h \colon A \to \mathcal{M}^{(\beta)}(\mathcal{U})$ after composing with $\mathcal{M}^{(\alpha)}(U) \to \mathcal{M}^{(\alpha)}(\mathcal{U}) \to \mathcal{M}^{(\beta)}(\mathcal{U})$. Then both $f$ and $g$ are equal to the composite
    \[ A \overset{h}{\longrightarrow} \mathcal{M}^{(\beta)}(\mathcal{U}) \longrightarrow \mathcal{M}^{(\beta+1)}(U) \longrightarrow \mathcal{M}^{(\kappa)}(U). \qedhere \]
\end{proof}

Thus, we obtain a functor $(-)^\# \colon \PSh(X, \mathcal{T}) \to \Sh(X, \mathcal{T})$.

\begin{theorem}\label{thm:sheafification}
    The functor $(-)^\#$ is exact and fits into an adjunction
    \[ (-)^\# : \PSh(X, \mathcal{T}) \rightleftarrows \Sh(X, \mathcal{T}) : (-), \]
    where the right adjoint is the forgetful functor. Thus, $(-)^\#$ is the sheafification functor.
\end{theorem}
\begin{proof}
    Let $\mathcal{M}$ be a presheaf and $\mathcal{N}$ a sheaf. By the sheaf condition of $\mathcal{N}$, every morphism $\mathcal{M} \to \mathcal{N}$ factors uniquely through the canonical morphism $\mathcal{M} \to \mathcal{M}^+$. Hence, by transfinite induction, every morphism $\mathcal{M} \to \mathcal{N}$ extends uniquely to $\mathcal{M}^{(\alpha)} \to \mathcal{N}$ for every $\alpha \le \kappa$. In particular, it extends uniquely to $\mathcal{M}^\# \to \mathcal{N}$, proving the adjunction.

    Products and equalizers preserve limits, and coinduction functors preserve limits because they are right adjoints. By (\ref{item:filtered_colimits_exact}), filtered colimits in $\mathcal{T}(U)$ are exact. Hence, $(-)^+$ is left exact, and transfinite induction shows that $(-)^\#$ is left exact. As a left adjoint, $(-)^\#$ is right exact. Therefore, it is exact.
\end{proof}

\begin{remark}\label{rem:double_plus_sheafification}
    Suppose further that every coinduction functor commutes with filtered colimits. For $U \in \mathcal{B}$, consider the presheaf $\mathcal{M}^U \in \PSh(U, \mathcal{T}(U))$ defined by
    \[ \mathcal{M}^U(V) \coloneqq \Coind_V^U \mathcal{M}(V). \]
    Since coinduction preserves limits and filtered colimits, the plus construction commutes with coinduction. Hence,
    \[ (\mathcal{M}^{++})^U \cong (\mathcal{M}^U)^{++}. \]
    By \cite[Theorem~8.6.3]{Schapira}, the right-hand side is the sheaf associated to $\mathcal{M}^U$. Since $U$ is arbitrary, $\mathcal{M}^{++}$ is the sheaf associated to $\mathcal{M}$, and therefore $(-)^\# \cong (-)^{++}$.

    This supposition tends to hold for algebraically defined examples. For instance, it holds for $\Set^\et_X$, $\Mod^\et_X$, $\Rep^\alg_X$, and $\QCoh_X$. By contrast, the coinduction functors for $\Set_X$ and $\Mod_X$ need not commute with filtered colimits.
\end{remark}

Before proceeding, we prove the following lemma for later use.

\begin{lemma}\label{lem:morphisms_on_basis}
    Let $\mathcal{M} \in \PSh(X, \mathcal{T})$, let $\mathcal{N} \in \Sh(X, \mathcal{T})$, and let $\mathfrak{U}$ be a basis for $X$. If two morphisms $\eta, \xi \colon \mathcal{M} \to \mathcal{N}$ satisfy $\eta_U = \xi_U$ for every $U \in \mathfrak{U}$, then $\eta = \xi$.
\end{lemma}
\begin{proof}
    Fix $U \in \mathcal{B}$, and choose an open cover $\mathcal{U}$ of $U$ by elements of $\mathfrak{U}$. Then
    \[ \mathcal{N}(U) \to \prod_{V \in \mathcal{U}} \Coind_V^U \mathcal{N}(V) \rightrightarrows \prod_{V, W \in \mathcal{U}} \Coind_{V \cap W}^U \mathcal{N}(V \cap W) \]
    is an equalizer diagram. For every $V \in \mathcal{U}$, the composites
    \[
    \begin{tikzcd}[column sep=1.5em]
        \mathcal{M}(U) \arrow[r, shift left, "\eta_U"] \arrow[r, shift right, "\xi_U"'] & \mathcal{N}(U) \arrow[r] & \Coind_V^U \mathcal{N}(V)
    \end{tikzcd}
    \]
    are adjoint to
    \[
    \begin{tikzcd}[column sep=1.5em]
        \Res_V^U \mathcal{M}(U) \arrow[r] & \mathcal{M}(V) \arrow[r, shift left, "\eta_V"] \arrow[r, shift right, "\xi_V"'] & \mathcal{N}(V),
    \end{tikzcd}
    \]
    which agree because $\eta_V = \xi_V$. Hence, the two composites $\mathcal{M}(U) \rightrightarrows \Coind_V^U \mathcal{N}(V)$ agree. Since the first arrow in the equalizer diagram is a monomorphism, $\eta_U = \xi_U$. Thus, $\eta = \xi$.
\end{proof}

\subsection{Enough Injectives}\label{sec:enough_injectives}

In this subsection, we show that, under suitable conditions on $\mathcal{T}$, the category $\Sh(X, \mathcal{T})$ is a Grothendieck category\footnote{A Grothendieck category is an AB5 category with a generator.} and hence has enough injectives \cite[Th\'eor\`eme~1.10.1]{Grothendieck}. Consequently, every left exact functor from $\Sh(X, \mathcal{T})$ to an abelian category $\mathcal{A}$ admits right derived functors, which we use to define sheaf cohomology. Together with (\ref{axiom:SC1_pre}) and (\ref{axiom:SC2_pre}), we assume the following conditions.

\begin{SC}[3]
    \item\label{axiom:SC3_pre} Each $\Res_V^U$ is left exact.
    \item\label{axiom:SC4_pre} Each $\mathcal{T}(U)$ is an abelian category.
\end{SC}

Since each $\Res_V^U$ is a left adjoint, it is right exact, and hence exact by (\ref{axiom:SC3_pre}). We first show that both $\PSh(X, \mathcal{T})$ and $\Sh(X, \mathcal{T})$ have a generator.

\begin{definition}\label{def:extension_by_zero}
    Given $U \in \mathcal{B}$ and $M \in \mathcal{T}(U)$, we define $M_U^- \in \PSh(X, \mathcal{T})$ by
    \[
    M_U^-(V) \coloneqq
    \begin{cases}
        \Res_V^U M & \text{if } V \subseteq U, \\
        0 & \text{otherwise.}
    \end{cases}
    \]
    This construction is functorial in $M$, so we have functors
    \[ (-)_U^- \colon \mathcal{T}(U) \to \PSh(X, \mathcal{T}) \quad\text{and}\quad (-)_U \coloneqq (-)^\# \circ (-)_U^- \colon \mathcal{T}(U) \to \Sh(X, \mathcal{T}). \]
\end{definition}

\begin{lemma}\label{lem:extension_by_zero_adjunction}
    Given $U \in \mathcal{B}$, there are adjunctions
    \begin{align*}
        (-)_U^- : \mathcal{T}(U) &\rightleftarrows \PSh(X, \mathcal{T}) : \Gamma(U, -) \quad \text{and} \\
        (-)_U : \mathcal{T}(U) &\rightleftarrows \Sh(X, \mathcal{T}) : \Gamma(U, -).
    \end{align*}
    Moreover, both $(-)_U^-$ and $(-)_U$ are exact.
\end{lemma}
\begin{proof}
    The first adjunction is an immediate consequence of the cocycle condition of restriction functors. The second adjunction follows from the first and the universal property of sheafification. The exactness of $(-)_U^-$ follows from (\ref{axiom:SC3_pre}), and then that of $(-)_U$ from \Cref{thm:sheafification}.
\end{proof}

\begin{lemma}\label{lem:generator}
    The categories $\PSh(X, \mathcal{T})$ and $\Sh(X, \mathcal{T})$ have a generator.
\end{lemma}
\begin{proof}
    For every $U \in \mathcal{B}$, choose a generating set $I_U$ for $\mathcal{T}(U)$ consisting of finitely presentable objects, which exists by (\ref{axiom:SC2_pre}). By \Cref{lem:extension_by_zero_adjunction},
    \[ \left\{ M_U^- \;\middle|\; U \in \mathcal{B}, \, M \in I_U \right\} \quad\text{and}\quad \left\{ M_U \;\middle|\; U \in \mathcal{B}, \, M \in I_U \right\} \]
    are generating sets for $\PSh(X, \mathcal{T})$ and $\Sh(X, \mathcal{T})$, respectively. Taking their coproducts gives a generator of the respective category.
\end{proof}

We next show that both $\PSh(X, \mathcal{T})$ and $\Sh(X, \mathcal{T})$ are AB5 categories.

\begin{lemma}\label{lem:PSh_preadditive}
    The category $\PSh(X, \mathcal{T})$ is preadditive.
\end{lemma}
\begin{proof}
    For $f, g \colon \mathcal{M} \to \mathcal{N}$, define $(f+g)_U \coloneqq f_U + g_U$. To show that $f+g$ is a morphism of presheaves, it suffices to verify that the following diagram commutes for every $V \subseteq U$ in $\mathcal{B}$.
    \[
    \begin{tikzcd}
        \mathcal{M}(U) \arrow[r] \arrow[d, "f_U + g_U"'] 
            & \Res_V^U \mathcal{M}(U) \arrow[r, "\alpha_V^U"] \arrow[d, "\Res_V^U(f_U + g_U)"] 
            & \mathcal{M}(V) \arrow[d, "f_V + g_V"] \\
        \mathcal{N}(U) \arrow[r] 
            & \Res_V^U \mathcal{N}(U) \arrow[r, "\beta_V^U"'] 
            & \mathcal{N}(V)
    \end{tikzcd}
    \]
    The left square commutes by the definition of $\Res_V^U$ on morphisms. Since $\Res_V^U$ is exact, it is additive \cite[\href{https://stacks.math.columbia.edu/tag/0DLP}{Tag~0DLP}]{StacksProject}. Thus,
    \begin{align*}
        (f_V + g_V) \circ \alpha_V^U &= f_V \circ \alpha_V^U + g_V \circ \alpha_V^U \\
        &= \beta_V^U \circ \Res_V^U(f_U) + \beta_V^U \circ \Res_V^U(g_U) \\
        &= \beta_V^U \circ \Res_V^U(f_U + g_U).
    \end{align*}
    Therefore, the right square also commutes. 
    
    The zero morphism and additive inverses are likewise defined sectionwise and are morphisms of presheaves by the same argument. Hence, each Hom-set is an abelian group under sectionwise addition. Since composition is bilinear sectionwise, it is bilinear for morphisms of presheaves. Therefore, $\PSh(X, \mathcal{T})$ is preadditive.
\end{proof}

\begin{lemma}\label{lem:psh_ab5}
    The category $\PSh(X, \mathcal{T})$ is an AB5 category.
\end{lemma}
\begin{proof}
    By \Cref{lem:PSh_preadditive}, the category $\PSh(X, \mathcal{T})$ is preadditive. Since its limits and colimits are computed sectionwise by \Cref{lem:limits_in_psh}, it is abelian by (\ref{axiom:SC4_pre}), and its filtered colimits are exact by (\ref{item:filtered_colimits_exact}).
\end{proof}

\begin{lemma}\label{lem:sh_ab}
    The category $\Sh(X, \mathcal{T})$ is an abelian category.
\end{lemma}
\begin{proof}
    The category $\Sh(X, \mathcal{T})$ is additive because it is closed under finite products in the additive category $\PSh(X, \mathcal{T})$. By \Cref{lem:psh_ab5}, the category $\PSh(X, \mathcal{T})$ is abelian. Since sheafification is exact by \Cref{thm:sheafification}, the result follows from \cite[\href{https://stacks.math.columbia.edu/tag/03A3}{Tag~03A3}]{StacksProject}.
\end{proof}

\begin{lemma}\label{lem:sh_ab5}
    The category $\Sh(X, \mathcal{T})$ satisfies the AB5 condition.
\end{lemma}
\begin{proof}
    Filtered colimits of sheaves are obtained by sheafifying the corresponding presheaf colimits. Hence, they are exact by \Cref{lem:psh_ab5,thm:sheafification}.
\end{proof}

Combining \Cref{lem:psh_ab5,lem:sh_ab,lem:sh_ab5,lem:generator} with the fact that every Grothendieck category has enough injectives \cite[Th\'eor\`eme~1.10.1]{Grothendieck}, we obtain the following.

\begin{theorem}\label{thm:sh_grothendieck}
    The categories $\PSh(X, \mathcal{T})$ and $\Sh(X, \mathcal{T})$ are Grothendieck categories and hence have enough injectives.
\end{theorem}

Hence, every left exact functor $\Sh(X, \mathcal{T}) \to \mathcal{A}$ to an abelian category $\mathcal{A}$ admits right derived functors. However, our previous examples carry more structure than a single such functor, as we explain in \Cref{sec:sheaf_cohomology}.

\subsection{SCn Categories}\label{sec:SCn_categories}

In this subsection, we summarize the axioms for a fibration $p \colon \mathcal{T}^\op \to \mathcal{B}$ over an arbitrary category $\mathcal{B}$ and present several examples. For every morphism $f \colon V \to U$ in $\mathcal{B}$, we denote the corresponding restriction functor by $f^* \colon \mathcal{T}(U) \to \mathcal{T}(V)$.

\begin{hypothesis}\label{hyp:SC}
    For every morphism $f \colon V \to U$ in $\mathcal{B}$, we consider the following conditions.
    \begin{SC}
        \item\label{axiom:SC1} The functor $f^*$ admits a right adjoint $f_*$.
        \item\label{axiom:SC2} The category $\mathcal{T}(U)$ is locally finitely presentable.
        \item\label{axiom:SC3} The functor $f^*$ is left exact.
        \item\label{axiom:SC4} The category $\mathcal{T}(U)$ is an abelian category.
        \item\label{axiom:SC5} The functor $f_*$ commutes with filtered colimits.
    \end{SC}
\end{hypothesis}

A morphism $f \colon V \to U$ in $\mathcal{B}$ satisfying (\ref{axiom:SC1})--(\hyperref[hyp:SC]{SCn}) is called an \emph{SCn morphism for $\mathcal{T}$}. If every morphism in $\mathcal{B}$ is an SCn morphism for $\mathcal{T}$, then $p$ and $\mathcal{T}$ are called an \emph{SCn fibration} and an \emph{SCn category}, respectively. Our sheaf cohomology theory requires SC4 categories.

The fibered categories $\Set_\Gpd$, $\Mod_\Gpd$, and $\Rep_\Gpd$ over $\Gpd$ are defined by
\[ \Set_\Gpd(G) \coloneqq \Set(G), \quad \Mod_\Gpd(G) \coloneqq \Mod(G), \quad\text{and}\quad \Rep_\Gpd(G) \coloneqq \Rep(G) \]
respectively, with the usual restriction functors.

\begin{lemma}\label{lem:group_SC}
    Every group homomorphism $\varphi \colon H \to G$ is an SC3 morphism for $\Set_\Gpd$ and an SC4 morphism for both $\Mod_\Gpd$ and $\Rep_\Gpd$.
\end{lemma}
\begin{proof}
    First consider $\Mod_\Gpd$. The restriction functor admits a right adjoint and is exact, giving (\ref{axiom:SC1}) and (\ref{axiom:SC3}). Moreover,
    \[ \Hom_{\Mod(G)}(\mathbb{Z}(G), -) \colon \Mod(G) \to \Set \]
    is naturally isomorphic to the forgetful functor. Hence, $\{\mathbb{Z}(G)\}$ is a strong generating set for $\Mod(G)$ consisting of finitely presentable objects, giving (\ref{axiom:SC2}). Finally, $\Mod(G)$ is an abelian category, giving (\ref{axiom:SC4}). The same argument, using $G$ and $k(G)$ in place of $\mathbb{Z}(G)$, gives the assertions for $\Set_\Gpd$ and $\Rep_\Gpd$, respectively.
\end{proof}

\begin{lemma}\label{lem:transitive_groupoid_SC}
    Every morphism $\varphi \colon H \to G$ of transitive groupoids is an SC3 morphism for $\Set_\Gpd$ and an SC4 morphism for both $\Mod_\Gpd$ and $\Rep_\Gpd$.
\end{lemma}
\begin{proof}
    Choose an object $x$ of $H$ and set $y \coloneqq \varphi(x)$. Restriction to the vertex groups gives equivalences
    \[ \Mod(G) \xrightarrow{\sim} \Mod(G_y) \quad\text{and}\quad \Mod(H) \xrightarrow{\sim} \Mod(H_x). \]
    Under these equivalences, restriction along $\varphi$ corresponds to restriction along the induced group homomorphism $\varphi_x \colon H_x \to G_y$. Hence, $\varphi$ is an SC4 morphism for $\Mod_\Gpd$ by \Cref{lem:group_SC}. The same argument applies to $\Set_\Gpd$ and $\Rep_\Gpd$.
\end{proof}

\begin{proposition}\label{prop:groupoid_SC}
    The fibered category $\Set_\Gpd$ is an SC3 category, while $\Mod_\Gpd$ and $\Rep_\Gpd$ are SC4 categories.
\end{proposition}
\begin{proof}
    Let $\varphi \colon H \to G$ be a morphism of groupoids, and let
    $ \sigma \colon \pi_0(H) \to \pi_0(G) $
    be the induced map. We have equivalences
    \[ \Mod(G) \xrightarrow{\sim} \prod_{C \in \pi_0(G)} \Mod(C) \quad\text{and}\quad \Mod(H) \xrightarrow{\sim} \prod_{D \in \pi_0(H)} \Mod(D). \]
    Under these equivalences, $\Res_H^G$ corresponds to
    \[ (M_C)_{C \in \pi_0(G)} \mapsto \left(\Res_D^{\sigma(D)} M_{\sigma(D)}\right)_{D \in \pi_0(H)}, \]
    whose right adjoint is
    \[ (N_D)_{D \in \pi_0(H)} \mapsto \left(\prod_{\sigma(D)=C} \Coind_D^C N_D\right)_{C \in \pi_0(G)}. \]
    Thus, (\ref{axiom:SC1}) and (\ref{axiom:SC3}) follow componentwise from \Cref{lem:transitive_groupoid_SC}. For every $C \in \pi_0(G)$, choose a finitely presentable generator $g_C$ of $\Mod(C)$. Regard each $g_C$ as an object of $\Mod(G)$ whose $C$-component is $g_C$ and whose other components are initial objects, hence zero. These objects form a strong generating set for $\Mod(G)$ consisting of finitely presentable objects, giving (\ref{axiom:SC2}). Moreover, $\Mod(G)$ is abelian componentwise, giving (\ref{axiom:SC4}). The same argument applies to $\Set_\Gpd$ and $\Rep_\Gpd$.
\end{proof}

\begin{theorem}\label{thm:ModX_SC}
    If $X$ is a complex variety, then the fibered category $\Set_X$ is an SC3 category, while the fibered categories $\Mod_X$ and $\Rep_X$ are SC4 categories.
\end{theorem}
\begin{proof}
    The fibration $\Mod_X^\op \to \Open(X)$ is the pullback of $\Mod_\Gpd^\op \to \Gpd$ along the fundamental groupoid functor $\Pi_1 \colon \Open(X) \to \Gpd$. Thus, the assertion for $\Mod_X$ follows from \Cref{prop:groupoid_SC}. The same argument applies to $\Set_X$ and $\Rep_X$.
\end{proof}

\begin{remark}\label{rem:SC5_counterexample}
    Condition (\ref{axiom:SC5}) need not hold for $\Set_X$, $\Mod_X$, and $\Rep_X$. Let $F_2$ be the free group on generators $a$ and $b$. Define a homomorphism
    $ F_2 \to \mathbb{Z}$ by $a \mapsto 1$ and $b \mapsto 0$.
    Its kernel $K$ is freely generated by $a^{-i} b a^i$ for $i \in \mathbb{Z}$. For $n \ge 0$, set
    \[ K_n \coloneqq \langle a^{-i} b a^i \mid -n \le i \le n \rangle. \]
    Then $\varinjlim_n F_2/K_n \cong F_2/K$. The $K$-fixed-point sets are
    \[ (F_2/K_n)^K = \varnothing \quad\text{and}\quad (F_2/K)^K = F_2/K. \]
    As $\Coind_{F_2}^\mathbb{Z} \cong (-)^K$, coinduction does not commute with this filtered colimit.
    Likewise, the $K$-invariants of $\mathbb{Z}(F_2/K_n)$ and $\mathbb{C}(F_2/K_n)$ vanish for every $n$ but not after passing to their colimits. Thus, $F_2 \to \mathbb{Z}$ does not satisfy (\ref{axiom:SC5}) for $\Set_\Gpd$, $\Mod_\Gpd$, or $\Rep_\Gpd$.

    Take $X = \mathbb{A}^1_\mathbb{C} \setminus \{1\}$ and $V = \mathbb{A}^1_\mathbb{C} \setminus \{1, -1\}$. The homomorphism $\pi_1(V, 0) \to \pi_1(X, 0)$ induced by $V \subseteq X$ is identified with the homomorphism $F_2 \to \mathbb{Z}$ defined above. Thus, (\ref{axiom:SC5}) fails for $\Set_X$, $\Mod_X$, and $\Rep_X$. By contrast, algebraic examples tend to satisfy (\ref{axiom:SC5}). Later in this paper, we prove that $\QCoh_X$, $\Mod^\et_X$, and $\Rep^\alg_X$ are SC5 categories.
\end{remark}

\begin{proposition}\label{prop:QCoh_SC5}
    If $X$ is a Noetherian scheme, then $\QCoh_X$ is an SC5 category.
\end{proposition}
\begin{proof}
    Condition (\ref{axiom:SC1}) follows from \cite[\href{https://stacks.math.columbia.edu/tag/0096}{Tag~0096}, \href{https://stacks.math.columbia.edu/tag/01LC}{Tag~01LC}]{StacksProject}, using that $V \hookrightarrow U$ is quasi-compact and quasi-separated since $X$ is Noetherian. Condition (\ref{axiom:SC2}) follows from \cite[\href{https://stacks.math.columbia.edu/tag/077P}{Tag~077P}, \href{https://stacks.math.columbia.edu/tag/01PK}{Tag~01PK}, \href{https://stacks.math.columbia.edu/tag/0GPE}{Tag~0GPE}]{StacksProject}. Condition (\ref{axiom:SC3}) follows from \cite[\href{https://stacks.math.columbia.edu/tag/0250}{Tag~0250}]{StacksProject}, and (\ref{axiom:SC4}) from \cite[\href{https://stacks.math.columbia.edu/tag/077P}{Tag~077P}]{StacksProject}. Finally, (\ref{axiom:SC5}) follows from \cite[\href{https://stacks.math.columbia.edu/tag/07TB}{Tag~07TB}]{StacksProject}.
\end{proof}

We conclude this subsection by defining morphisms of SC3 and SC4 categories.

\begin{definition}\label{def:SC_morphism}
    Let $q \colon \mathcal{S}^\op \to \mathcal{B}$ and $p \colon \mathcal{T}^\op \to \mathcal{B}$ be SC3 fibrations, and let
    \[ (\delta^{-1})^\op \colon \mathcal{S}^\op \to \mathcal{T}^\op \]
    be a Cartesian functor over $\mathcal{B}$. We call $\delta^{-1}$ a \emph{morphism of SC3 categories} if the following conditions hold for every $U \in \mathcal{B}$.
    \begin{enumerate}
        \item The functor $\delta^{-1}(U)$ admits a right adjoint.
        \item The functor $\delta^{-1}(U)$ is left exact.
    \end{enumerate}
    If $\mathcal{S}$ and $\mathcal{T}$ are SC4 categories, we call $\delta^{-1}$ a \emph{morphism of SC4 categories}.
\end{definition}

\subsection{Sheaf Cohomology}\label{sec:sheaf_cohomology}

We now return to the assumption that $\mathcal{B} = \Open(X)$. We define sheaf cohomology for SC4 categories over $\mathcal{B}$ and prove a comparison theorem between two such cohomology theories. Let $\mathcal{B}_{X/*} = \Open(X/*)$ be the category obtained from $\mathcal{B}$ by adjoining a new object $*$ together with a unique morphism $U \to *$ for every $U \in \mathcal{B}$. We assume that the SC4 fibration $\mathcal{T}^\op \to \mathcal{B}$ extends to an SC4 fibration $\mathcal{T}_{X/*}^\op \to \mathcal{B}_{X/*}$.

We refer to $*$ as the \emph{terminal space} and to each $U \to *$ as a \emph{structure morphism}. We call $\mathcal{T}_{X/*}(*)$ the \emph{coefficient category} and denote it by $\mathcal{A}$. We associate with $\mathcal{T}$ a \emph{groupoid symbol} $\mathfrak{G}$ and refer to objects of $\mathcal{T}(U)$ as $\mathfrak{G}(U)$-modules. The symbol $\mathfrak{G}$ is formal, and $\mathfrak{G}(U)$ need not be an actual groupoid. We define the \emph{$\mathfrak{G}(U)$-invariants functor} and the \emph{trivial $\mathfrak{G}(U)$-module functor}, respectively, by
\[ (-)^{\mathfrak{G}(U)} \coloneqq \Coind_U^* \colon \mathcal{T}(U) \to \mathcal{A} \quad\text{and}\quad \Delta_{\mathfrak{G}(U)} \coloneqq \Res_U^* \colon \mathcal{A} \to \mathcal{T}(U). \]

\begin{definition}\label{def:sheaf_cohomology}
    For $U \in \mathcal{B}$ and $i \ge 0$, we define the \emph{$i$-th groupoid cohomology functor of $\mathfrak{G}(U)$} by
    \[ H^i(\mathfrak{G}(U), -) \coloneqq R^i(-)^{\mathfrak{G}(U)} \colon \mathcal{T}(U) \to \mathcal{A}, \]
    and the \emph{$i$-th sheaf cohomology functor on $U$} by
    \[ H^i_{\mathcal{T}}(U, -) \coloneqq R^i\!\left((-)^{\mathfrak{G}(U)} \circ \Gamma_\mathcal{T}(U, -)\right) \colon \Sh(X, \mathcal{T}) \to \mathcal{A}. \]
\end{definition}

\begin{example}
    For $\mathcal{T} = \Mod_X$, the coefficient category is $\Ab$ and the groupoid symbol is $\mathfrak{G} = \Pi_1$. The terminal space is $* = \Spec \mathbb{C}$, and $\Pi_1(*)$ is the terminal groupoid $\mathds{1}$. The invariants functor is $(-)^{\mathfrak{G}(U)} = (-)^{\Pi_1(U)}$. The trivial module functor $\Delta_{\mathfrak{G}(U)} \colon \Ab \to \Mod(\Pi_1(U))$ sends an abelian group to the corresponding trivial $\Pi_1(U)$-module. The resulting groupoid and sheaf cohomology functors agree with those defined in \Cref{sec:sheaf_via_groupoid}.
\end{example}

\begin{remark}\label{rem:spectral_sequences}
    This sheaf cohomology theory admits two natural spectral sequences. First, consider the adjunction
    \[ (-)_X : \mathcal{T}(X) \rightleftarrows \Sh(X, \mathcal{T}) : \Gamma_\mathcal{T}(X, -). \]
    Since the left adjoint $(-)_X$ is exact by \Cref{lem:extension_by_zero_adjunction}, we obtain a Grothendieck spectral sequence
    \begin{equation}\label{eq:groupoid_section_spectral_sequence}
        E_2^{p,q} = H^p\!\left(\mathfrak{G}(X), R^q\Gamma_\mathcal{T}(X, \mathcal{M})\right) \Rightarrow H^{p+q}_{\mathcal{T}}(X, \mathcal{M}).
    \end{equation}
    In the case $\mathcal{T} = \Mod_X$, this recovers the Cartan--Leray spectral sequence on $X_\bt$.
    Since the coinduction functors commute with limits and satisfy the cocycle condition, the functor
    \begin{align*}
        (-)^{\mathfrak{G}} \colon \Sh(X, \mathcal{T}) &\to \Sh(X, \mathcal{A}) \\
        \mathcal{M} &\mapsto \left(U \mapsto \mathcal{M}(U)^{\mathfrak{G}(U)}\right)
    \end{align*}
    is well-defined. Similarly, we define
    \begin{align*}
        \Delta_{\mathfrak{G}} \colon \Sh(X, \mathcal{A}) &\to \Sh(X, \mathcal{T}) \\
        \mathcal{F} &\mapsto \left(U \mapsto \Delta_{\mathfrak{G}(U)} \mathcal{F}(U)\right)^\#.
    \end{align*}
    Then $\Delta_{\mathfrak{G}}$ is left adjoint to $(-)^{\mathfrak{G}}$, and it is exact by (\ref{axiom:SC3}) and \Cref{thm:sheafification}. Since
    \[ \Gamma_\mathcal{A}(X, -) \circ (-)^{\mathfrak{G}} \cong (-)^{\mathfrak{G}(X)} \circ \Gamma_\mathcal{T}(X, -), \]
    we obtain a Grothendieck spectral sequence
    \begin{equation}\label{eq:sheaf_invariants_spectral_sequence}
        E_2^{p,q} = H^p\!\left(X, R^q(-)^{\mathfrak{G}}(\mathcal{M})\right) \Rightarrow H^{p+q}_{\mathcal{T}}(X, \mathcal{M}).
    \end{equation}
    When $\mathcal{T} = \Mod_X$, this recovers the Leray spectral sequence for the morphism of sites $X_\bt \to X^\zar$ under \Cref{thm:sheaf_equivalence}.
\end{remark}

For the comparison theorem, let $\mathcal{S}_{X/*}^\op \to \mathcal{B}_{X/*}$ be another SC4 fibration, and fix a morphism $\delta^{-1} \colon \mathcal{S}_{X/*} \to \mathcal{T}_{X/*}$ of SC4 categories. For every $U \in \mathcal{B}_{X/*}$, denote $\delta^{-1}(U)$ and its right adjoint by $\Res_{\mathcal{T}(U)}^{\mathcal{S}(U)}$ and $\Coind_{\mathcal{T}(U)}^{\mathcal{S}(U)}$, respectively. Thus, we have an adjunction
\[
    \Res_{\mathcal{T}(U)}^{\mathcal{S}(U)} : \mathcal{S}_{X/*}(U) \rightleftarrows \mathcal{T}_{X/*}(U) : \Coind_{\mathcal{T}(U)}^{\mathcal{S}(U)}.
\]
Since $(\delta^{-1})^\op$ is Cartesian, every $V \subseteq U$ in $\mathcal{B}_{X/*}$ induces a natural isomorphism
\[
    \Res_{\mathcal{T}(V)}^{\mathcal{T}(U)} \circ \Res_{\mathcal{T}(U)}^{\mathcal{S}(U)} \xrightarrow{\sim} \Res_{\mathcal{T}(V)}^{\mathcal{S}(V)} \circ \Res_{\mathcal{S}(V)}^{\mathcal{S}(U)}.
\]
Taking right adjoints gives a natural isomorphism
\[
    \Coind_{\mathcal{S}(V)}^{\mathcal{S}(U)} \circ \Coind_{\mathcal{T}(V)}^{\mathcal{S}(V)} \xrightarrow{\sim} \Coind_{\mathcal{T}(U)}^{\mathcal{S}(U)} \circ \Coind_{\mathcal{T}(V)}^{\mathcal{T}(U)}.
\]
We also call these the cocycle conditions for the restriction and coinduction functors.

We write $\mathcal{A}_\mathcal{S} \coloneqq \mathcal{S}_{X/*}(*)$ and $\mathcal{A}_\mathcal{T} \coloneqq \mathcal{T}_{X/*}(*)$ for the respective coefficient categories and set $\delta_\mathcal{A} \coloneqq \Coind_{\mathcal{T}(*)}^{\mathcal{S}(*)} \colon \mathcal{A}_\mathcal{T} \to \mathcal{A}_\mathcal{S}$. We denote the groupoid symbols associated with $\mathcal{S}$ and $\mathcal{T}$ by $\mathfrak{H}$ and $\mathfrak{G}$, respectively. There is an adjunction
\[ \delta^* : \Sh(X, \mathcal{S}) \rightleftarrows \Sh(X, \mathcal{T}) : \delta_*, \]
given by
\[ \delta_* \mathcal{N} = \left(U \mapsto \Coind_{\mathcal{T}(U)}^{\mathcal{S}(U)} \mathcal{N}(U)\right) \quad\text{and}\quad \delta^* \mathcal{M} = \left(U \mapsto \Res_{\mathcal{T}(U)}^{\mathcal{S}(U)} \mathcal{M}(U)\right)^\#. \]
The presheaf $\delta_*\mathcal{N}$ is a sheaf because the coinduction functors commute with limits and satisfy the cocycle condition.

\begin{theorem}[Comparison Theorem]\label{thm:comparison_general}
    Let $\mathcal{M} \in \Sh(X, \mathcal{T})$. Suppose that $\delta_\mathcal{A}$ is exact, and that there exists a basis $\mathfrak{U}$ of $X$ such that for every $U \in \mathfrak{U}$ and every $i > 0$,
    \begin{enumerate}
        \item\label{cond:comparison_local_acyclicity} $R^i \Gamma_\mathcal{T}(U, \mathcal{M}) = 0$, and
        \item\label{cond:comparison_coind_acyclicity} $R^i \Coind_{\mathcal{T}(U)}^{\mathcal{S}(U)} \mathcal{M}(U) = 0$.
    \end{enumerate}
    Then the natural map
    \[ H^i_\mathcal{S}(X, \delta_*\mathcal{M}) \to \delta_\mathcal{A}\, H^i_\mathcal{T}(X, \mathcal{M}) \]
    is an isomorphism for all $i \ge 0$.
\end{theorem}
\begin{proof}
    We first show that $R^i \delta_* \mathcal{M} = 0$ for all $i > 0$. Let $0 \to \mathcal{M} \to \mathcal{I}^\bullet$ be an injective resolution in $\Sh(X, \mathcal{T})$. Then $R^i \delta_* \mathcal{M}$ is the sheaf associated to the presheaf
    \[
        U \mapsto
        H^i\left((\delta_* \mathcal{I}^\bullet)(U)\right) \cong H^i\left(\Coind_{\mathcal{T}(U)}^{\mathcal{S}(U)} \mathcal{I}^\bullet(U)\right) \cong R^i\left(\Coind_{\mathcal{T}(U)}^{\mathcal{S}(U)} \circ \Gamma_\mathcal{T}(U, -)\right) \mathcal{M}.
    \]
    The Grothendieck spectral sequence for $\Coind_{\mathcal{T}(U)}^{\mathcal{S}(U)} \circ \Gamma_\mathcal{T}(U, -)$ gives
    \[
        E_2^{p,q} = R^p \Coind_{\mathcal{T}(U)}^{\mathcal{S}(U)}\left(R^q \Gamma_\mathcal{T}(U, \mathcal{M})\right) \Rightarrow R^{p+q}\left(\Coind_{\mathcal{T}(U)}^{\mathcal{S}(U)} \circ \Gamma_\mathcal{T}(U, -)\right) \mathcal{M}.
    \]
    For $U \in \mathfrak{U}$, conditions~(\ref{cond:comparison_local_acyclicity}) and~(\ref{cond:comparison_coind_acyclicity}) imply that the $E_2$-terms vanish in positive total degrees. Thus, the universal property of sheafification and \Cref{lem:morphisms_on_basis} give $R^i \delta_* \mathcal{M} = 0$ for all $i > 0$.
    
    We have the following diagram, commuting up to the canonical isomorphism
    \[
    \begin{tikzcd}[column sep=7.5em]
        \Sh(X, \mathcal{T}) \arrow[r, "(-)^{\mathfrak{G}(X)} \circ \Gamma_\mathcal{T}(X{,} -)"] \arrow[d, "\delta_*"'] & \mathcal{A}_\mathcal{T} \arrow[d, "\delta_\mathcal{A}"] \\
        \Sh(X, \mathcal{S}) \arrow[r, "(-)^{\mathfrak{H}(X)} \circ \Gamma_\mathcal{S}(X{,} -)"'] & \mathcal{A}_\mathcal{S}
    \end{tikzcd}
    \]
    Applying the Grothendieck spectral sequences to the two composites and using the vanishing above yields the desired isomorphism.
\end{proof}

\begin{remark}
    \Cref{thm:comparison_general} applies when $\mathcal{T}$ carries a finer ``topology'' than $\mathcal{S}$. Condition~(\ref{cond:comparison_local_acyclicity}) is the analogue of $U$ being a $K(\pi, 1)$ space with respect to $\mathcal{T}$. Condition~(\ref{cond:comparison_coind_acyclicity}) is the analogue of the goodness condition, asserting that $\mathfrak{G}(U)$ is a ``good groupoid,'' meaning that its groupoid cohomology is unchanged upon completion to $\mathfrak{H}(U)$.
\end{remark}
\section{Profinite Theory}\label{sec:profinite}

We apply the framework of \Cref{sec:formalization} with the \'etale fundamental groupoid $\Pi_1^\et$ \cite[Exp.~V, \S7]{SGA1}. Fix a Noetherian scheme $X$ throughout this section. Recall that \[(\Mod^\et_X)^\op \to \Open(X)\] is a fibration whose fiber over $U$ is equivalent to $\Mod(\Pi_1^\et(U))$. We adjoin a terminal space $* = \Spec \mathbb{F}_1$ to $\Open(X)$ to form $\Open(X/*)$ and set $\Pi_1^\et(*)$ to be the terminal groupoid $\mathds{1}$. This extends the fibration to $(\Mod_{X/*}^\et)^\op \to \Open(X/*)$. The construction in \Cref{def:sheaf_cohomology} then gives a sheaf cohomology functor
\[ H^\bullet_\fet(X, -) \colon \Sh(X, \Mod^\et_X) \to \Ab. \]
The goal of this section is to compare $H^\bullet_\fet(X, -)$ with $H^\bullet_\bt(X, -)$ and $H^\bullet(X_\et, -)$. Background on profinite groupoids and \'etale fundamental groupoids is collected in \Cref{sec:profinite_groupoids} for the reader's convenience.

\subsection{Comparison with Betti Cohomology}\label{sec:profinite_betti}

We first show that $\Mod_{X/*}^\et$ is an SC5 category. Let $\PfGpd$ denote the category of profinite groupoids. The fibration $\Mod_\PfGpd^\op \to \PfGpd$ is defined by $\Mod_\PfGpd(G) \coloneqq \Mod(G)$ with the usual restriction functors.

\begin{proposition}\label{prop:ModPfGpd_SC5}
    The fibered category $\Mod_\PfGpd$ is an SC5 category.
\end{proposition}
\begin{proof}
    As in \Cref{lem:transitive_groupoid_SC,prop:groupoid_SC}, it suffices to verify the conditions for a morphism $\varphi \colon H \to G$ of profinite groups. The restriction functor $\varphi^* = \Res_H^G$ is exact and has the right adjoint
    \[ \varphi_* = \Coind_H^G \cong \varinjlim_{\substack{N \trianglelefteq G \\ N \text{ open}}} \Hom_{\Set(H)}(G/N, -), \]
    giving (\ref{axiom:SC1}) and (\ref{axiom:SC3}). The category $\Mod(G)$ admits all colimits, which are computed in $\Ab$. Since $\Set(G) = \Ind(\FSet(G))$, for every $S \in \FSet(G)$, the functor $\Hom_{\Set(G)}(S, -)$ commutes with filtered colimits. Hence,
    \[ \left\{ \mathbb{Z}(S) \mid S \in \FSet(G) \right\} \]
    is a strong generating set for $\Mod(G)$ consisting of finitely presentable objects, giving (\ref{axiom:SC2}). The abelianness of $\Mod(G)$ gives (\ref{axiom:SC4}). Finally, the displayed formula for $\Coind_H^G$ shows that it commutes with filtered colimits, giving (\ref{axiom:SC5}).
\end{proof}

Pulling back along the functor $\Pi_1^\et \colon \Open(X/*) \to \PfGpd$ gives the following corollary.

\begin{corollary}\label{cor:ModEt_SC5}
    If $X$ is a Noetherian scheme, then $\Mod_{X/*}^\et$ is an SC5 category.
\end{corollary}

\begin{remark}
    Define $\Set_\PfGpd$ and $\Set^\et_X$ analogously to $\Mod_\PfGpd$ and $\Mod^\et_X$, respectively, with $\Mod(G)$ replaced by $\Set(G)$. The proof of \Cref{prop:ModPfGpd_SC5} and the pullback argument above show that they satisfy (\ref{axiom:SC1}), (\ref{axiom:SC3}), and (\ref{axiom:SC5}). They also satisfy (\ref{axiom:SC2}), since the objects of $\FSet(G)$ form a strong generating set of $\Set(G)$ consisting of finitely presentable objects.
\end{remark}

For the remainder of this subsection, we assume that $X$ is a smooth complex variety. For every $U \in \Open(X)$, analytification of finite \'etale covers induces a functor
\[ \Mod(\Pi_1^\et(U)) \to \Mod(\Pi_1(U)). \]
These functors assemble into a morphism of SC categories $\delta^{-1} \colon \Mod_{X/*}^\et \to \Mod_{X/*}$, whose right adjoint on each fiber is given by the coinduction functors below.

Given an abstract group $G$, we have the adjunction
\[ \Res_G^{\widehat{G}} : \Mod(\widehat{G}) \rightleftarrows \Mod(G) : \Coind_G^{\widehat{G}}, \]
where $\Res_G^{\widehat{G}}$ is the forgetful functor and
\[ \Coind_G^{\widehat{G}} M = \bigcup_{\substack{N \trianglelefteq G \\ [G:N] < \infty}} M^N. \]

\begin{definition}[Serre, {\cite[Ch.~I, \S2.6]{Serre}}]\label{def:good_group}
    An abstract group $G$ is \emph{good in the sense of Serre} if for every finite $\widehat{G}$-module $M$, the natural map
    \[ H^i_\cnt(\widehat{G}, M) \to H^i_\grp(G, M) \]
    is an isomorphism for all $i \ge 0$.
\end{definition}

\begin{lemma}\label{lem:profinite_coind_vanishing}
    Let $G$ be a good group. If $M \in \Mod(G)$ is finite, then for all $i > 0$,
    \[ R^i \Coind_G^{\widehat{G}} M = 0. \]
\end{lemma}
\begin{proof}
    Let $0 \to M \to I^\bullet$ be an injective resolution in $\Mod(G)$. Since $\Mod(\widehat{G})$ is locally finitely presentable, filtered colimits are exact, so
    \[ R^i \Coind_G^{\widehat{G}} M = H^i\!\left(\Coind_G^{\widehat{G}} I^\bullet\right) = \varinjlim_{\substack{N \trianglelefteq G \\ [G:N] < \infty}} H^i\!\left((I^\bullet)^N\right) = \varinjlim_{\substack{N \trianglelefteq G \\ [G:N] < \infty}} H^i_\grp(N, M). \]
    Here, the last equality follows because $N \subseteq G$ implies that $\Res_N^G$ preserves injective objects. By \cite[Ch.~I, \S2.6, Exercise~1(a)]{Serre}, the goodness of $G$ is equivalent to the vanishing of the last colimit for all $i > 0$.
\end{proof}

\begin{theorem}[Finite \'Etale to Betti Comparison]\label{thm:profinite_betti_comparison}
    Let $X$ be a smooth complex variety, and let $\mathcal{M} \in \Sh(X, \Mod_X)$ be a locally constant sheaf of finite abelian groups. Then the natural map
    \[ H^i_\fet(X, \delta_*\mathcal{M}) \to H^i_\bt(X, \mathcal{M}) \]
    is an isomorphism for all $i \ge 0$.
\end{theorem}
\begin{proof}
    We apply \Cref{thm:comparison_general} with $\mathcal{S} = \Mod^\et_X$ and $\mathcal{T} = \Mod_X$, taking $\mathfrak{U}$ to be the basis of Artin neighborhoods. Condition~(\ref{cond:comparison_local_acyclicity}) follows from \Cref{lem:betti_kpi1}. Condition~(\ref{cond:comparison_coind_acyclicity}) follows from \Cref{lem:profinite_coind_vanishing} and the fact that the topological fundamental group of an Artin neighborhood is good \cite[Variante~4.6]{SGA4_3}\cite[Ch.~I, \S2.6, Exercice~2(d)]{Serre}.
\end{proof}

\subsection{\texorpdfstring{Comparison with \'Etale Cohomology}{Comparison with Étale Cohomology}}\label{sec:profinite_etale}

We now return to the standing assumption that $X$ is a Noetherian scheme. The goal of this subsection is to compare $H^\bullet_\fet(X, -)$ with $H^\bullet(X_\et, -)$. To this end, we introduce the finite \'etale site $X_\fet$, analogous to the Betti site but with analytic coverings replaced by finite \'etale covers. For $U \in \Open(X)$, let $\EtCov(U)$ denote the site of finite \'etale covers of $U$, equipped with the jointly surjective topology.

\begin{definition}\label{def:category_fet}
    The \emph{finite \'etale site}\footnote{The terminology follows \cite[p.~39]{Milne2}.} $X_\fet$ of $X$ is defined as follows.
    \begin{enumerate}
        \item \textbf{Objects:} A morphism of schemes $p_E\colon E \to X$ such that the image $p_E(E) \subseteq X$ is a Zariski open subset and the induced map $E \to p_E(E)$ is a finite \'etale cover.
        \item \textbf{Morphisms:} A morphism of schemes $f\colon E \to F$ making the following diagram commute.
        \[
        \begin{tikzcd}[column sep=1em, row sep=2.5em]
            E \arrow[rr, "f"] \arrow[dr, "p_E"'] & & F \arrow[dl, "p_F"] \\
            & X &
        \end{tikzcd}
        \]
    \end{enumerate}
    We endow $X_\fet$ with the Grothendieck topology generated by Zariski open covers in $X^\zar$ and jointly surjective families in $\EtCov(U)$ for every $U \in \Open(X)$. Objects of $X^\zar$ are regarded as trivial coverings.
\end{definition}

For every $U \in \Open(X)$, there are natural equivalences
\begin{equation}\label{eq:finite_etale_covering_groupoid_equivalence}
    \Phi_U : \Sh(\EtCov(U), \Set) \rightleftarrows \Set(\Pi_1^\et(U)) : \Psi_U.
\end{equation}
Explicitly,
\begin{align*}
    \Phi_U \colon \Sh(\EtCov(U), \Set) &\xrightarrow{\sim} \Set(\Pi_1^\et(U)), \\
    \mathcal{F} &\mapsto \left(x \mapsto \varinjlim_{(E, y) \to (U, x)} \mathcal{F}(E)\right),
\end{align*}
where the colimit ranges over pointed connected finite \'etale covers $(E, y) \to (U, x)$. For $E \in \EtCov(U)$, let $E_{(-)}$ denote the $\Pi_1^\et(U)$-set sending a geometric point $x \to U$ to the geometric fiber $E_x$. Then
\begin{align*}
    \Psi_U \colon \Set(\Pi_1^\et(U)) &\xrightarrow{\sim} \Sh(\EtCov(U), \Set), \\
    M &\mapsto \left(E \mapsto \Hom_{\Set(\Pi_1^\et(U))}\left(E_{(-)}, M\right)\right).
\end{align*}

Applying the argument of \Cref{sec:sheaf_via_groupoid}, mutatis mutandis, yields the following theorem.

\begin{theorem}\label{thm:fet_sheaf_equivalence}
    There is an equivalence of categories
    \[ \Phi : \Sh(X_\fet, \Set) \rightleftarrows \Sh(X, \Set^\et_X) : \Psi. \]
\end{theorem}

The same equivalence holds with $\Set$ replaced by $\Ab$. Since $X_\et$ is finer than $X_\fet$, there is a natural morphism of sites
\[ \delta \colon X_\et \to X_\fet. \]
We therefore compare sheaf cohomology on these two sites along $\delta$. For this, we recall the \'etale analogue of $K(\pi, 1)$ spaces.

\begin{definition}[Achinger, {\cite[Definition~4.1]{Achinger}}]\label{def:etale_kpi1}
    A connected scheme $X$ is an \emph{\'etale $K(\pi, 1)$ scheme} if for every locally constant \'etale sheaf $\mathcal{F}$ of finite abelian groups and every (equivalently, some) geometric base point $x \to X$, the natural map
    \[ H^i_\cnt(\pi_1^\et(X, x), \mathcal{F}_x) \to H^i(X_\et, \mathcal{F}) \]
    is an isomorphism for all $i \ge 0$.
\end{definition}

\begin{theorem}[Achinger]\label{thm:etale_kpi1}
    Let $X$ be either an Artin neighborhood over a field of characteristic $0$, or a connected affine scheme over $\mathbb{F}_p$. Then $X$ is an \'etale $K(\pi, 1)$ scheme.
\end{theorem}
\begin{proof}
    The first case follows from \cite[\S1.1]{Achinger}. The second case follows from \cite[Theorem~1.1]{Achinger}.
\end{proof}

In particular, every smooth variety over a field of characteristic $0$, and every Noetherian scheme over $\mathbb{F}_p$, admits a basis of \'etale $K(\pi, 1)$ schemes.

\begin{theorem}[Finite \'Etale to \'Etale Comparison]\label{thm:profinite_etale_comparison}
    Let $X$ be a smooth variety over a field of characteristic $0$, or a Noetherian scheme over $\mathbb{F}_p$. Let $\mathcal{F}$ be a filtered colimit of locally constant \'etale sheaves of finite abelian groups. Then the natural map
    \[ H^i_\fet(X, \Phi(\delta_* \mathcal{F})) \to H^i(X_\et, \mathcal{F}) \]
    is an isomorphism for all $i \ge 0$.
\end{theorem}
\begin{proof}
    By \Cref{thm:fet_sheaf_equivalence}, with $\Set$ replaced by $\Ab$, we have $H^i_\fet(X, \Phi(\delta_*\mathcal{F})) \cong H^i(X_\fet, \delta_*\mathcal{F})$. Therefore, it suffices to show that the natural map
    \[ H^i(X_\fet, \delta_*\mathcal{F}) \to H^i(X_\et, \mathcal{F}) \]
    is an isomorphism for all $i \ge 0$. Write $\mathcal{F} = \varinjlim_\lambda \mathcal{F}_\lambda$, where each $\mathcal{F}_\lambda$ is a locally constant \'etale sheaf of finite abelian groups.

    Let $0 \to \mathcal{F} \to \mathcal{I}^\bullet$ be an injective resolution in $\Sh(X_\et, \Ab)$. Restriction to $E_\et$ preserves injective sheaves \cite[\href{https://stacks.math.columbia.edu/tag/03F3}{Tag~03F3}]{StacksProject}, so $R^i \delta_* \mathcal{F}$ is the sheaf associated to the presheaf
    \[ E \mapsto H^i\!\left((\delta_* \mathcal{I}^\bullet)(E)\right) = H^i\!\left(\Gamma(E_\et, p_E^* \mathcal{I}^\bullet)\right) \cong H^i(E_\et, p_E^* \mathcal{F}). \]
    Since $p_E^*$ is a left adjoint, it commutes with filtered colimits. Hence, by \cite[\href{https://stacks.math.columbia.edu/tag/03Q5}{Tag~03Q5}]{StacksProject},
    \[ H^i(E_\et, p_E^* \mathcal{F}) \cong \varinjlim_\lambda H^i(E_\et, p_E^* \mathcal{F}_\lambda). \]
    Consequently, every element of $H^i(E_\et, p_E^* \mathcal{F})$ is represented by an element of $H^i(E_\et, p_E^* \mathcal{F}_\lambda)$ for some $\lambda$. By \Cref{thm:etale_kpi1} and \cite[Proposition~4.2]{Achinger}, the latter element is locally zero in the finite \'etale topology when $i > 0$. Therefore, $R^i \delta_* \mathcal{F} = 0$ for all $i > 0$, and the result follows from the Leray spectral sequence for $\delta_*$.
\end{proof}

\section{Algebraic Theory}\label{sec:algebraic}

We apply the framework of \Cref{sec:formalization} with the pro-algebraic fundamental groupoid $\Pi_1^\alg$, for which Grothendieck topologies are not appropriate tools. Fix a smooth variety $X$ over an algebraically closed field $k$ of characteristic $0$ throughout this section. The pro-algebraic fundamental groupoid $\Pi_1^\alg(X)$ is the Tannaka groupoid of the tensor category
\[ \MIC^\rs(X) \coloneqq \{\text{vector bundles with regular singular integrable connections on } X\} \]
with respect to the fiber functor over $X$ that forgets the connection \cite[Proposition~10.32(a)]{Deligne1}. Recall that $(\Rep^\alg_X)^\op \to \Open(X)$ is a fibration whose fiber over $U$ is equivalent to $\Rep(\Pi_1^\alg(U))$. We adjoin a terminal space $* \coloneqq \Spec k$ to $\Open(X)$ to form $\Open(X/*)$. This extends the fibration to $(\Rep_{X/*}^\alg)^\op \to \Open(X/*)$. The construction in \Cref{def:sheaf_cohomology} then gives a sheaf cohomology functor
\[ H^\bullet_\alg(X, -) \colon \Sh(X, \Rep^\alg_X) \to \Vect(k). \]
The goal of this section is to compare $H^\bullet_\alg(X, -)$ with $H^\bullet_\dR(X, -)$ and $H^\bullet_{\bt,\mathbb{C}}(X, -)$. Background on affine groupoid schemes, their representations, and the restriction functors is collected in \Cref{sec:affine_groupoids} for the reader's convenience.

\subsection{Comparison with Betti Cohomology}\label{sec:algebraic_betti}

We first work over an arbitrary field $k$. Let $\AfGpd$ denote the category of affine groupoid schemes over $k$. The fibration $\Rep_\AfGpd^\op \to \AfGpd$ is defined by $\Rep_\AfGpd(G) \coloneqq \Rep(G)$ with the usual restriction functors.

\begin{lemma}\label{lem:rep_adjunction}
    Given a morphism $H \to G$ in $\AfGpd$, we have an adjunction
    \[ \Res_H^G : \Rep(G) \rightleftarrows \Rep(H) : \Coind_H^G. \]
    Moreover, the coinduction functor $\Coind_H^G$ commutes with filtered colimits.
\end{lemma}
\begin{proof}
    As in \Cref{prop:groupoid_SC}, we may assume that $H$ and $G$ are transitive. We may further assume that they act on the same affine scheme $\Spec B$ by \cite[Remarque~1.8, (3.5.1)]{Deligne}. Let $H = \Spec K$ and $G = \Spec L$. According to \cite[1.14]{Deligne}, $K$ and $L$ are $B$-corings, and the morphism $H \to G$ corresponds to a coring homomorphism $L \to K$. By \cite[1.15]{Deligne}, the categories $\Rep(H)$ and $\Rep(G)$ are equivalent to the categories of comodules over $K$ and $L$, respectively. Thus, by \cite[22.12]{BrzezinskiWisbauer}, the coinduction functor is identified with the cotensor product, which commutes with filtered colimits by \cite[21.3(3)]{BrzezinskiWisbauer}.
\end{proof}

\begin{proposition}\label{prop:RepAfGpd_SC5}
    The fibration $\Rep_\AfGpd^\op \to \AfGpd$ is an SC5 fibration.
\end{proposition}
\begin{proof}
    As in \Cref{prop:groupoid_SC}, we may assume that $H$ and $G$ are transitive. Since exact sequences of vector bundles are locally split, pullback of vector bundles is exact. Thus, $\Res_H^G \colon \FRep(G) \to \FRep(H)$ is exact, and hence so is its extension to ind-categories by \cite[Corollary~8.6.8]{KashiwaraSchapira}, giving (\ref{axiom:SC3}). Conditions (\ref{axiom:SC1}) and (\ref{axiom:SC5}) follow from \Cref{lem:rep_adjunction}. By \Cref{lem:rep_ind}, we have $\Rep(G) \simeq \Ind(\FRep(G))$, so $\Rep(G)$ is abelian, giving (\ref{axiom:SC4}), and a set of representatives of the isomorphism classes in $\FRep(G)$ forms a strong generating set for $\Rep(G)$ consisting of finitely presentable objects, giving (\ref{axiom:SC2}).
\end{proof}

\begin{corollary}\label{cor:RepAlg_SC5}
    If $X$ is a smooth variety over an algebraically closed field $k$ of characteristic $0$, then $\Rep_{X/*}^\alg$ is an SC5 category.
\end{corollary}
\begin{proof}
    The result follows from \Cref{prop:RepAfGpd_SC5} by pulling back along the functor
    \[ \Pi_1^\alg \colon \Open(X/*) \to \AfGpd. \qedhere \]
\end{proof}

To compare $H^\bullet_\alg(X, -)$ with $H^\bullet_{\bt,\mathbb{C}}(X, -)$, we assume for the remainder of this subsection that $k = \mathbb{C}$, so $X$ is a smooth complex variety. By \cite[10.25]{Deligne1}, we have $\FRep(\Pi_1^\alg(U)) \simeq \FRep(\Pi_1(U))$. Passing to ind-categories gives a morphism of SC categories
\[ \delta^{-1} \colon \Rep_{X/*}^\alg \to \Rep_{X/*}, \]
given fiberwise by the full embedding $\Rep(\Pi_1^\alg(U)) \hookrightarrow \Rep(\Pi_1(U))$, whose right adjoint on each fiber is given by the coinduction functor below.

Recall that the pro-algebraic completion $G^\alg$ of an abstract group $G$ is the Tannaka group of the Tannakian category of finite-dimensional $G$-representations. Then we have the adjunction
\[ \Res_G^{G^\alg} : \Rep(G^\alg) \rightleftarrows \Rep(G) : \Coind_G^{G^\alg}, \]
where $\Res_G^{G^\alg}$ is the forgetful functor and
\begin{align*}
    \Coind_G^{G^\alg} \colon \Rep(G) &\to \Rep(G^\alg), \\
    N &\mapsto \bigcup_{\substack{L \subseteq N \\ L \in \FRep(G)}} L.
\end{align*}

\begin{definition}[cf.\ {\cite[Lemma~4.11]{KPT}}]\label{def:algebraically_good}
    An abstract group $G$ is \emph{algebraically good} if for every finite-dimensional $G^\alg$-representation $M$, the natural map
    \[ H^i_\rat(G^\alg, M) \to H^i_\grp(G, M) \]
    is an isomorphism for all $i \ge 0$.
\end{definition}

\begin{theorem}[Katzarkov--Pantev--To\"en, {\cite[Remark~4.17(ii)]{KPT}}]\label{prop:artin_nbd_alg_good}
    Let $U$ be an Artin neighborhood and let $x \in U(\mathbb{C})$ be a base point. Then $\pi_1(U, x)$ is algebraically good.
\end{theorem}

\begin{lemma}\label{lem:alg_coind_vanishing}
    Let $G$ be an algebraically good group. If $M \in \Rep(G)$ is finite-dimensional, then for all $i > 0$,
    \[ R^i \Coind_G^{G^\alg} M = 0. \]
\end{lemma}
\begin{proof}
    Let $V \in \FRep(G^\alg)$ be arbitrary. We have a Grothendieck spectral sequence
    \[ E_2^{p,q} = \Ext^p_{G^\alg}(V, R^q \Coind_G^{G^\alg} M) \Rightarrow \Ext^{p+q}_G(V, M). \]
    Since $M$ is finite-dimensional, $\Coind_G^{G^\alg} M \cong M$. Moreover, there are natural isomorphisms
    \[ \Ext^p_{G^\alg}(V, M) \cong H^p_\rat(G^\alg, V^\vee \otimes M) \quad\text{and}\quad \Ext^p_G(V, M) \cong H^p_\grp(G, V^\vee \otimes M). \]
    Hence, the edge map $E_2^{p,0} \to H^p$ corresponds to the natural map
    \[ H^p_\rat(G^\alg, V^\vee \otimes M) \to H^p_\grp(G, V^\vee \otimes M). \]
    Since $G$ is algebraically good, this is an isomorphism for all $p \ge 0$.

    We prove that $R^q \Coind_G^{G^\alg} M = 0$ for all $q > 0$ by induction. Suppose $R^j \Coind_G^{G^\alg} M = 0$ for all $0 < j < q$. Then $E_2^{p,j} = 0$ for all $p$ and $0 < j < q$. Consequently, the only potentially non-zero differential leaving the $(0,q)$-term is
    \[ d_{q+1} \colon E_{q+1}^{0,q} \to E_{q+1}^{q+1, 0}. \]
    However, since the edge map $E_2^{q+1, 0} \to H^{q+1}$ is an isomorphism, we must have $d_{q+1} = 0$, so $E_2^{0,q} \cong E_\infty^{0,q}$. On the other hand, because the edge map $E_2^{q, 0} \to H^q$ is an isomorphism,
    \[ E_2^{0,q} \cong \Hom_{\Rep(G^\alg)}(V, R^q \Coind_G^{G^\alg} M) = 0. \]
    Since $\FRep(G^\alg)$ generates $\Rep(G^\alg)$, we conclude that $R^q \Coind_G^{G^\alg} M = 0$.
\end{proof}

\begin{theorem}[Pro-Algebraic to Betti Comparison]\label{thm:alg_betti_comparison}
    Let $X$ be a smooth complex variety. Let $\mathcal{M} \in \Sh(X, \Rep_X)$ be a locally constant sheaf of finite-dimensional complex vector spaces. Then the natural map
    \[ H^i_\alg(X, \delta_*\mathcal{M}) \to H^i_{\bt,\mathbb{C}}(X, \mathcal{M}) \]
    is an isomorphism for all $i \ge 0$.
\end{theorem}
\begin{proof}
    We apply \Cref{thm:comparison_general} with $\mathcal{S} = \Rep^\alg_X$ and $\mathcal{T} = \Rep_X$, taking $\mathfrak{U}$ to be the basis of Artin neighborhoods. Condition~(\ref{cond:comparison_local_acyclicity}) follows from \Cref{lem:betti_kpi1}. Condition~(\ref{cond:comparison_coind_acyclicity}) follows from \Cref{lem:alg_coind_vanishing,prop:artin_nbd_alg_good}.
\end{proof}

\subsection{Comparison with Algebraic de Rham Cohomology}\label{sec:algebraic_dR}

We return to the standing assumptions that $k$ is an algebraically closed field of characteristic $0$ and that $X$ is a smooth variety over $k$. Let $\Mod(\mathcal{D}_U)$ denote the category of all $\mathcal{D}_U$-modules on $U \in \Open(X)$, which is a Grothendieck category by \cite[Theorem~18.1.6]{KashiwaraSchapira}. Let $j_V^U \colon V \to U$ be an open immersion in $\Open(X)$. Since $\mathcal{D}_V \cong (j_V^U)^*\mathcal{D}_U \cong \mathcal{D}_U|_V$, the pullback $(j_V^U)^* \colon \Mod(\mathcal{D}_U) \to \Mod(\mathcal{D}_V)$, pushforward $(j_V^U)_* \colon \Mod(\mathcal{D}_V) \to \Mod(\mathcal{D}_U)$, and extension by zero $(j_V^U)_! \colon \Mod(\mathcal{D}_V) \to \Mod(\mathcal{D}_U)$ are compatible with the corresponding functors on the underlying $\mathcal{O}$-modules. These functors fit into the adjunctions
\[ (j_V^U)_! \dashv (j_V^U)^* \dashv (j_V^U)_*. \]
For convenience, we write $\Dcat(U) \coloneqq \Mod(\mathcal{D}_U)$, especially when used as a subscript, and
\[ \Res_{\Dcat(V)}^{\Dcat(U)} \coloneqq (j_V^U)^* = (-)|_V \quad\text{and}\quad \Coind_{\Dcat(V)}^{\Dcat(U)} \coloneqq (j_V^U)_*. \]

We write $\Rcat(U)$ for the full subcategory of $\Mod(\mathcal{D}_U)$ consisting of \emph{ind-regular singular} $\mathcal{D}_U$-modules, which are defined as filtered unions of objects of $\MIC^\rs(U)$. Henceforth, we identify $\Rep(\Pi_1^\alg(U))$ and $\Rep^\alg_X(U)$ with $\Rcat(U)$. Colimits in $\Rep^\alg_X(U)$ and $\Mod(\mathcal{D}_U)$ are computed on the underlying sheaves. The colimit of the underlying sheaves admits the corresponding natural $\Pi_1^\alg(U)$-representation and $\mathcal{D}_U$-module structures. Thus, the full inclusion functor
\[ \Res_{\Dcat(U)}^{\Rcat(U)} \colon \Rep^\alg_X(U) \hookrightarrow \Mod(\mathcal{D}_U) \]
commutes with colimits. Therefore, by \cite[Proposition~8.3.27(iii)]{KashiwaraSchapira}, we have an adjunction
\[ \Res_{\Dcat(U)}^{\Rcat(U)} : \Rep^\alg_X(U) \rightleftarrows \Mod(\mathcal{D}_U) : \Coind_{\Dcat(U)}^{\Rcat(U)}. \]
The restriction functors satisfy the cocycle condition
\[ \Res_{\Dcat(V)}^{\Rcat(U)} \coloneqq \Res_{\Dcat(V)}^{\Rcat(V)} \circ \Res_{\Rcat(V)}^{\Rcat(U)} = \Res_{\Dcat(V)}^{\Dcat(U)} \circ \Res_{\Dcat(U)}^{\Rcat(U)}. \]
Taking right adjoints, we obtain a natural isomorphism
\[ \Coind_{\Dcat(V)}^{\Rcat(U)} \coloneqq \Coind_{\Rcat(V)}^{\Rcat(U)} \circ \Coind_{\Dcat(V)}^{\Rcat(V)} \cong \Coind_{\Dcat(U)}^{\Rcat(U)} \circ \Coind_{\Dcat(V)}^{\Dcat(U)}. \]

We next construct an adjunction
\[ \delta^* : \Sh(X, \Rep^\alg_X) \rightleftarrows \Mod(\mathcal{D}_X) : \delta_*. \]
Given $\mathcal{F} \in \Mod(\mathcal{D}_X)$, define $\delta_* \mathcal{F}$ by
\[ (\delta_* \mathcal{F})(U) \coloneqq \Coind_{\Dcat(U)}^{\Rcat(U)} \mathcal{F}|_U. \]
For every inclusion $V \subseteq U$, we have $(\delta_* \mathcal{F})(U)|_V \subseteq (\delta_* \mathcal{F})(V)$, so these inclusions make $\delta_* \mathcal{F}$ a presheaf.

\begin{lemma}\label{lem:delta_sheaf}
    For every $\mathcal{F} \in \Mod(\mathcal{D}_X)$, the presheaf $\delta_* \mathcal{F}$ is a sheaf in $\Sh(X, \Rep^\alg_X)$.
\end{lemma}
\begin{proof}
    Fix $U \in \Open(X)$, and let $\mathcal{U}$ be an open cover of $U$. By the sheaf condition for $\mathcal{F}$,
    \[ \mathcal{F}|_U \to \prod_{V \in \mathcal{U}} \Coind_{\Dcat(V)}^{\Dcat(U)} \mathcal{F}|_V \rightrightarrows \prod_{V, W \in \mathcal{U}} \Coind_{\Dcat(V \cap W)}^{\Dcat(U)} \mathcal{F}|_{V \cap W} \]
    is an equalizer diagram. Since $\Coind_{\Dcat(U)}^{\Rcat(U)}$ is a right adjoint, it preserves limits, so applying it gives an equalizer diagram
    \[ \Coind_{\Dcat(U)}^{\Rcat(U)} \mathcal{F}|_U \to \prod_{V \in \mathcal{U}} \Coind_{\Dcat(V)}^{\Rcat(U)} \mathcal{F}|_V \rightrightarrows \prod_{V, W \in \mathcal{U}} \Coind_{\Dcat(V \cap W)}^{\Rcat(U)} \mathcal{F}|_{V \cap W}, \]
    where we used the cocycle condition for coinduction. This is precisely the equalizer diagram
    \[ (\delta_* \mathcal{F})(U) \to \prod_{V \in \mathcal{U}} (\delta_* \mathcal{F})(V) \rightrightarrows \prod_{V, W \in \mathcal{U}} (\delta_* \mathcal{F})(V \cap W), \]
    which is the sheaf condition for $\delta_* \mathcal{F}$.
\end{proof}

Let $\mathcal{M} \in \Sh(X, \Rep^\alg_X)$. For $V \subseteq U$ in $\Open(X)$, the morphism
\[ \Res_{\Rcat(V)}^{\Rcat(U)} \mathcal{M}(U) = \mathcal{M}(U)|_V \to \mathcal{M}(V) \]
induces the composite
\[ \mathcal{M}(U)(U) \to \mathcal{M}(U)(V) \to \mathcal{M}(V)(V). \]
Thus, $U \mapsto \mathcal{M}(U)(U)$ is a presheaf of $\mathcal{D}_X$-modules. Define
\[ \delta^* \mathcal{M} \coloneqq \bigl(U \mapsto \mathcal{M}(U)(U)\bigr)^\#, \]
where $(-)^\#$ denotes sheafification in the category of presheaves of $\mathcal{D}_X$-modules.

\begin{lemma}\label{lem:delta_adjunction}
    There is an adjunction $\delta^* : \Sh(X, \Rep^\alg_X) \rightleftarrows \Mod(\mathcal{D}_X) : \delta_*$.
\end{lemma}
\begin{proof}
    Let $\mathcal{M} \in \Sh(X, \Rep^\alg_X)$ and $\mathcal{F} \in \Mod(\mathcal{D}_X)$. By the universal property of sheafification, a morphism $\delta^* \mathcal{M} \to \mathcal{F}$ is equivalent to a family of morphisms
    \[ \mathcal{M}(U)(U) \to \mathcal{F}(U) \]
    of $\mathcal{D}_X(U)$-modules, natural in $U$, or equivalently, to a family of morphisms
    \[ \mathcal{M}(U) \to \mathcal{F}|_U \]
    in $\Mod(\mathcal{D}_U)$, natural in $U$. By the fiberwise adjunctions and their compatibility with restriction, such a family corresponds to a family of morphisms
    \[ \mathcal{M}(U) \to \Coind_{\Dcat(U)}^{\Rcat(U)}(\mathcal{F}|_U) \]
    in $\Rep^\alg_X(U)$, natural in $U$, which is precisely a morphism $\mathcal{M} \to \delta_* \mathcal{F}$. Therefore, naturally in $\mathcal{M}$ and $\mathcal{F}$,
    \[ \Hom_{\Mod(\mathcal{D}_X)}(\delta^* \mathcal{M}, \mathcal{F}) \cong \Hom_{\Sh(X, \Rep^\alg_X)}(\mathcal{M}, \delta_* \mathcal{F}). \qedhere \]
\end{proof}

\begin{lemma}\label{lem:delta_exact}
    The functor $\delta^* \colon \Sh(X, \Rep^\alg_X) \to \Mod(\mathcal{D}_X)$ is exact.
\end{lemma}
\begin{proof}
    For each $U$, the functor $\mathcal{M} \mapsto \mathcal{M}(U)(U)$ is the composite of the functor $\mathcal{M} \mapsto \mathcal{M}(U)$, the inclusion $\Rcat(U) \hookrightarrow \Mod(\mathcal{D}_U)$, and the ordinary section functor, all of which are left exact. Since ordinary sheafification of presheaves of $\mathcal{D}_X$-modules is exact, $\delta^*$ is left exact. Since $\delta^*$ is a left adjoint by \Cref{lem:delta_adjunction}, it is also right exact.
\end{proof}

\begin{remark}\label{rem:alg_dR_edge_map}
    Recall that $H^i_\dR(X, -) \cong \Ext^i_{\Dcat(X)}(\mathcal{O}_X, -)$ via the Spencer resolution, where $\mathcal{O}_X$ is regarded as a $\mathcal{D}_X$-module via the trivial connection $d \colon \mathcal{O}_X \to \Omega^1_{X/k}$ \cite[23.2.1]{ABC}. Then
    \begin{align*}
        \Hom_{\Dcat(X)}(\mathcal{O}_X, -)
        &\cong \Hom_{\Rcat(X)}\!\left(\mathcal{O}_X, \Coind_{\Dcat(X)}^{\Rcat(X)} -\right) \\
        &\cong \left(\Coind_{\Dcat(X)}^{\Rcat(X)} -\right)^{\Pi_1^\alg(X)} \\
        &\cong \Gamma(X, -)^{\Pi_1^\alg(X)} \circ \delta_*.
    \end{align*}
    By \Cref{lem:delta_exact}, the right adjoint $\delta_*$ preserves injective objects. Hence, for $M \in \Mod(\mathcal{D}_X)$, the Grothendieck spectral sequence for $\Gamma(X, -)^{\Pi_1^\alg(X)} \circ \delta_*$ takes the form
    \begin{equation}\label{eq:alg_dR_spectral_sequence}
        E_2^{p,q} = H^p_\alg\!\left(X, R^q \delta_* M\right) \Rightarrow H^{p+q}_\dR(X, M).
    \end{equation}
\end{remark}

\begin{definition}\label{def:alg_kpi1}
    A smooth connected variety $X$ over $k$ is an \emph{algebraic $K(\pi, 1)$ variety} if for every ind-regular singular $\mathcal{D}_X$-module $M$ and every (equivalently, some) base point $x \in X(k)$, the natural map
    \[ H^i_\rat(\pi_1^\alg(X, x), M_x) \to H^i_\dR(X, M) \]
    is an isomorphism for all $i \ge 0$.
\end{definition}

\begin{remark}\label{rem:alg_dR_edge_maps}
    Let $M \in \Rcat(X)$. Since $(\delta_* M)(X) \cong M$, the spectral sequence \eqref{eq:groupoid_section_spectral_sequence}, applied to $\delta_* M$, gives an edge map
    \[ H^i_\rat(\Pi_1^\alg(X), M) \to H^i_\alg(X, \delta_* M). \]
    Together with the edge map of \eqref{eq:alg_dR_spectral_sequence}, this fits into a natural commutative diagram
    \[
    \begin{tikzcd}
        H^i_\rat(\Pi_1^\alg(X), M) \arrow[r] \arrow[d, "{\rotatebox{-90}{\scalebox{1}{$\sim$}}}"'] & H^i_\alg(X, \delta_* M) \arrow[r] & H^i_\dR(X, M) \arrow[d, "{\rotatebox{-90}{\scalebox{1}{$\sim$}}}"'] \\
        \Ext^i_{\Rcat(X)}(\mathcal{O}_X, M) \arrow[rr] & & \Ext^i_{\Dcat(X)}(\mathcal{O}_X, M).
    \end{tikzcd}
    \]
    If $X$ is connected, evaluation at a base point $x \in X(k)$ identifies $\Pi_1^\alg(X)$-representations with $\pi_1^\alg(X, x)$-representations. Thus, the natural map in \Cref{def:alg_kpi1} identifies with
    \[ \Ext^i_{\Rcat(X)}(\mathcal{O}_X, M) \to \Ext^i_{\Dcat(X)}(\mathcal{O}_X, M). \]
\end{remark}

We defer the proof of the following theorem to \Cref{sec:algebraic_kpi1}.

\begin{theorem}\label{thm:alg_kpi1}
    Let $X$ be an Artin neighborhood over an algebraically closed field $k$ of characteristic $0$. Then $X$ is an algebraic $K(\pi, 1)$ variety.
\end{theorem}

\begin{lemma}\label{lem:dR_coind_vanishing}
    Let $X$ be an algebraic $K(\pi, 1)$ variety. If $M \in \Mod(\mathcal{D}_X)$ is ind-regular singular, then for all $i > 0$,
    \[ R^i \Coind_{\Dcat(X)}^{\Rcat(X)} M = 0. \]
\end{lemma}
\begin{proof}
    Let $V \in \MIC^\rs(X)$. We have the Grothendieck spectral sequence
    \[ E_2^{p,q} = \Ext^p_{\Rcat(X)}(V, R^q \Coind_{\Dcat(X)}^{\Rcat(X)} M) \Rightarrow \Ext^{p+q}_{\Dcat(X)}(V, M). \]
    Since $M$ is ind-regular singular, $\Coind_{\Dcat(X)}^{\Rcat(X)} M \cong M$. Via the natural tensor--Hom identifications, the edge map identifies with
    \[ \Ext^p_{\Rcat(X)}(\mathcal{O}_X, V^\vee \otimes M) \to \Ext^p_{\Dcat(X)}(\mathcal{O}_X, V^\vee \otimes M). \]
    By \Cref{rem:alg_dR_edge_maps}, applied to $V^\vee \otimes M$, this is the natural map in \Cref{def:alg_kpi1}, and hence is an isomorphism for all $p \ge 0$. The inductive argument in the proof of \Cref{lem:alg_coind_vanishing}, mutatis mutandis, then gives
    \[ \Hom_{\Rcat(X)}\!\left(V, R^q \Coind_{\Dcat(X)}^{\Rcat(X)} M\right) = 0 \]
    for every $q > 0$. Since $\MIC^\rs(X)$ generates $\Rcat(X)$, the result follows.
\end{proof}

\begin{lemma}\label{lem:dR_delta_vanishing}
    If $\mathcal{F} \in \Mod(\mathcal{D}_X)$ is ind-regular singular, then for all $i > 0$,
    \[ R^i \delta_* \mathcal{F} = 0. \]
\end{lemma}
\begin{proof}
    Let $0 \to \mathcal{F} \to \mathcal{I}^\bullet$ be an injective resolution in $\Mod(\mathcal{D}_X)$. Since extension by zero of $\mathcal{D}_U$-modules is computed on the underlying sheaves of abelian groups, $(j_U^X)_!$ is exact by \cite[\href{https://stacks.math.columbia.edu/tag/01AK}{Tag~01AK}]{StacksProject}, and its right adjoint $(j_U^X)^* = (-)|_U$ preserves injective objects. Hence, $R^i \delta_* \mathcal{F}$ is the sheaf associated to the presheaf
    \[ U \mapsto H^i\!\left((\delta_* \mathcal{I}^\bullet)(U)\right) \cong H^i\!\left(\Coind_{\Dcat(U)}^{\Rcat(U)}\!\left(\mathcal{I}^\bullet|_U\right)\right) \cong R^i \Coind_{\Dcat(U)}^{\Rcat(U)}\!\left(\mathcal{F}|_U\right). \]
    By \Cref{thm:alg_kpi1,prop:artin_good_neighborhood}, $X$ admits a basis $\mathfrak{U}$ of Artin neighborhoods, each of which is an algebraic $K(\pi, 1)$ variety. For every $U \in \mathfrak{U}$, \Cref{lem:dR_coind_vanishing} gives
    \[ R^i \Coind_{\Dcat(U)}^{\Rcat(U)}\!\left(\mathcal{F}|_U\right) = 0 \]
    for all $i > 0$. Thus, the result follows from the universal property of sheafification and \Cref{lem:morphisms_on_basis}.
\end{proof}

\begin{theorem}[Pro-Algebraic to de Rham Comparison]\label{thm:alg_dR_comparison}
    Let $X$ be a smooth variety over an algebraically closed field $k$ of characteristic $0$. Let $\mathcal{F} \in \Mod(\mathcal{D}_X)$ be an ind-regular singular $\mathcal{D}_X$-module. Then the natural map
    \[ H^i_\alg(X, \delta_*\mathcal{F}) \to H^i_\dR(X, \mathcal{F}) \]
    is an isomorphism for all $i \ge 0$.
\end{theorem}
\begin{proof}
    By \Cref{lem:dR_delta_vanishing}, $R^i \delta_* \mathcal{F} = 0$ for all $i > 0$. Hence, the spectral sequence \eqref{eq:alg_dR_spectral_sequence} degenerates, and its edge map is the desired isomorphism.
\end{proof}

\subsection{\texorpdfstring{Algebraic $K(\pi, 1)$ Varieties}{Algebraic K(π, 1) Varieties}}\label{sec:algebraic_kpi1}

We now prove \Cref{thm:alg_kpi1}. We first reduce the assertion to a weaker condition. We then establish this condition over $\mathbb{C}$, using that the fundamental group of an Artin neighborhood is algebraically good. The general case follows by spreading out and base change.

\begin{lemma}\label{lem:alg_kpi1_surjective}
    Suppose that $X$ is connected, and let $x \in X(k)$ be a base point. If the natural map
    \[ H^i_\rat(\pi_1^\alg(X, x), M_x) \to H^i_\dR(X, M) \]
    is surjective for every finite-dimensional $\Pi_1^\alg(X)$-representation $M$ and every $i \ge 0$, then $X$ is an algebraic $K(\pi, 1)$ variety.
\end{lemma}
\begin{proof}
    By \cite[Lemma~I.4.17]{Jantzen}, $H^i_\rat(\pi_1^\alg(X, x), -)$ commutes with filtered colimits. By \cite[\href{https://stacks.math.columbia.edu/tag/01FF}{Tag~01FF}]{StacksProject}, cohomology of $\mathcal{O}_X$-modules commutes with filtered colimits. The hypercohomology spectral sequence
    \[
    E_1^{p,q} = H^q\left(X, - \otimes_{\mathcal{O}_X} \Omega^p_{X/k}\right) \Rightarrow H^{p+q}_\dR(X, -)
    \]
    together with the Comparison Theorem for spectral sequences \cite[Theorem~5.2.12]{Weibel} then shows that $H^i_\dR(X, -)$ commutes with filtered colimits. Let $M$ be an arbitrary $\Pi_1^\alg(X)$-representation. By \cite[Corollaire~3.9]{Deligne}, $M$ is a filtered union of its finite-dimensional sub-representations $M_i$. Since the natural map is surjective for every $M_i$, passing to the filtered colimit shows that it is surjective for $M$. Applying \cite[Lemma~11]{Anderson} with
    \[ F : \Rcat(X) \hookrightarrow \Dcat(X) \quad\text{and}\quad H \coloneqq \Hom_{\Dcat(X)}(\mathcal{O}_X, -), \]
    the natural map $\Ext^i_{\Rcat(X)}(\mathcal{O}_X, M) \to \Ext^i_{\Dcat(X)}(\mathcal{O}_X, M)$ is an isomorphism for all $i \ge 0$.
\end{proof}

\begin{lemma}\label{lem:alg_kpi1_C}
    An Artin neighborhood $X$ over $\mathbb{C}$ is an algebraic $K(\pi, 1)$ variety.
\end{lemma}
\begin{proof}
    Let $M$ be a finite-dimensional $\Pi_1^\alg(X)$-representation, and let $x \in X(\mathbb{C})$ be a base point. By \Cref{prop:artin_nbd_alg_good}, we have natural isomorphisms
    \[ H^i_\rat(\pi_1^\alg(X, x), M_x) \xrightarrow{\sim} H^i_\grp(\pi_1(X, x), M_x) \]
    for all $i \ge 0$. Let $\mathcal{L}_M$ denote the local system on $X^\an$ associated with $M_x$. Since $X^\an$ is a $K(\pi, 1)$ space by \Cref{thm:top_kpi1}, the Poincaré lemma and \cite[Th\'eor\`eme~II.6.2]{Deligne4} give natural isomorphisms
    \[ H^i_\grp(\pi_1(X, x), M_x) \cong H^i(X^\an, \mathcal{L}_M) \cong H^i_{\dR,\an}(X, M^\an) \cong H^i_\dR(X, M). \]
    The hypothesis of \Cref{lem:alg_kpi1_surjective} is satisfied, so the result follows.
\end{proof}

\begin{lemma}\label{lem:alg_ext_surjectivity_base_change}
    Let $k \subseteq K$ be an extension of algebraically closed fields of characteristic $0$. Let $X$ be a smooth affine variety over $k$, and let $M \in \MIC^\rs(X)$. Set $X_K \coloneqq X \otimes_k K$ and $M_K \coloneqq M \otimes_k K$. For every $i \ge 0$, the natural map
    \[ \Ext^i_{\Rcat(X)}\!\left(\mathcal{O}_X, M\right) \to \Ext^i_{\Dcat(X)}\!\left(\mathcal{O}_X, M\right) \]
    is surjective if and only if the natural map
    \[ \Ext^i_{\Rcat(X_K)}\!\left(\mathcal{O}_{X_K}, M_K\right) \to \Ext^i_{\Dcat(X_K)}\!\left(\mathcal{O}_{X_K}, M_K\right) \]
    is surjective.
\end{lemma}
\begin{proof}
    The assertion for $i = 0$ is immediate, so fix $i > 0$. Recall the natural isomorphism
    \[ \Ext^i_{\Dcat(X)}(\mathcal{O}_X, -) \cong H^i_\dR(X, -). \]

    \addvspace{\medskipamount}\noindent\textbf{($\Rightarrow$)}
    Suppose that the map over $X$ is surjective. Choose a basis $e_0, \ldots, e_{m-1}$ of $H^i_\dR(X, M)$ and a preimage $\widetilde{e}_j$ of each $e_j$. Since $\Rcat(X) = \Ind(\MIC^\rs(X))$, each $\widetilde{e}_j$ can be represented by a Yoneda $i$-extension with middle terms in $\MIC^\rs(X)$ \cite[Theorem~15.3.1(i)]{KashiwaraSchapira}. After base change to $K$, these extensions have regular singular middle terms, and their classes form a basis of
    \[ H^i_\dR(X_K, M_K) \cong H^i_\dR(X, M) \otimes_k K. \]
    Hence, the map over $X_K$ is surjective.

    \addvspace{\medskipamount}\noindent\textbf{($\Leftarrow$)}
    Assume that the map over $X_K$ is surjective. Fix a basis $e_0, \ldots, e_{m-1}$ of $H^i_\dR(X_K, M_K)$ and a preimage $\widetilde{e}_j$ of each $e_j$. Since $\Rcat(X_K) = \Ind(\MIC^\rs(X_K))$, each $\widetilde{e}_j$ can be represented by a Yoneda $i$-extension
    \[ 0 \to M_K \to E^{(j)}_{i-1} \to \cdots \to E^{(j)}_0 \to \mathcal{O}_{X_K} \to 0, \]
    where $E^{(j)}_n \in \MIC^\rs(X_K)$ \cite[Theorem~15.3.1(i)]{KashiwaraSchapira}.

    Choose a smooth compactification $\overline{X}_K$ of $X_K$ such that
    \[ D := \overline{X}_K \setminus X_K = \bigcup_{\alpha \in I} D_\alpha \]
    is a simple normal crossings divisor with smooth irreducible components $D_\alpha$. For every $j$ and $n$, choose a logarithmic extension $\overline{E}^{(j)}_n$ of $E^{(j)}_n$ to $\overline{X}_K$. Choose a finitely generated $k$-subalgebra $R \subseteq K$ over which the compactification, the boundary components, the logarithmic extensions, and the morphisms in the Yoneda $i$-extensions are defined.

    By \cite[Proposition~17.7.8(ii)]{EGA4_4}, after localizing $R$, we may assume that $\overline{X}_R$ is smooth over $R$ and that
    \[ D_{J,R} := \bigcap_{\alpha \in J} D_{\alpha,R} \]
    is smooth over $R$ of relative codimension $|J|$ for every nonempty $J \subseteq I$. After further localizing $R$, we may assume that every logarithmic extension $\overline{E}^{(j)}_{n,R}$ is locally free of finite rank. Its restriction $E^{(j)}_{n,R}$ to $X_R$ is then locally free of finite rank. Since $X_R$ is affine, the cohomology sheaves of the spread-out Yoneda $i$-extensions correspond to finitely generated modules. Their generic fibers vanish because the extensions become exact after the faithfully flat base change $\Frac(R) \to K$. After further localizing $R$, we may therefore assume that all the spread-out Yoneda $i$-extensions are exact.

    Since $R$ is flat over $k$, base change of the de Rham complex gives
    \[ H^i_\dR(X_R/R, M_R) \cong H^i_\dR(X, M) \otimes_k R. \]
    Let $e_{j,R} \in H^i_\dR(X_R/R, M_R)$ be the class of the spread-out Yoneda extension indexed by $j$. Since its base change to $K$ is $e_j$, faithful flatness of $\Frac(R) \to K$ shows that $e_{0,R}, \ldots, e_{m-1,R}$ form a basis after tensoring with $\Frac(R)$. After further localizing $R$, we may assume that $H^i_\dR(X_R/R, M_R)$ is a free $R$-module of rank $m$ with basis $e_{0,R}, \ldots, e_{m-1,R}$.

    Choose a closed point $s \in \Spec R$. Since $k$ is algebraically closed, its residue field is $k$. The spread-out Yoneda $i$-extensions remain exact after base change to $s$ because their terms are locally free. Moreover, their middle terms are regular singular, since each $\overline{E}^{(j)}_{n,R}$ specializes to a logarithmic connection on $\overline{X}_s$. The specialized classes $e_{0,s}, \ldots, e_{m-1,s}$ form a basis of $H^i_\dR(X, M)$. Hence, the map over $X$ is surjective.
\end{proof}

\begin{lemma}\label{lem:alg_ext_surjectivity_artin}
    Let $X$ be an Artin neighborhood over an algebraically closed field $k$ of characteristic $0$, and let $M \in \MIC^\rs(X)$. Then the natural map
    \[ \Ext^i_{\Rcat(X)}\!\left(\mathcal{O}_X, M\right) \to \Ext^i_{\Dcat(X)}\!\left(\mathcal{O}_X, M\right) \]
    is surjective for all $i \ge 0$.
\end{lemma}
\begin{proof}
    Choose an algebraically closed subfield $k_0 \subseteq k$ of finite transcendence degree over $\mathbb{Q}$ such that an Artin-neighborhood presentation of $X$ and a logarithmic extension of $M$ descend to $k_0$. Let $X_0$ and $M_0$ denote the descended objects. Then $X_0$ is an Artin neighborhood, $M_0 \in \MIC^\rs(X_0)$, and
    \[ X \cong X_0 \otimes_{k_0} k \quad\text{and}\quad M \cong M_0 \otimes_{k_0} k. \]

    Since $k_0$ embeds into $\mathbb{C}$, \Cref{rem:alg_dR_edge_maps,lem:alg_kpi1_C} show that the natural map
    \[ \Ext^i_{\Rcat((X_0)_\mathbb{C})}\!\left(\mathcal{O}_{(X_0)_\mathbb{C}}, (M_0)_\mathbb{C}\right) \to \Ext^i_{\Dcat((X_0)_\mathbb{C})}\!\left(\mathcal{O}_{(X_0)_\mathbb{C}}, (M_0)_\mathbb{C}\right) \]
    is surjective for all $i \ge 0$. Applying \Cref{lem:alg_ext_surjectivity_base_change} first to $k_0 \subseteq \mathbb{C}$ and then to $k_0 \subseteq k$, we conclude that the map over $X$ is surjective for all $i \ge 0$.
\end{proof}

\Cref{thm:alg_kpi1} now follows from \Cref{rem:alg_dR_edge_maps,lem:alg_kpi1_surjective,lem:alg_ext_surjectivity_artin}.

\section{Unified Theory}\label{sec:unified}

The goal of this section is to provide a uniform construction of \'etale and algebraic de Rham cohomology, over an algebraically closed field $k$ of characteristic $0$. Let $X$ be a smooth variety over $k$, and choose $x \in X(k)$. Then $\pi_1^\et(X, x)$ is the profinite completion of $\pi_1^\alg(X, x)$, so $\Pi_1^\alg(X)$ already encodes all the information of $\Pi_1^\et(X)$. Thus, $\Pi_1^\alg(X)$ alone should recover \'etale cohomology. However, the coefficient category used in \Cref{sec:algebraic} is simply too small to recover \'etale cohomology, in particular its torsion classes.

In this section, we enlarge the coefficient category to the category $\FC$ of commutative formal groups over $k$ and show that the resulting cohomology theory simultaneously recovers both \'etale and algebraic de Rham cohomology. Let $\AC$ denote the category of commutative affine group schemes over $k$. Cartier duality gives anti-equivalences \cite[Ch.~II, \S4]{Demazure}
\[ (-)^\dual \colon \AC^\op \xrightarrow{\sim} \FC \quad\text{and}\quad (-)^\dual \colon \FC^\op \xrightarrow{\sim} \AC. \]
By \cite[Ch.~II, \S6, Theorem~3]{Demazure} and \cite[\href{https://stacks.math.columbia.edu/tag/01ZB}{Tag~01ZB}]{StacksProject}, via Cartier duality, $\FC$ is abelian (\ref{axiom:SC4}) and locally finitely presentable (\ref{axiom:SC2}), with the Cartier duals of commutative algebraic group schemes forming a strong generating set consisting of finitely presentable objects. Hence, it is a suitable coefficient category for our cohomology theory.

In this section, for an affine groupoid scheme $G$ acting on a scheme $S$, we use the notation of \Cref{sec:groupoids_via_structure_maps}. That is, the structure maps are denoted
\[ s, t \colon G \to S, \quad e \colon S \to G, \quad i \colon G \to G, \quad c \colon G^{(2)} \to G, \]
where $G^{(2)} \coloneqq G \times_{s, S, t} G$, with projections $\pr_0, \pr_1 \colon G^{(2)} \to G$.

The first two subsections are developed over an arbitrary field $k$. From the third subsection onwards, we additionally assume $\charac k = 0$ and $k = \bar{k}$. For background on formal groups and Cartier duality, we refer to \cite{Demazure}.

\subsection{Groupoid Module Theory with Formal Group Schemes}\label{sec:groupoid_formal}

The goal of this subsection is to define the category $\FMod(G)$ of formal $G$-modules, which are $G$-modules with coefficients in $\FC$. We also define related categories and discuss their basic properties. A direct definition presents a technical difficulty because $G$ is a scheme, whereas an object of $\FC$ is a formal scheme. We therefore encode the desired action indirectly, following the contravariant pattern of Cartier duality.

Consider the classical situation where $G$ is an abstract group and $M$ is a finite-dimensional $G$-representation. The dual vector space $M^\dual$ has a natural right $G$-representation structure, given by
\[ (\varphi \cdot g)(m) \coloneqq \varphi(g \cdot m) \]
for $\varphi \in M^\dual$, $g \in G$, and $m \in M$. This yields a canonical equivalence
\[ \FRep_k(G) \simeq \FRep_k(G^\op)^\op. \]
Cartier duality suggests the same categorical pattern. Accordingly, we first develop the category $\AMod(G)$ of affine $G$-modules and then define $\FMod(G)$ as $\AMod(G^\op)^\op$. For the terminal groupoid $\mathds{1}$, we obtain an equivalence $\FC \simeq \FMod(\mathds{1})$.

\begin{definition}\label{def:asch}
    An \emph{affine $G$-scheme} is a pair $(M, \rho)$, where $M$ is an affine scheme over $S$ and
    \[ \rho \colon s^* M \xrightarrow{\sim} t^* M \]
    is an isomorphism of affine schemes over $G$ satisfying
    \[ e^* \rho = \id_M \quad\text{and}\quad c^* \rho = \pr_0^* \rho \circ \pr_1^* \rho. \]
    A morphism $(M, \rho) \to (N, \sigma)$ is a morphism $\eta \colon M \to N$ of affine schemes over $S$ satisfying
    \[ t^* \eta \circ \rho = \sigma \circ s^* \eta. \]
    We denote by $\ASch(G)$ the category of affine $G$-schemes and often write $M$ for $(M, \rho)$ when no confusion arises.
\end{definition}

\begin{remark}\label{rem:asch_functor}
    Let $M$ and $N$ be affine schemes over $S$. Giving $M$ the structure of an affine $G$-scheme is equivalent to giving $M(A)$ the structure of a $G(A)$-set for every test scheme $A$, functorially in $A$. Similarly, a morphism $M \to N$ of affine $G$-schemes corresponds to a collection of $G(A)$-equivariant maps $M(A) \to N(A)$, functorial in $A$. Thus, $\ASch(G)$ is the scheme-theoretic analogue of the category of groupoid sets for a fixed abstract groupoid.
\end{remark}

\begin{definition}\label{def:amod}
    An \emph{affine $G$-module} is an abelian group object in $\ASch(G)$. Equivalently, it is a pair $(M, \rho)$ as in \Cref{def:asch} such that $M$ is a commutative affine group scheme over $S$ and $\rho \colon s^* M \xrightarrow{\sim} t^* M$ is an isomorphism of commutative group schemes over $G$. We denote by $\AMod(G) \coloneqq \AbObj(\ASch(G))$ the category of affine $G$-modules.

    We say that $M \in \AMod(G)$ is \emph{pro-unipotent} (resp. \emph{of multiplicative type}) if the geometric fiber of $M$ over every geometric point of $S$ is pro-unipotent (resp. of multiplicative type). We denote by $\UMod(G)$ (resp. $\MMod(G)$) the full subcategory of $\AMod(G)$ spanned by pro-unipotent affine $G$-modules (resp. affine $G$-modules of multiplicative type).
\end{definition}

Throughout the remainder of this section, all points are scheme-theoretic $A$-points, with the test scheme $A$ possibly implicit. We call $A$-points of $S$ \emph{objects} and $A$-points of $G$ \emph{arrows}. We write $f \colon x \to y$ to mean $s(f) = x$ and $t(f) = y$, and for $f \colon x \to y$ and $g \colon y \to z$, we set $\id_x \coloneqq e(x)$, $f^{-1} \coloneqq i(f)$, and $g \circ f \coloneqq c(g, f)$. For affine $G$-schemes $(M, \rho)$ and $(N, \sigma)$ and a morphism $\eta \colon (M, \rho) \to (N, \sigma)$, we further set $M_x \coloneqq x^* M$, $\rho_f \coloneqq f^* \rho$, and $\eta_x \coloneqq x^* \eta$.

\begin{remark}\label{rem:asch_pts}
    We explain the meaning of \Cref{def:asch} from the viewpoint of scheme-theoretic points. Fix objects $x, y, z \in S(A)$ and arrows $f \colon x \to y$ and $g \colon y \to z$. Giving $\rho$ is equivalent to giving, for each arrow $f \colon x \to y$, an isomorphism
    \[ \rho_f \colon M_x \xrightarrow{\sim} M_y \]
    of affine schemes over $A$, functorially in $A$, subject to
    \[ \rho_{\id_x} = \id_{M_x} \quad\text{and}\quad \rho_{g \circ f} = \rho_g \circ \rho_f. \]
    A morphism of affine $G$-schemes is then a collection of morphisms $\eta_x \colon M_x \to N_x$, functorially in $A$, satisfying
    \[ \eta_y \circ \rho_f = \sigma_f \circ \eta_x. \]
    Conversely, taking $A = S$ and $x = \id_S$ recovers the identity condition. Taking $A = G^{(2)}$, $f = \pr_1$, and $g = \pr_0$ recovers the composition condition. Taking $A = G$, $f = \id_G$, $x = s$, and $y = t$ recovers the condition on a morphism.
\end{remark}

\begin{definition}\label{def:asch_res}
    Let $(\varphi, u) \colon (H, T) \to (G, S)$ be a morphism of affine groupoid schemes. The \emph{restriction functor}
    \[ \Res_H^G \colon \ASch(G) \to \ASch(H) \]
    is defined by sending $(M, \rho)$ to $(u^* M, \varphi^* \rho)$ and a morphism $\eta$ to $u^* \eta$.
\end{definition}

\begin{remark}
    The conditions $s_G \circ \varphi = u \circ s_H$ and $t_G \circ \varphi = u \circ t_H$ yield canonical isomorphisms $s_H^* u^* M \cong \varphi^* s_G^* M$ and $\varphi^* t_G^* M \cong t_H^* u^* M$. Thus, $\varphi^* \rho$ may be viewed as an isomorphism
    \[ \varphi^* \rho \colon s_H^* u^* M \xrightarrow{\sim} t_H^* u^* M. \]
    For each test scheme $A$, the functor $\Res_H^G$ evaluated at $A$ is merely the restriction $\Res_{H(A)}^{G(A)}$ of $G(A)$-sets to $H(A)$-sets, functorial in $A$. Thus, by \Cref{rem:asch_functor}, $\Res_H^G$ is a well-defined functor.
\end{remark}

Given a morphism $u \colon T \to S$, the fiber product $G_T \coloneqq T \times_{u,S,s} G \times_{t,S,u} T$ is a groupoid scheme acting on $T$, with
\[ e_T(x) \coloneqq (x,\, e(u(x)),\, x), \ (x, f, y)^{-1} \coloneqq (y,\, f^{-1},\, x), \ \text{and} \ (y, g, z) \circ (x, f, y) \coloneqq (x,\, g \circ f,\, z). \]
The anchor map $(s_T, t_T) \colon G_T \to T \times T$ is given by the first and third projections and is affine. The natural groupoid morphism $G_T \to G$ is the second projection.

\begin{definition}\label{def:asch_base_change}
    The \emph{base change functor} along $u \colon T \to S$ is defined as the restriction
    \[ (-)_T \coloneqq \Res_{G_T}^G \colon \ASch(G) \to \ASch(G_T). \]
\end{definition}

As in \cite[Remarque~1.8, (3.5.1)]{Deligne}, base change of the object scheme does not change the category of affine $G$-schemes when $G$ is transitive and $T \neq \emptyset$. The proof is somewhat technical, and we defer it to \Cref{sec:base_change}.

\begin{theorem}[Base Change Theorem]\label{thm:base_change}
    Let $G$ be a transitive affine groupoid scheme acting on $S$, and let $u \colon T \to S$ be a morphism of schemes with $T \neq \emptyset$. Then the base change functor
    \[ (-)_T \colon \ASch(G) \to \ASch(G_T) \]
    is an equivalence of categories.
\end{theorem}

\begin{remark}\label{rem:asch_spec_k}
    Suppose that $G$ is transitive and that there is a rational point $x \in S(k)$. Then by \Cref{thm:base_change},
    \[ \ASch(G) \simeq \ASch(G_x). \]
    Here, $G_x$ is an affine group scheme over $k$. By \Cref{rem:asch_functor}, an object of $\ASch(G_x)$ is an affine scheme $M$ over $k$ equipped with a $G_x(A)$-set structure on $M(A)$, functorially in the test scheme $A$.
\end{remark}

\begin{remark}\label{rem:base_change_variants}
    \Cref{thm:base_change} and \Cref{rem:asch_spec_k} apply analogously to $\AMod(G)$, $\MMod(G)$, and $\UMod(G)$.
\end{remark}

We now define the notion of a $G$-module with coefficients in $\FC$. The \emph{opposite groupoid scheme} $G^\op$ is the groupoid scheme with the same underlying scheme as $G$ but with source and target interchanged and composition reversed.

\begin{definition}\label{def:fmod}
    The category of \emph{formal $G$-modules} is defined as
    \[ \FMod(G) \coloneqq \left(\AMod(G^\op)\right)^\op. \]
    For $M \in \AMod(G^\op)$, the \emph{Cartier dual} $M^\dual$ of $M$ is the corresponding object in $\FMod(G)$.
\end{definition}

When $G = \mathds{1}$ is the terminal groupoid, $\AMod(G^\op) \simeq \AC$. Hence, Cartier duality identifies $\FMod(\mathds{1})$ with $\FC$, and $(-)^\dual$ agrees with classical Cartier duality.

\begin{remark}\label{rem:amod_triv}
    For a morphism $H \to G$ in $\AfGpd$, \Cref{def:asch_res} induces a restriction functor
    \[ \Res_H^G \colon \FMod(G) \to \FMod(H). \]
    In particular, the unique morphism $G \to \mathds{1}$ induces the \emph{trivial module functor}
    \[ \Delta_G \coloneqq \Res_G^{\mathds{1}} \colon \FC \to \FMod(G). \]
\end{remark}

The fibration $\FMod_\AfGpd^\op \to \AfGpd$ is defined by $\FMod_\AfGpd(G) \coloneqq \FMod(G)$, with these restriction functors. We further define the fibration $(\FMod^\alg_X)^\op \to \Open(X)$ as its pullback along $\Pi_1^\alg \colon \Open(X) \to \AfGpd$.

\subsection{The Base Change Theorem}\label{sec:base_change}

Throughout this subsection, let $G$ be a transitive affine groupoid scheme acting on $S$, and let $u \colon T \to S$ be a morphism with $T \neq \emptyset$. We prove that the base change functor
\[ (-)_T \colon \ASch(G) \to \ASch(G_T) \]
is an equivalence of categories. The proof proceeds in six steps using fpqc descent for relative affine schemes and their morphisms \cite[\href{https://stacks.math.columbia.edu/tag/023V}{Tag~023V}, \href{https://stacks.math.columbia.edu/tag/0245}{Tag~0245}, \href{https://stacks.math.columbia.edu/tag/040L}{Tag~040L}]{StacksProject}.

Let $(M, \rho)$ be an object and $\eta$ be a morphism in $\ASch(G_T)$. In Steps 1--3, we use fpqc descent to construct $M^{\natural}$, $\rho^{\natural}$, and $\eta^{\natural}$ over $S$, respectively, thereby defining a functor
\[ (-)^{\natural} \colon \ASch(G_T) \to \ASch(G). \]
In Steps 4--6, we prove that $(-)_T$ is essentially surjective, full, and faithful, establishing an equivalence $\ASch(G) \simeq \ASch(G_T)$.

Let $P \coloneqq T \times_{u,S,s} G$. Then a point $(x, f)$ of $P$ satisfies $u(x) = s(f)$. Define morphisms
\[ \begin{aligned} a \colon P &\to T \\ (x, f) &\mapsto x \end{aligned} \quad\text{and}\quad \begin{aligned} b \colon P &\to S \\ (x, f) &\mapsto t(f). \end{aligned} \]
Since $G$ is transitive, both $a$ and $b$ are fpqc.

\subsubsection*{Step 1}
We construct $M^{\natural}$ over $S$ by descending $a^* M \to P$ along $b \colon P \to S$. Take a point $(p_0, p_1)$ of $P \times_S P$ with $p_i = (x_i, f_i)$. Since $t(f_0) = t(f_1)$, we have an arrow
\[ f_1^{-1} \circ f_0 \colon u(x_0) \to u(x_1) \]
in $G$, hence $(x_0, f_1^{-1} \circ f_0, x_1) \in G_T$. The required descent datum is an isomorphism
\[ \alpha_M \colon \pr_0^* a^* M \xrightarrow{\sim} \pr_1^* a^* M \]
over $P \times_S P$ with projections $\pr_0$ and $\pr_1$. We define it pointwise by
\[ (\alpha_M)_{(p_0,\, p_1)} \coloneqq \rho_{(x_0,\, f_1^{-1} \circ f_0,\, x_1)} \colon M_{x_0} \xrightarrow{\sim} M_{x_1}. \]
The composition rule for $\rho$ immediately gives the cocycle condition for $\alpha_M$. By fpqc descent for affine schemes, $(a^* M, \alpha_M)$ descends to an affine scheme $M^{\natural} \to S$.

\subsubsection*{Step 2}
We construct $\rho^{\natural} \colon s^* M^{\natural} \xrightarrow{\sim} t^* M^{\natural}$ by descending a pullback $q^* \rho$ along an fpqc morphism $r$. Here, $r \colon G_P \to G$ is the middle projection and $q \colon G_P \to G_T$ is defined by
\[ q((x, f), g, (y, h)) = (x,\, h^{-1} \circ g \circ f,\, y), \]
where $G_P \coloneqq P \times_{b,S,s} G \times_{t,S,b} P$. Then $s_P^* a^* M = q^* s_T^* M$ and $t_P^* a^* M = q^* t_T^* M$, so we have an isomorphism
\[ q^* \rho \colon s_P^* a^* M \xrightarrow{\sim} t_P^* a^* M \]
of affine schemes over $G_P$.

We check that $q^* \rho$ is compatible with the effective descent data of $s_P^* a^* M$ and $t_P^* a^* M$. Over $G_P \times_G G_P$, with projections $\pr_0$ and $\pr_1$, the compatibility condition is
\[ (t_P \times t_P)^* \alpha_M \circ \pr_0^*(q^* \rho) = \pr_1^*(q^* \rho) \circ (s_P \times s_P)^* \alpha_M. \]
At a point $(p_0, p_1)$ of $G_P \times_G G_P$ with $p_i = ((x_i, f_i), g, (y_i, h_i))$, this condition becomes
\[ \rho_{(y_0,\, h_1^{-1} \circ h_0,\, y_1)} \circ \rho_{(x_0,\, h_0^{-1} \circ g \circ f_0,\, y_0)} = \rho_{(x_1,\, h_1^{-1} \circ g \circ f_1,\, y_1)} \circ \rho_{(x_0,\, f_1^{-1} \circ f_0,\, x_1)}, \]
and both sides are equal to $\rho_{(x_0,\, h_1^{-1} \circ g \circ f_0,\, y_1)}$. Hence, by fpqc descent for morphisms, $q^* \rho$ descends uniquely to an isomorphism
\[ \rho^{\natural} \colon s^* M^{\natural} \xrightarrow{\sim} t^* M^{\natural}. \]
The two axioms $e^* \rho^{\natural} = \id_{M^{\natural}}$ and $c^* \rho^{\natural} = \pr_0^* \rho^{\natural} \circ \pr_1^* \rho^{\natural}$ may be checked after pullback along suitable fpqc coverings, where they follow from the corresponding axioms for $q^* \rho$. Hence, $(M^{\natural}, \rho^{\natural}) \in \ASch(G)$.

\subsubsection*{Step 3}
We descend $a^* \eta \colon a^* M \to a^* N$ along $b \colon P \to S$. The compatibility condition with the descent data over $P \times_S P$ is
\[ \alpha_N \circ \pr_0^*(a^* \eta) = \pr_1^*(a^* \eta) \circ \alpha_M. \]
At a point $(p_0, p_1)$ of $P \times_S P$ with $p_i = (x_i, f_i)$, this condition becomes
\[ \sigma_{(x_0,\, f_1^{-1} \circ f_0,\, x_1)} \circ \eta_{x_0} = \eta_{x_1} \circ \rho_{(x_0,\, f_1^{-1} \circ f_0,\, x_1)}, \]
which holds since $\eta$ is a morphism in $\ASch(G_T)$. By fpqc descent, $a^* \eta$ descends uniquely to
\[ \eta^{\natural} \colon M^{\natural} \to N^{\natural}. \]
The condition that $\eta^{\natural}$ be a morphism in $\ASch(G)$ may likewise be checked after pullback along $r$, where it follows from the corresponding condition for $\eta$. Hence, we have defined a functor
\[ (-)^{\natural} \colon \ASch(G_T) \to \ASch(G). \]

\begin{remark}\label{rem:descent_iso}
    By construction, every $(M, \rho) \in \ASch(G_T)$ has a descent isomorphism
    \[ \beta_M \colon b^* M^{\natural} \xrightarrow{\sim} a^* M. \]
    These isomorphisms are compatible with actions and morphisms. Concretely, using $r^* s^* M^{\natural} = s_P^* b^* M^{\natural}$ and $r^* t^* M^{\natural} = t_P^* b^* M^{\natural}$, the diagrams
    \[ \begin{tikzcd}
        r^* s^* M^{\natural} \arrow[r, "r^* \rho^{\natural}"] \arrow[d, "s_P^* \beta_M"'] & r^* t^* M^{\natural} \arrow[d, "t_P^* \beta_M"] \\
        s_P^* a^* M \arrow[r, "q^* \rho"'] & t_P^* a^* M
    \end{tikzcd} \quad\text{and}\quad \begin{tikzcd}
        b^* M^{\natural} \arrow[r, "b^* \eta^{\natural}"] \arrow[d, "\beta_M"'] & b^* N^{\natural} \arrow[d, "\beta_N"] \\
        a^* M \arrow[r, "a^* \eta"'] & a^* N
    \end{tikzcd} \]
    commute. Equivalently,
    \begin{align}
        t_P^* \beta_M \circ r^* \rho^{\natural} &= q^* \rho \circ s_P^* \beta_M, \quad\text{and} \label{eq:descent_action_compatibility} \\
        \beta_N \circ b^* \eta^{\natural} &= a^* \eta \circ \beta_M. \label{eq:descent_morphism_compatibility}
    \end{align}
\end{remark}

\subsubsection*{Step 4}
We show that $(-)_T$ is essentially surjective. Given $(M, \rho) \in \ASch(G_T)$, it suffices to construct an isomorphism $(M^{\natural}, \rho^{\natural})_T \xrightarrow{\sim} (M, \rho)$, where $(M^{\natural}, \rho^{\natural})_T \coloneqq (u^* M^{\natural}, u^* \rho^{\natural})$. Consider the section $d$ of $a$ given by
\begin{align*}
    d \colon T &\to P \\
    x &\mapsto (x, e_{u(x)}).
\end{align*}
Since $a \circ d = \id_T$ and $b \circ d = u$, pulling back $\beta_M$ along $d$ yields an isomorphism
\[ \mu_M \coloneqq d^* \beta_M \colon u^* M^{\natural} \xrightarrow{\sim} M \]
of affine schemes over $T$.

It remains to check that $\mu_M$ respects the actions $u^* \rho^{\natural}$ and $\rho$. For an arrow $(x, g, y)$ of $G_T$, consider the arrow
\[ p \coloneqq \bigl((x, e_{u(x)}), g, (y, e_{u(y)})\bigr) \]
of $G_P$. The identity \eqref{eq:descent_action_compatibility} gives
\begin{align*}
    (t_P^* \beta_M)_p \circ (r^* \rho^{\natural})_p &= (q^* \rho)_p \circ (s_P^* \beta_M)_p \\
    (\beta_M)_{(y,\, e_{u(y)})} \circ \rho^{\natural}_g &= \rho_{(x,\, g,\, y)} \circ (\beta_M)_{(x,\, e_{u(x)})} \\
    (\mu_M)_y \circ (u^* \rho^{\natural})_{(x,\, g,\, y)} &= \rho_{(x,\, g,\, y)} \circ (\mu_M)_x.
\end{align*}
Hence, $\mu_M \colon (M^{\natural}, \rho^{\natural})_T \xrightarrow{\sim} (M, \rho)$ is an isomorphism in $\ASch(G_T)$, and $(-)_T$ is essentially surjective.

\subsubsection*{Step 5}
We show that $(-)_T$ is full. Given $(M, \rho), (N, \sigma) \in \ASch(G)$ and a morphism $\eta \colon (M, \rho)_T \to (N, \sigma)_T$ in $\ASch(G_T)$, it suffices to find a morphism $\psi$ in $\ASch(G)$ such that $\eta = \psi_T$. We first construct a canonical isomorphism
\[ \nu_M \colon (M, \rho)_T^{\natural} \xrightarrow{\sim} (M, \rho) \]
in $\ASch(G)$. Define
\[ \lambda_M \colon a^* u^* M \xrightarrow{\sim} b^* M \]
pointwise by
\[ (\lambda_M)_{(x,\, f)} \coloneqq \rho_f \colon M_{u(x)} \xrightarrow{\sim} M_{b(x,f)}. \]
Then $\lambda_M$ is an isomorphism in $\ASch(G_P)$. Indeed, for every arrow $p \coloneqq \bigl((x, f), g, (y, h)\bigr)$ of $G_P$, we have
\begin{align*}
    (\lambda_M)_{(y,h)} \circ (a^* u^* \rho)_p &= (b^* \rho)_p \circ (\lambda_M)_{(x,f)} \\
    \rho_h \circ \rho_{h^{-1} \circ g \circ f} &= \rho_g \circ \rho_f.
\end{align*}
The compatibility condition with the descent data over $P \times_S P$ is
\[ \pr_1^*(\lambda_M) \circ \alpha_{u^* M} = \pr_0^*(\lambda_M). \]
At a point $(p_0, p_1)$ of $P \times_S P$ with $p_i = (x_i, f_i)$, this condition becomes
\[ \rho_{f_1} \circ \rho_{f_1^{-1} \circ f_0} = \rho_{f_0}. \]
Hence, fpqc descent yields the isomorphism $\nu_M$ in $\ASch(G)$. By construction, $(\nu_M)_T = \mu_{M_T} = d^*\beta_{M_T}$.

Now define
\[ \psi \coloneqq \nu_N \circ \eta^{\natural} \circ \nu_M^{-1} \colon (M, \rho) \to (N, \sigma). \]
Pulling back \eqref{eq:descent_morphism_compatibility} along $d$ gives
\begin{align*}
    \beta_{N_T} \circ b^* \eta^\natural &= a^* \eta \circ \beta_{M_T} \\
    \mu_{N_T} \circ (\eta^\natural)_T &= \eta \circ \mu_{M_T}.
\end{align*}
It follows that
\[ \psi_T = (\nu_N)_T \circ (\eta^\natural)_T \circ (\nu_M)_T^{-1} = \mu_{N_T} \circ (\eta^\natural)_T \circ \mu_{M_T}^{-1} = \eta, \]
so $(-)_T$ is full.

\subsubsection*{Step 6}
We show that $(-)_T$ is faithful. Given morphisms $\eta, \psi \colon (M, \rho) \to (N, \sigma)$ in $\ASch(G)$ with $\eta_T = \psi_T$, it suffices to prove that $\eta = \psi$. For a point $p \coloneqq (x, f)$ of $P$, we have $f \colon u(x) \to b(p)$ in $G$. The equivariance of $\eta$ and $\psi$, together with $\eta_T = \psi_T$, gives
\[ \eta_{b(p)} \circ \rho_f = \sigma_f \circ \eta_{u(x)} = \sigma_f \circ \psi_{u(x)} = \psi_{b(p)} \circ \rho_f. \]
Since $\rho_f$ is an isomorphism, we obtain $\eta_{b(p)} = \psi_{b(p)}$, so $b^* \eta = b^* \psi$. As $b$ is an fpqc cover, $\eta = \psi$, proving that $(-)_T$ is faithful.

We have thus shown that
\[ (-)_T : \ASch(G) \rightleftarrows \ASch(G_T) : (-)^\natural \]
is an equivalence of categories.

\subsection{Decomposition of Group Modules}\label{sec:decomposition_group}

For the remainder of this section, we assume that $k$ is algebraically closed of characteristic $0$. In this subsection, let $G$ be an affine group scheme over $k$. We will show that the functor
\begin{align*}
    \MMod(G) \times \UMod(G) &\to \AMod(G) \\
    (M_\mt, M_\un) &\mapsto M_\mt \times M_\un
\end{align*}
is an equivalence of categories. Thus, every affine $G$-module decomposes into a multiplicative-type part and a pro-unipotent part.

Let $\widehat{G} \coloneqq \varprojlim G/N$ denote the profinite completion of $G$, where the limit ranges over all finite quotient group schemes $G/N$ of $G$. Since $k$ is algebraically closed of characteristic $0$, these group schemes are constant, and we regard $\widehat{G}$ as the corresponding profinite group. We will also establish the corresponding equivalence
\[ \Mod(\widehat{G}) \times \Rep(G) \xrightarrow{\sim} \FMod(G). \]

In the present setting, the notion of affine $G$-module from \Cref{def:amod} can be described equivalently as follows. It is a pair $(M, \alpha)$, where $M$ is a commutative affine group scheme over $k$ and
\[ \alpha \colon G \times M \to M \]
is a morphism that makes $M(R)$ into a left $G(R)$-module for every $k$-algebra $R$.

For a $k$-algebra $A$, we write $g \cdot m \coloneqq \alpha(g, m)$ for $g \in G(A)$ and $m \in M(A)$. Under the description in \Cref{def:asch}, the corresponding groupoid action $\rho \colon G \times M \to G \times M$ is given by $(g, m) \mapsto (g, g \cdot m)$. Given an abelian group $D$ and a $k$-algebra $R$, we denote by $R(D)$ the group ring of $D$ over $R$.

\begin{definition}\label{def:vdiag}
    Let $D$ be an abelian group and $V$ a $k$-vector space. We write
    \[ \diag D \coloneqq \Spec(k(D)) \quad\text{and}\quad \vect V \coloneqq \Spec(\Sym V) \]
    for the \emph{diagonalizable group scheme} associated to $D$ and the \emph{vector group scheme} associated to $V$, respectively.
\end{definition}

Since $\charac k = 0$, every commutative affine group scheme $M$ over $k$ admits a canonical decomposition
\[ M \cong M_\mt \times M_\un, \]
where $M_\mt$ is of multiplicative type and $M_\un$ is pro-unipotent \cite[Ch.~IV, \S3, 1.1]{DemazureGabriel}. Since $k$ is algebraically closed, $M_\mt$ and $M_\un$ are respectively of the form
\[ M_\mt \cong \diag D \quad\text{and}\quad M_\un \cong \vect V \]
for some abelian group $D$ \cite[Ch.~IV, \S1, 3.2]{DemazureGabriel} and $k$-vector space $V$ \cite[Ch.~IV, \S2, 4.2]{DemazureGabriel}.

We show that any $G$-module structure on $M$ restricts to $G$-module structures on $M_\mt$ and $M_\un$. Under the identifications above, these structures correspond to $G^\op$-actions on $D$ and $V$, respectively. The action on $D$ will be shown below to correspond to a continuous $\widehat{G}^\op$-action. For a $k$-algebra $R$, let $\AC(R)$ denote the category of commutative affine group schemes over $R$.

\begin{lemma}\label{lem:orth_un_mt}
    With the notation above,
    \[ \Hom_{\AC(R)}((M_\mt)_R, (M_\un)_R) = 0. \]
\end{lemma}
\begin{proof}
    A homomorphism $(M_\mt)_R \to (M_\un)_R$ corresponds contravariantly to a morphism
    \[ f \colon \Sym V_R \to R(D) \]
    of Hopf algebras. The comultiplication of $v \in V_R$ is $\mu(v) = v \otimes 1 + 1 \otimes v$. Write
    \[ f(v) = \sum_{d \in D} a_d \cdot d \]
    with $a_d \in R$. Since $\mu(d) = d \otimes d$ in $R(D)$ and $f$ preserves comultiplication, we obtain
    \[ \sum_{d \in D} a_d \cdot (d \otimes d) = \sum_{d \in D} a_d \cdot (d \otimes 1 + 1 \otimes d). \]
    Comparing coefficients gives $a_d = 0$ for $d \neq 0_D$ and $a_{0_D} = 2a_{0_D}$, hence $a_{0_D} = 0$. Thus, $f(v) = 0$ for every $v \in V_R$, so $f$ is the trivial morphism.
\end{proof}

\begin{lemma}\label{lem:orth_mt_un}
    Let $R$ be a reduced $k$-algebra. With the notation above,
    \[ \Hom_{\AC(R)}((M_\un)_R, (M_\mt)_R) = 0. \]
\end{lemma}
\begin{proof}
    A homomorphism $(M_\un)_R \to (M_\mt)_R$ corresponds contravariantly to a morphism
    \[ g \colon R(D) \to \Sym V_R \]
    of Hopf algebras. Every $d \in D$ is a unit in $R(D)$, so $g(d)$ is a unit in $\Sym V_R$. Since $R$ is reduced, $(\Sym V_R)^\times = R^\times$; hence $g(d) \in R^\times$. Writing $\varepsilon$ for the counits gives
    \[ g(d) = \varepsilon(g(d)) = \varepsilon(d) = 1. \]
    Thus, $g$ is the trivial morphism.
\end{proof}

Since $\charac k = 0$, every affine group scheme over $k$ is reduced \cite[Proposition~4.2.8]{Perrin}. Let $\RAlg$ denote the category of reduced $k$-algebras, and let
\begin{align*}
    \underline{\Aut}(M) \colon \RAlg &\to \mathbf{Grp} \\
    R &\mapsto \Aut_R(M_R).
\end{align*}
A $G$-module structure on $M$ is then equivalent to a homomorphism
\[ G \to \underline{\Aut}(M) \]
of group functors on $\RAlg$.

\begin{theorem}\label{thm:decomposition_group}
    The functor
    \begin{align*}
        \MMod(G) \times \UMod(G) &\to \AMod(G) \\
        (M_\mt, M_\un) &\mapsto M_\mt \times M_\un
    \end{align*}
    is an equivalence of categories.
\end{theorem}
\begin{proof}
    By \Cref{lem:orth_un_mt,lem:orth_mt_un} with $R = k$, the only morphisms between groups of multiplicative type and pro-unipotent groups in either direction are the trivial morphisms. Thus, the functor is fully faithful. For essential surjectivity, let $M \in \AMod(G)$ and write its underlying group scheme as
    \[ M \cong M_\mt \times M_\un. \]
    Applying \Cref{lem:orth_un_mt,lem:orth_mt_un} over every reduced $k$-algebra gives an isomorphism
    \[ \underline{\Aut}(M) \cong \underline{\Aut}(M_\mt) \times \underline{\Aut}(M_\un) \]
    of group functors on $\RAlg$. Under this isomorphism, the homomorphism $G \to \underline{\Aut}(M)$ corresponds to $G$-actions on $M_\mt$ and $M_\un$. Thus, $M$ lies in the essential image.
\end{proof}

Given a set (resp.\ group) $A$, denote by $\underline{A} \coloneqq \coprod_A \Spec k$ the associated constant $k$-scheme (resp.\ $k$-group scheme).

\begin{lemma}\label{lem:mod_mmod}
    The assignment $D \mapsto \diag D$ defines an equivalence
    \[ \Mod(\widehat{G}) \xrightarrow{\sim} \MMod(G^\op)^\op. \]
\end{lemma}
\begin{proof}
    By \cite[Expos\'e~VIII, Corollaire~1.4]{SGA3}, we have an isomorphism
    \[ \underline{\Aut}(\diag D)^\op \cong \underline{\Aut}(\underline{D}) \]
    of group functors. Therefore, a $G^\op$-module structure on $\diag D$ corresponds to a $G$-module structure $\rho \colon G \times \underline{D} \to \underline{D}$.

    For $d \in D$, let $\rho_d \colon G \times \underline{\{d\}} \to \underline{D}$ denote the restriction of $\rho$. Since $G$ is quasi-compact, the image of $\rho_d$ is $\underline{G \cdot d}$, where $G \cdot d \subset D$ is finite. Then the induced action on $\underline{G \cdot d}$ gives a homomorphism
    \[ G \to \underline{\Sym(G \cdot d)}, \]
    where $\Sym(G \cdot d)$ denotes the permutation group of $G \cdot d$.
    Since $\Sym(G \cdot d)$ is finite, this induces
    \[ \widehat{G} \to \Sym(G \cdot d). \]
    Gluing the resulting maps $\widehat\rho_d \colon \widehat{G} \times \{d\} \to G \cdot d$ yields a $\widehat{G}$-module structure $\widehat\rho \colon \widehat{G} \times D \to D$. The two constructions are easily seen to be mutually inverse.
\end{proof}

\begin{lemma}\label{lem:rep_umod}
    The assignment $V \mapsto \vect V$ defines an equivalence
    \[ \Rep(G) \xrightarrow{\sim} \UMod(G^\op)^\op. \]
\end{lemma}
\begin{proof}
    We have an isomorphism
    \[ \underline{\Aut}(\vect V)^\op \cong \mathbf{GL}(V) \]
    of group functors, where $\mathbf{GL}(V)$ denotes the functor $R \mapsto \Aut_R(V_R)$. Hence, $G^\op$-module structures on $\vect V$ correspond to $G$-representation structures on $V$, proving the result.
\end{proof}

\Cref{lem:mod_mmod,lem:rep_umod} prove the following theorem.

\begin{theorem}\label{thm:decomposition_group_dual}
    The natural functor
    \begin{align*}
        \Mod(\widehat{G}) \times \Rep(G) &\to \FMod(G) \\
        (D, V) &\mapsto (\diag D)^\dual \times (\vect V)^\dual
    \end{align*}
    is an equivalence of categories.
\end{theorem}

\subsection{Decomposition of Groupoid Modules}\label{sec:decomposition_groupoid}

Let $G$ be a transitive affine groupoid scheme acting on a scheme $S$ of finite type over $k$, although the transitivity assumption will be dropped later. Since $k$ is algebraically closed, we may fix a $k$-rational point $x \in S(k)$. The goal of this subsection is to obtain a natural equivalence
\[ \FMod(G) \simeq \Mod(\widehat{G}) \times \Rep(G) \]
as in \Cref{thm:decomposition_group_dual}. Here, naturality means independence of the choice of $x$. For example,
\begin{align*}
    \MMod(G) \times \UMod(G) &\to \AMod(G) \\
    (M_\mt, M_\un) &\mapsto M_\mt \times M_\un
\end{align*}
is a natural functor, since it is defined without reference to $x$.
The chosen point $x$, however, allows us to apply \Cref{thm:decomposition_group,thm:base_change} to show that this functor is an equivalence of categories. It now suffices to define the profinite completion $\widehat{G}$ of $G$ and, by analogy with \Cref{lem:mod_mmod,lem:rep_umod}, establish natural equivalences
\[ \MMod(G) \simeq \Mod(\widehat{G}^\op)^\op \quad\text{and}\quad \UMod(G) \simeq \Rep(G^\op)^\op. \]

Throughout this subsection, set $P \coloneqq s^{-1}(x) \subseteq G$, and let
\[ a \coloneqq s|_P \colon P \to \{x\} \quad\text{and}\quad b \coloneqq t|_P \colon P \to S. \]
Then $b$ is fpqc, since it is obtained from $(s, t)$ by base change. The following lemma is an analogue of the fact that $\lambda_M$ in \Cref{sec:base_change} is an isomorphism.

\begin{lemma}\label{lem:trivial_pullback}
    Let $\mathcal{F}$ be an fpqc sheaf on $S$ equipped with an isomorphism
    \[ \rho \colon s^* \mathcal{F} \xrightarrow{\sim} t^* \mathcal{F}. \]
    Then $b^* \mathcal{F}$ is constant over $P$. More precisely,
    \[ b^* \mathcal{F} \cong a^* \mathcal{F}_x. \]
\end{lemma}
\begin{proof}
    Restricting $\rho$ to $P$ gives the required isomorphism $a^* \mathcal{F}_x \xrightarrow{\sim} b^* \mathcal{F}$.
\end{proof}

Given abelian fpqc sheaves $\mathcal{F}$ and $\mathcal{G}$ over $S$, we denote their sheaf hom by $\underline{\Hom}_{\Ab(S)}(\mathcal{F}, \mathcal{G})$, which is again an abelian fpqc sheaf. All fpqc sheaves are taken on the big fpqc site over $S$. The subscript $\Ab(S)$ suggests the category of abelian fpqc sheaves on $S$, but the category itself is not needed, so we need not worry about the set-theoretic issues it may raise. Abelian fpqc sheaves on $S$ may be represented by geometric objects, such as commutative group schemes or quasi-coherent sheaves on $S$. We freely identify such geometric objects with the corresponding abelian fpqc sheaves below.

\begin{lemma}\label{lem:umod_rep}
    We have an equivalence of categories
    \[ \UMod(G) \simeq \Rep(G^\op)^\op, \]
    where both directions are given by $\underline{\Hom}_{\Ab(S)}(-, \mathbb{G}_{a,S})$.
\end{lemma}
\begin{proof}
    In this proof, we write $(-)^\vee \coloneqq \underline{\Hom}_{\Ab(S)}(-, \mathbb{G}_{a,S})$. Let $M \in \UMod(G)$. Then $M_x \cong \vect V$ for some $k$-vector space $V$, so $b^* M \cong P \times \vect V$. Moreover,
    \[ b^*(M^\vee) \cong (b^* M)^\vee \cong \underline{\Hom}_{\Ab(P)}(P \times \vect V, \mathbb{G}_{a,P}) \cong \mathcal{O}_P \otimes_k V. \]
    Since quasi-coherence is fpqc local, $M^\vee$ is a quasi-coherent sheaf on $S$. The $G^\op$-action on $M^\vee$ is naturally induced by the $G$-action on $M$, so $M^\vee \in \Rep(G^\op)^\op$.

    Conversely, let $M \in \Rep(G^\op)^\op$. Then $M_x \cong \mathcal{O}_{\Spec k} \otimes_k V$ for some $k$-vector space $V$, so $b^* M \cong \mathcal{O}_P \otimes_k V$. Moreover,
    \[ b^*(M^\vee) \cong (b^* M)^\vee \cong \underline{\Hom}_{\Ab(P)}(\mathcal{O}_P \otimes_k V, \mathbb{G}_{a,P}) \cong P \times \vect V. \]
    Thus, $M^\vee$ is represented by an affine group scheme over $S$ by fpqc descent. The $G$-action is naturally induced by the $G^\op$-action on $M$, and its fibers are commutative pro-unipotent group schemes. Hence, $M^\vee \in \UMod(G)$.

    The functor $\underline{\Hom}_{\Ab(P)}(-, \mathbb{G}_{a,P})$ gives a duality between $P \times \vect V$ and $\mathcal{O}_P \otimes_k V$. The desired duality over $S$ follows by fpqc descent.
\end{proof}

\begin{definition}\label{def:fset}
    An affine $G$-scheme $M \in \ASch(G)$ is a \emph{finite $G$-set} if its geometric fiber over every geometric point of $S$ is a finite constant scheme. We denote by $\FSet(G)$ the full subcategory of $\ASch(G)$ spanned by finite $G$-sets. We set
    \[ \Set(G) \coloneqq \Ind(\FSet(G)) \quad\text{and}\quad \Mod(G) \coloneqq \AbObj(\Set(G)). \]
\end{definition}

\Cref{thm:base_change} and \Cref{rem:asch_spec_k} apply analogously to $\FSet(G)$, $\Set(G)$, and $\Mod(G)$. In particular, the functors
\[ \FSet(G) \xrightarrow{\sim} \FSet(G_x) \xrightarrow{\sim} \FSet(\widehat{G_x}) \]
are equivalences. The category $\FSet(\widehat{G_x})$ is a Galois category, so the same holds for $\FSet(G)$. Every geometric point of $S$ provides a fiber functor on $\FSet(G)$.

\begin{definition}\label{def:profinite_completion}
    The \emph{profinite completion} $\widehat{G}$ of $G$ is the transitive profinite groupoid associated with the Galois category $\FSet(G)$. Its objects are geometric points of $S$, and its morphisms are isomorphisms between the corresponding fiber functors on $\FSet(G)$.
\end{definition}

By construction, the vertex group $\widehat{G}_x$ is naturally isomorphic to the profinite completion of $G_x$. Moreover, there are natural equivalences
\[ \FSet(\widehat{G}) \simeq \FSet(G), \quad \Set(\widehat{G}) \simeq \Set(G), \quad\text{and}\quad \Mod(\widehat{G}) \simeq \Mod(G). \]
We henceforth identify $\Mod(\widehat{G})$ with $\Mod(G)$ through this equivalence.

\begin{lemma}\label{lem:mmod_mod}
    We have an equivalence of categories
    \[ \MMod(G) \simeq \Mod(\widehat{G}^\op)^\op, \]
    where both directions are given by $\underline{\Hom}_{\Ab(S)}(-, \mathbb{G}_{m,S})$.
\end{lemma}
\begin{proof}
    In this proof, we write $(-)^\vee \coloneqq \underline{\Hom}_{\Ab(S)}(-, \mathbb{G}_{m,S})$. Let $M \in \MMod(G)$. Then $M_x \cong \diag D$ for some abelian group $D$, so $b^* M \cong P \times \diag D$. Moreover,
    \[ b^*(M^\vee) \cong (b^* M)^\vee \cong \underline{\Hom}_{\Ab(P)}(P \times \diag D, \mathbb{G}_{m,P}) \cong P \times D. \]
    The $G$-action on $M$ naturally induces a $G^\op$-action on $M^\vee$. By \Cref{lem:mod_mmod}, the object $M_x \cong \diag D$ corresponds to a $\widehat{G_x}^\op$-module $D$. Since $\widehat{G_x} \cong \widehat{G}_x$, this yields $M^\vee \in \Mod(\widehat{G}^\op)^\op$.

    Conversely, let $M \in \Mod(\widehat{G}^\op)^\op$. Then $M_x \cong \Spec k \times D$ for some abelian group $D$, so $b^* M \cong P \times D$. Moreover,
    \[ b^*(M^\vee) \cong (b^* M)^\vee \cong \underline{\Hom}_{\Ab(P)}(P \times D, \mathbb{G}_{m,P}) \cong P \times \diag D. \]
    Thus, $M^\vee$ is represented by an affine group scheme over $S$ by fpqc descent. The $G$-action is naturally induced by the $G^\op$-action on $M$, and its fibers are group schemes of multiplicative type. Hence, $M^\vee \in \MMod(G)$.

    The functor $\underline{\Hom}_{\Ab(P)}(-, \mathbb{G}_{m,P})$ gives a duality between $P \times \diag D$ and $P \times D$. The desired duality over $S$ follows by fpqc descent.
\end{proof}

\begin{remark}
    Except for \Cref{lem:trivial_pullback}, every lemma in this subsection remains valid without the transitivity assumption, since its proof applies componentwise. The natural equivalence
    \[ \MMod(G) \times \UMod(G) \xrightarrow{\sim} \AMod(G) \]
    also extends componentwise. For $G = \coprod_{C \in \pi_0(G)} C$, we set $\widehat{G} \coloneqq \coprod_{C \in \pi_0(G)} \widehat{C}$.
\end{remark}

Applying \Cref{lem:umod_rep,lem:mmod_mod} to this equivalence, we obtain the following theorem, where $(-)^\dual$ denotes the Cartier dual as in \Cref{def:fmod}.

\begin{theorem}\label{thm:decomposition_groupoid}
    Let $G$ be an affine groupoid scheme acting on a scheme $S$ of finite type over an algebraically closed field $k$ of characteristic $0$. Then the natural functor
    \begin{align*}
        \Xi_G \colon \Mod(\widehat{G}) \times \Rep(G) &\to \FMod(G) \\
        (D, V) &\mapsto \underline{\Hom}_{\Ab(S)}(D, \mathbb{G}_{m,S})^\dual \times \underline{\Hom}_{\Ab(S)}(V, \mathbb{G}_{a,S})^\dual
    \end{align*}
    is an equivalence of categories.
\end{theorem}

\subsection{Decomposition of Cohomology Groups}\label{sec:decomposition_cohomology}

Finally, we construct unified cohomology, which combines \'etale and algebraic de Rham cohomology. Let $X$ be a smooth connected variety over $k$, although the connectedness assumption will be dropped later. We establish an equivalence
\[ \Sh(X, \FMod^\alg_X) \simeq \Sh(X, \Mod^\et_X) \times \Sh(X, \Rep^\alg_X). \]
Together with the \'etale and algebraic de Rham comparison theorems, this equivalence yields the corresponding decomposition of $H^i_\uni(X, -)$ into $H^i(X_\et, -)$ and $H^i_\dR(X, -)$.

Let $\FtGpd$ denote the full subcategory of $\AfGpd$ consisting of affine groupoid schemes acting on schemes of finite type over $k$. Let $\Rep_\FtGpd$ and $\FMod_\FtGpd$ denote the restrictions of $\Rep_\AfGpd$ and $\FMod_\AfGpd$ to $\FtGpd$, respectively. Define the fibration $\Mod_\FtGpd^\op \to \FtGpd$ as the pullback of $\Mod_\PfGpd^\op \to \PfGpd$ along the profinite completion functor
\[ (\widehat{-}) \colon \FtGpd \to \PfGpd. \]
For every $G \in \FtGpd$, let $\Xi_G$ denote the equivalence of \Cref{thm:decomposition_groupoid}.

\begin{lemma}\label{lem:decomposition_FtGpd}
    The fiberwise assignments $(D, V) \mapsto \Xi_G(D, V)$ for $G \in \FtGpd$ define an equivalence
    \[ \Xi \colon \Mod_\FtGpd \times_\FtGpd \Rep_\FtGpd \xrightarrow{\sim} \FMod_\FtGpd \]
    of fibered categories over $\FtGpd$.
\end{lemma}
\begin{proof}
    By construction, $\Xi_G$ is compatible with the restriction maps along every morphism $H \to G$ in $\FtGpd$. Thus, $\Xi$ is well-defined, and it is fiberwise an equivalence by \Cref{thm:decomposition_groupoid}.
\end{proof}

\begin{corollary}\label{cor:FModFtGpd_SC5}
    The fibration $\FMod_\FtGpd^\op \to \FtGpd$ is an SC5 fibration.
\end{corollary}
\begin{proof}
    This follows from \Cref{prop:ModPfGpd_SC5,prop:RepAfGpd_SC5} via the equivalence $\Xi$ of \Cref{lem:decomposition_FtGpd}.
\end{proof}

We adjoin a terminal space $* \coloneqq \Spec k$ to $\Open(X)$ to form $\Open(X/*)$. The functor $\Pi_1^\alg \colon \Open(X/*) \to \AfGpd$ takes values in $\FtGpd$. Hence, \Cref{cor:FModFtGpd_SC5} shows that $\FMod_{X/*}^\alg$ is an SC5 category, so the following cohomology is well-defined.

\begin{definition}\label{def:uni_cohomology}
    We define the \emph{$i$-th unified cohomology functor} by
    \[ H^i_\uni(X, -) \coloneqq R^i\!\left(\Gamma(X, -)^{\Pi_1^\alg(X)}\right) \colon \Sh(X, \FMod^\alg_X) \to \FC. \]
\end{definition}

We now show that the sheaf category $\Sh(X, \FMod^\alg_X)$ decomposes. First, we identify the profinite completion of $\Pi_1^\alg(X)$ with $\Pi_1^\et(X)$. Every object $M \in \FSet(\Pi_1^\alg(X))$ has an underlying finite \'etale $X$-scheme, by \Cref{lem:trivial_pullback} and \cite[\href{https://stacks.math.columbia.edu/tag/02LA}{Tag~02LA},~\href{https://stacks.math.columbia.edu/tag/02VN}{Tag~02VN}]{StacksProject}. Via the standard equivalence $\EtCov(X) \simeq \FSet(\Pi_1^\et(X))$ \cite[Exp.~V, \S7]{SGA1}, this construction defines an exact functor
\[ \FSet(\Pi_1^\alg(X)) \to \FSet(\Pi_1^\et(X)) \]
compatible with the fiber functors at every geometric point of $X$. We will show that this functor is an equivalence. This implies that the induced morphism
\[ \Pi_1^\et(X) \to \widehat{\Pi_1^\alg(X)} \]
is an isomorphism. The proof requires an intermediate groupoid scheme.

\begin{definition}[cf.\ {\cite[Lemma~3.1]{Nori}}]\label{def:weil_finite}
    Let $\mathcal{T}$ be a Tannakian category. An object $M \in \mathcal{T}$ is \emph{finite} if $f(M) \cong g(M)$ for distinct polynomials $f, g \in \mathbb{N}[T]$, where addition and multiplication are interpreted as $\oplus$ and $\otimes$, respectively.
\end{definition}

By \cite[Corollary~7.10]{BorneVistoli}, the finite objects of $\mathcal{T}$ form a Tannakian subcategory when the base field has characteristic $0$. Moreover, the Tannakian subcategory generated by any finite object has a finite Tannaka group \cite[Proposition~2.20(a)]{DeligneMilne}.

\begin{definition}[Esnault--Hai, {\cite[Definition~2.4]{EsnaultHai2}}]\label{def:fconn}
    We denote by $\Pi_1^\fin(X)$ the Tannaka groupoid of
    \[ \FConn(X) \coloneqq \{ \text{vector bundles with finite integrable connections on } X \} \]
    with respect to the forgetful fiber functor over $X$. For a rational point $x \in X(k)$, we write $\pi_1^\fin(X, x) \coloneqq \Pi_1^\fin(X)_x$.
\end{definition}

Every finite connection becomes trivial after a finite \'etale cover. For a smooth curve, it is therefore regular at infinity by \cite[Proposition~1.13~(iii)]{Deligne4}. For arbitrary $X$, its restriction to every smooth curve is regular at infinity, so \cite[Proposition~4.4~(ii)]{Deligne4} and the remark following \cite[D\'efinition~4.5]{Deligne4} show that it is regular at infinity. Hence, $\FConn(X)$ is a full subcategory of $\MIC^\rs(X)$. The inclusion induces a morphism of groupoid schemes
\[ \Pi_1^\alg(X) \to \Pi_1^\fin(X). \]
Since $\FConn(X)$ is the full subcategory of finite objects in $\MIC^\rs(X)$, the induced map $\pi_1^\alg(X, x) \to \pi_1^\fin(X, x)$ is the profinite completion for every closed point $x \in X$. Consequently, the natural functor
\[ \FSet(\Pi_1^\fin(X)) \to \FSet(\Pi_1^\alg(X)) \]
is an equivalence.

\begin{lemma}\label{lem:galois_equiv}
    Let $u\colon \mathcal{C} \to \mathcal{D}$ be an exact functor of Galois categories, and let $F_{\mathcal{C}}$ and $F_{\mathcal{D}}$ be fiber functors with corresponding fundamental pro-objects $P_{\mathcal{C}}$ and $P_{\mathcal{D}}$, respectively. Denote by $\widehat{u}\colon \Pro(\mathcal{C}) \to \Pro(\mathcal{D})$ the functor induced by $u$. Suppose that
    \[ F_{\mathcal{C}} \cong F_{\mathcal{D}} \circ u \quad\text{and}\quad \widehat{u}(P_{\mathcal{C}}) \cong P_{\mathcal{D}}. \]
    Then $u$ is an equivalence.
\end{lemma}
\begin{proof}
    The two hypotheses yield isomorphisms
    \[ \Aut_{\Pro(\mathcal{C})}(P_{\mathcal{C}}) \cong \widehat{F}_{\mathcal{C}}(P_{\mathcal{C}}) \cong \widehat{F}_{\mathcal{D}}(\widehat{u}(P_{\mathcal{C}})) \cong \widehat{F}_{\mathcal{D}}(P_{\mathcal{D}}) \cong \Aut_{\Pro(\mathcal{D})}(P_{\mathcal{D}}) \]
    of profinite sets, and hence of profinite groups. Thus, $u$ induces an isomorphism of fundamental groups and is therefore an equivalence.
\end{proof}

\begin{proposition}\label{prop:profinite_completion}
    The natural morphism of profinite groupoids
    \[ \Pi_1^\et(X) \to \widehat{\Pi_1^\alg(X)} \]
    is an isomorphism.
\end{proposition}
\begin{proof}
    Choose a point $x \in X(k)$. The fundamental pro-object of $\FSet(\Pi_1^\fin(X))$ is
    \[ P \coloneqq \Spec k \times_{x, X, s} \Pi_1^\fin(X). \]
    By \cite[Lemma~4.3]{EsnaultHai2}, $P$ is the universal profinite \'etale cover of $X$ based at $x$. Hence, \Cref{lem:galois_equiv} shows that the composite
    \[ \FSet(\Pi_1^\fin(X)) \longrightarrow \FSet(\Pi_1^\alg(X)) \longrightarrow \FSet(\Pi_1^\et(X)) \]
    is an equivalence. Since the first functor is an equivalence, so is the second.
\end{proof}

\begin{remark}
    The connectedness assumption can now be dropped. For a smooth variety $X = \coprod_{C \in \pi_0(X)} C$, apply \Cref{prop:profinite_completion} to each connected component $C$. Since the fundamental groupoids and profinite completion decompose componentwise, taking finite coproducts gives
    \[ \Pi_1^\et(X) \xrightarrow{\sim} \widehat{\Pi_1^\alg(X)}. \]
    Thus, \Cref{prop:profinite_completion} and the subsequent equivalences hold for arbitrary smooth varieties.
\end{remark}

\Cref{prop:profinite_completion} gives $\Pi_1^\et(U) \cong \widehat{\Pi_1^\alg(U)}$, naturally in $U \in \Open(X)$. Together with \Cref{lem:decomposition_FtGpd}, these isomorphisms yield the following corollary.

\begin{corollary}\label{cor:decomposition_X}
    The fiberwise assignments $(D, V) \mapsto \Xi_{\Pi_1^\alg(U)}(D, V)$ for $U \in \Open(X)$ define an equivalence
    \[ \Xi_X \colon \Mod^\et_X \times_{\Open(X)} \Rep^\alg_X \xrightarrow{\sim} \FMod^\alg_X \]
    of fibered categories over $\Open(X)$.
\end{corollary}

Consequently, the equivalences $\Xi_U$ for $U \in \Open(X)$ induce an equivalence
\begin{align*}
    \Xi_{\Sh(X)} \colon \Sh(X, \Mod^\et_X) \times \Sh(X, \Rep^\alg_X) &\xrightarrow{\sim} \Sh(X, \FMod^\alg_X) \\
    (\mathcal{D}, \mathcal{V}) &\mapsto \left(U \mapsto \Xi_U(\mathcal{D}(U), \mathcal{V}(U))\right).
\end{align*}

\begin{definition}\label{def:VD}
    Let $\mathbb{G}_a^\vee$ denote the formal additive group and $\underline{\mathbb{Z}}$ the constant formal group associated to $\mathbb{Z}$. For $M \in \FC$, define its discrete and vector space parts by
    \[ \mathrm{D}(M) \coloneqq \Hom_\FC(\underline{\mathbb{Z}}, M) \quad\text{and}\quad \mathrm{V}(M) \coloneqq \Hom_\FC(\mathbb{G}_a^\vee, M). \]
\end{definition}

For the terminal groupoid $\mathds{1}$, the equivalence $\Xi_\mathds{1}$ and its inverse are given by
\[
\begin{aligned}
    \Xi_\mathds{1} \colon \Ab \times \Vect(k) &\to \FC \\
    (D, V) &\mapsto (\diag D)^\dual \times (\vect V)^\dual
\end{aligned}
\quad\text{and}\quad
\begin{aligned}
    \Xi_\mathds{1}^{-1} \colon \FC &\to \Ab \times \Vect(k) \\
    M &\mapsto (\mathrm{D}(M), \mathrm{V}(M)).
\end{aligned}
\]
These equivalences give the following diagram, which commutes up to a natural isomorphism.
\[\begin{tikzcd}[column sep=4em]
    \Sh(X, \Mod^\et_X) \times \Sh(X, \Rep^\alg_X) \arrow[r, "\Xi_{\Sh(X)}"] \arrow[d, "\Gamma(X{,} -)^{\Pi_1^\et(X)} \times \Gamma(X{,} -)^{\Pi_1^\alg(X)}"'] & \Sh(X, \FMod^\alg_X) \arrow[d, "\Gamma(X{,} -)^{\Pi_1^\alg(X)}"] \\
    \Ab \times \Vect(k) \arrow[r, "\Xi_\mathds{1}"'] & \FC
\end{tikzcd}\]
Since $\Xi_{\Sh(X)}$ and $\Xi_\mathds{1}$ are exact equivalences, they preserve injective objects. Passing to right derived functors gives natural isomorphisms
\begin{equation}\label{eq:unified_cohomology_decomposition}
    H^i_\uni\bigl(X, \Xi_{\Sh(X)}(\mathcal{D}, \mathcal{V})\bigr) \xrightarrow{\sim} \Xi_\mathds{1}\bigl(H^i_\fet(X, \mathcal{D}), H^i_\alg(X, \mathcal{V})\bigr).
\end{equation}
Applying this to the objects arising from the \'etale and algebraic de Rham comparison functors yields the main theorem below.

\begin{theorem}[Comparison Theorem]\label{thm:uni_comparison}
    Let $X$ be a smooth variety over an algebraically closed field $k$ of characteristic $0$. Let $\mathcal{D}$ be a filtered colimit of locally constant \'etale sheaves of finite abelian groups, and let $\mathcal{V}$ be an ind-regular singular $\mathcal{D}_X$-module. Set
    \[ \mathcal{M} \coloneqq \Xi_{\Sh(X)}\bigl(\Phi(\delta_* \mathcal{D}), \delta_* \mathcal{V}\bigr). \]
    Then the natural map
    \[ H^i_\uni(X, \mathcal{M}) \to (\diag H^i(X_\et, \mathcal{D}))^\dual \times (\vect H^i_\dR(X, \mathcal{V}))^\dual \]
    is an isomorphism for all $i \ge 0$. Equivalently, the natural maps
    \[ \mathrm{D}(H^i_\uni(X, \mathcal{M})) \to H^i(X_\et, \mathcal{D}) \quad\text{and}\quad \mathrm{V}(H^i_\uni(X, \mathcal{M})) \to H^i_\dR(X, \mathcal{V}) \]
    are isomorphisms for all $i \ge 0$.
\end{theorem}

\section{Further Questions}\label{sec:further_remarks}

Throughout this section, let $X$ be a variety over a field $k$, with further hypotheses imposed as needed. We have treated singular, \'etale, and algebraic de Rham cohomology using their corresponding fundamental groupoids and coefficient categories. We conclude with several possible extensions of this framework and questions for further research.

\subsection*{\texorpdfstring{Toward $p$-Adic Cohomology}{Toward p-Adic Cohomology}}

The most natural question is whether this framework also yields $p$-adic cohomology over a perfect field $k$ of characteristic $p > 0$. We expect that the Nori fundamental groupoid $\Pi_1^\N$, with coefficients in the category $\FC$ of commutative formal groups, yields $p$-adic cohomology alongside \'etale cohomology. This would provide a uniform description of the four major cohomology theories of arithmetic geometry.

Let $X$ be a reduced variety over a perfect field $k$. For $U \in \Open(X)$, let
\[ \FTor(U) \coloneqq \{ \text{torsors over $U$ under finite $k$-group schemes} \}. \]
Define the fibration $(\FMod^\N_X)^\op \to \Open(X)$ by $\FMod^\N_X(U) \coloneqq \AbObj(\Ind(\FTor(U)))$, with restrictions induced by pullback of torsors. If $U$ admits a $k$-point $x$, Nori's construction \cite{Nori} yields an equivalence $\FMod^\N_X(U) \simeq \FMod(\pi_1^\N(U, x))$, justifying this definition. One may show that $\FMod^\N_X$ is an SC4 category. Hence, for $i \ge 0$, we obtain cohomology functors
\[ H^i_\N(X, -) \colon \Sh(X, \FMod^\N_X) \to \FC. \]

Let $\mathrm{CW} \in \FC$ denote the formal $p$-group representing Fontaine's Witt covectors \cite[Ch.~II, \S1, and Ch.~III, \S I.1]{Fontaine}. Let $\W$ denote the affine group scheme representing the $p$-typical Witt vectors.

\padicCohomologyConjecture

The Nori fundamental groupoid detects finite torsors of degree divisible by $p$ that are invisible to the \'etale fundamental groupoid. For example, let $A$ be an ordinary elliptic curve over an algebraically closed field $k$ of characteristic $p > 0$. Then
\[ \pi_1^\et(A, 0) \cong \mathbb{Z}_p \times \prod_{\ell \neq p} \mathbb{Z}_\ell^2 \quad\text{while}\quad \pi_1^\N(A, 0) \cong \varprojlim_n \mu_{p^n} \times \underline{\mathbb{Z}_p} \times \prod_{\ell \neq p} \underline{\mathbb{Z}_\ell}^2. \]
On the coefficient side, Fontaine identifies $\Hom_\FC(M, \mathrm{CW})$ with the Dieudonn\'e module of a commutative formal $p$-group $M$ \cite[Ch.~III, Th\'eor\`eme~1]{Fontaine}. This plays the role of first $p$-adic group cohomology of $M$. We expect this behavior to extend to the $p$-adic cohomology of $X$.

One caveat is that $H^i_\N(X, \underline{\mathrm{CW}})$ may contain summands with no counterpart in classical crystalline cohomology. Besides $\mathrm{CW}$ and $\W_n$, possible factors include $\mathbb{Z}/p\mathbb{Z}$, $\boldsymbol{\alpha}_p$, and $\boldsymbol{\mu}_p$, although it is unclear whether these factors actually occur. Here, $\W_n \in \FC$ denotes the formal group representing the truncated $p$-typical Witt vector group functor of length $n$.

\subsection*{\texorpdfstring{Toward $p$-Adic Group Cohomology}{Toward p-Adic Group Cohomology}}

By \Cref{def:sheaf_cohomology}, each sheaf cohomology theory considered here comes with a corresponding groupoid cohomology theory. For $H^\bullet_\bt$, $H^\bullet_\fet$, and $H^\bullet_\alg$, these are $H^\bullet_\grp$, $H^\bullet_\cnt$, and $H^\bullet_\rat$, respectively. Likewise, $H^\bullet_\uni$ and $H^\bullet_\N$ come with the groupoid cohomology $H^\bullet_\fm$ of affine groupoid schemes with coefficients in $\FC$. This theory serves as a prototype for the corresponding cohomology of varieties and exhibits analogous phenomena. For example, given an affine groupoid scheme $G$ over $k$ of characteristic $p > 0$, the inclusion $\underline{\mathbb{Q}_p/\mathbb{Z}_p} \hookrightarrow \mathrm{CW}$ in $\FC$ induces a map
\[ H^i_\fm(G, \underline{\mathbb{Q}_p/\mathbb{Z}_p}) \to H^i_\fm(G, \mathrm{CW}). \]
Thus, $H^\bullet_\fm$ deserves investigation in its own right.

The coefficient category $\FC$ has a characteristic-dependent injective cogenerator. Let $G_k$ be the absolute Galois group of $k$, and let $\E \in \FC$ denote the \'etale descent of the $G_k$-module $\Map_\cnt(G_k, \mathbb{Q}/\mathbb{Z})$. If $\charac k = 0$, then $\E \oplus \mathbb{G}_a^\vee$ is an injective cogenerator of $\FC$. If, moreover, $k$ is algebraically closed, we showed in \Cref{sec:decomposition_cohomology} that its two summands recover \'etale and algebraic de Rham cohomology, respectively. If $k$ is perfect of characteristic $p > 0$, then $\E \oplus \mathrm{CW}$ is an injective cogenerator by \cite[Ch.~III, Th\'eor\`eme~2]{Fontaine}. We expect its two summands to recover \'etale and $p$-adic cohomology, respectively.

\subsection*{Algebraic Poincar\'e Lemma}

Let $X$ be a smooth variety over a field $k$ of characteristic $0$. We conjecture that the Poincar\'e lemma holds in $\Sh(X, \Rep^\alg_X)$. The category $\Sh(X_\bt, \Ab)$ carries the usual theory of quasi-coherent sheaves, and, via \Cref{thm:fet_sheaf_equivalence}, $\Sh(X, \Mod^\et_X)$ carries an analogous theory. We expect this theory to extend to $\Sh(X, \Rep^\alg_X)$. Let $\underline{k}$ denote the constant sheaf associated with the one-dimensional trivial representation of $\Pi_1^\alg(X)$. If $k$ is algebraically closed, the comparison theorem \Cref{thm:alg_dR_comparison} gives
\[ H^i_\alg(X, \underline{k}) \xrightarrow{\sim} H^i_\dR(X) \]
for every $i \ge 0$. This comparison suggests that, under the expected extension of quasi-coherent sheaves, the de Rham complex
\[ 0 \to \underline{k} \to \mathcal{O}_X \to \Omega^1_X \to \cdots \]
is exact in $\Sh(X, \Rep^\alg_X)$.

\subsection*{Beyond Fundamental Groupoids}

Being an SC4 category is a relatively mild requirement, and many fibered categories should satisfy it without being constructed from fundamental groupoids. For example, $\QCoh_X$ is an SC5 category by \Cref{prop:QCoh_SC5}, yet does not arise from a fundamental groupoid. More generally, suitable categories $\mathcal{C}(U)$ may produce SC4 fibers by applying $\AbObj$, $\Ind$, or both. For example, the fibers used in this paper are $\AbObj(\Cov(U))$, $\AbObj(\Ind(\EtCov(U)))$, $\Ind(\MIC^\rs(U))$, and $\AbObj(\Ind(\FTor(U)))$.

This flexibility suggests three directions. First, unconventional choices of $\mathcal{C}(U)$ may produce genuinely exotic cohomology theories. Second, suitable fibered categories may provide an easy way to define the existing cohomology theories over a more general base scheme. Finally, one may seek an SC4 category $\mathcal{T}$ which, at least conjecturally, produces the universal cohomology theory.

If $\mathcal{T}$ produces the universal cohomology theory, its coefficient category $\mathcal{A} \coloneqq \mathcal{T}(\Spec k)$ would have to be large. We expect $\mathcal{A}$ to contain every commutative algebraic group and every abelian variety over $k$. Moreover, $\mathcal{A}$ may contain objects that are not ind-representable by group schemes, including objects encoding intermediate cohomology groups $H^i$ of an $n$-dimensional variety for $1 < i < 2n - 1$. Nevertheless, $\mathcal{A}$ may still be a familiar category, such as the category of abelian fppf sheaves over $\Spec k$. Ideally, the constant sheaf
\[ \underline{k^\times} \coloneqq \bigl(\Delta_{\mathfrak{G}(X)} \mathbb{G}_m\bigr)_X \]
associated with $\mathbb{G}_m$ would serve as the universal coefficient. Then for every smooth projective variety $X$, the resulting cohomology would satisfy $H^1_\mathcal{T}(X, \underline{k^\times}) \cong \mathbf{Pic}^\tau X$.

\subsection*{Other Fundamental Groupoids}

Other fundamental groupoids may be used in the same construction. In characteristic $0$, one may use the Tannaka groupoids $\Pi_1^\un$ and $\Pi_1^\diff$ associated with unipotent integrable connections and all integrable connections, respectively \cite{Deligne1}. Over a perfect field of positive characteristic, one may use the crystalline fundamental groupoid $\Pi_1^\crys$, whose vertex groups are defined in \cite[Definition~4.1.5]{Shiho}.

\section*{Acknowledgements}

I thank Jaehyeok Lee for reading an earlier version of this paper. His suggestion to formalize the construction led me to formulate conditions (\ref{axiom:SC1})--(\ref{axiom:SC5}).

\appendix
\crefalias{section}{appendix}

\section{Groupoids}\label{sec:appendix_groupoids}

This appendix collects basic background material on three types of groupoids used throughout the paper. We first treat classical groupoids as categories, then describe them via structure maps. We also discuss profinite groupoids and affine groupoid schemes. The material is well known and we claim no originality. For the first two subsections, we refer to \cite{IbortRodriguez} and \cite[Ch.~6]{Brown}.

\subsection{Groupoids as Categories}\label{sec:groupoids_as_categories}

We regard a groupoid as a category in which every morphism is an isomorphism. Unless otherwise stated, all groupoids are assumed to be small. A morphism of groupoids $H \to G$ is a functor, and groupoids together with their morphisms form a category $\Gpd$. For a groupoid $G$, we define the categories of $G$-sets, $G$-modules, and $G$-representations as
\[
\Set(G) \coloneqq \Fun(G, \Set), \quad
\Mod(G) \coloneqq \Fun(G, \Ab), \quad\text{and}\quad
\Rep(G) \coloneqq \Fun(G, \Vect(\mathbb{C})),
\]
respectively. The full subcategories $\FSet(G) \subseteq \Set(G)$ and $\FRep(G) \subseteq \Rep(G)$ of finite $G$-sets and finite-dimensional $G$-representations are defined similarly. For any coefficient category $\mathcal{A}$, a morphism in $\Fun(G, \mathcal{A})$ is in general a natural transformation of functors. In particular, the morphisms in $\Set(G)$, $\Mod(G)$, and $\Rep(G)$ are natural transformations.

A groupoid $G$ is \emph{transitive} if it is nonempty and any two objects of $G$ are isomorphic. A \emph{connected component} of $G$ is the full subcategory spanned by an isomorphism class of objects. We denote the set of connected components of $G$ by $\pi_0(G)$. Then $G$ decomposes canonically as
\[ G = \coprod_{C \in \pi_0(G)} C \]
into transitive groupoids $C$.

Let $S$ denote the set of objects of $G$, and let $u \colon T \to S$ be a map. Define the groupoid $G_T$ to have object set $T$, where a morphism from $x$ to $y$ in $G_T$ is a morphism $u(x) \to u(y)$ in $G$. The groupoid morphism $G_T \to G$ is induced by $x \mapsto u(x)$. We regard a group as a one-object groupoid whose automorphism group is the given group. For an object $x$ of $G$, the \emph{vertex group} at $x$ is the group $G_x \coloneqq G_{\{x\}}$.

If $G$ is transitive and $T$ is nonempty, then precomposition with $G_T \to G$ is an equivalence
\[ \Fun(G, \mathcal{A}) \xrightarrow{\sim} \Fun(G_T, \mathcal{A}) \]
for every coefficient category $\mathcal{A}$. In particular, if $x$ is an object of $G$, then
\[ \Set(G) \simeq \Set(G_x), \quad \Mod(G) \simeq \Mod(G_x), \quad\text{and}\quad \Rep(G) \simeq \Rep(G_x). \]

Given a groupoid morphism $\varphi \colon H \to G$, the \emph{restriction functor}
\[ \Res_H^G \colon \Set(G) \to \Set(H) \]
is defined by $\Res_H^G M \coloneqq M \circ \varphi$ and is exact.

\begin{lemma}\label{lem:coind_exists}
The restriction functor $\Res_H^G$ fits into an adjunction
\[ \Res_H^G : \Set(G) \rightleftarrows \Set(H) : \Coind_H^G. \]
\end{lemma}
\begin{proof}
The \emph{coinduction functor} $\Coind_H^G$ is the right Kan extension along $\varphi$. By the end formula \cite[Ch.~X, \S4]{MacLane}, for $N \in \Set(H)$ and an object $y$ of $G$,
\[
(\Coind_H^G N)(y) \cong \int_{h \in H} N(h)^{\Hom_G(y, \varphi(h))} \cong \Hom_{\Set(H)}(\Hom_G(y, \varphi(-)), N).
\]
This gives an explicit formula for the right adjoint.
\end{proof}
The same construction applies to $\Mod(G)$ and $\Rep(G)$, where $N$ is regarded as an $H$-set via the forgetful functor. For composable groupoid morphisms $K \to H \to G$, the restriction functors satisfy the cocycle condition
\[ \Res_K^G = \Res_K^H \circ \Res_H^G. \]
Thus, as right adjoints, the coinduction functors satisfy
\[ \Coind_K^G \cong \Coind_H^G \circ \Coind_K^H. \]

Let $\mathds{1}$ denote the terminal groupoid, which has a single object and a single morphism. The trivial $G$-set functor and the $G$-invariants functor are given by
\[ \Delta_G \coloneqq \Res_G^{\mathds{1}} \quad\text{and}\quad (-)^G \coloneqq \Coind_G^{\mathds{1}}, \]
respectively. Let $\mathbf{1}_\Set$ be the singleton set in $\Set = \Set(\mathds{1})$, so that $\Hom_\Set(\mathbf{1}_\Set, -) \cong \id_\Set$, and let $\mathbf{1}_{\Set(G)} \coloneqq \Delta_G(\mathbf{1}_\Set)$. By the adjunction $\Delta_G \dashv (-)^G$,
\[ \Hom_{\Set(G)}(\mathbf{1}_{\Set(G)}, -) \cong \Hom_\Set(\mathbf{1}_\Set, (-)^G) \cong (-)^G. \]
The same applies to $\Mod(G)$ with $\mathbf{1}_{\Mod(G)} = \mathbb{Z}$ and $\Rep(G)$ with $\mathbf{1}_{\Rep(G)} = \mathbb{C}$.

\begin{example}
The \emph{fundamental groupoid} $\Pi_1(X)$ of a topological space $X$ is a central example for this paper. Its objects are the points of $X$, and its morphisms from $x$ to $y$ are homotopy classes of paths from $x$ to $y$ in $X$. The vertex group at $x$ is the fundamental group $\pi_1(X, x)$. If $X$ is path-connected, locally path-connected, and semilocally simply connected, then for each $x \in X$, let $P_x$ be the set of homotopy classes of paths in $X$ starting at $x$. It carries a natural topology under which $P_x \to X$ is the universal cover of $X$ based at $x$.
\end{example}

\subsection{Groupoids via Structure Maps}\label{sec:groupoids_via_structure_maps}

This subsection reformulates the definitions of groupoids and their representations from \Cref{sec:groupoids_as_categories} in terms of structure maps. This formulation is less elegant but better suited for the algebro-geometric setting in \Cref{sec:affine_groupoids}.

\begin{definition}\label{def:groupoid_structure_maps}
Let $G$ and $S$ be sets with structure maps
\[ s, t \colon G \to S, \quad e \colon S \to G, \quad i \colon G \to G, \quad c \colon G^{(2)} \to G, \]
where $G^{(2)} \coloneqq G \times_{s, S, t} G$ with projections $\pr_0, \pr_1 \colon G^{(2)} \to G$. We call $G$ a \emph{groupoid} acting on $S$ if the following axioms are satisfied.
\begin{enumerate}
    \item $s \circ e = t \circ e = \id_S$
    \item $s \circ i = t$ and $t \circ i = s$
    \item $s \circ c = s \circ \pr_1$ and $t \circ c = t \circ \pr_0$
    \item $c \circ (e \circ t, \id_G) = c \circ (\id_G, e \circ s) = \id_G$
    \item $c \circ (i, \id_G) = e \circ s$ and $c \circ (\id_G, i) = e \circ t$
    \item $c \circ (\id_G \times_S c) = c \circ (c \times_S \id_G)$
\end{enumerate}
\end{definition}

\begin{remark}\label{rem:groupoid_categorical}
In terms of the categorical description of \Cref{sec:groupoids_as_categories}, the elements of $S$ and $G$ are the objects and morphisms of the groupoid. Fix objects $x, y, z \in S$ and morphisms $f, g, h \in G$. We set $\id_x \coloneqq e(x)$ and $f^{-1} \coloneqq i(f)$. We say $f \colon x \to y$ if and only if $s(f) = x$ and $t(f) = y$. If $f \colon x \to y$ and $g \colon y \to z$, we write $g \circ f \coloneqq c(g, f)$. The axioms above then translate as follows.
\begin{enumerate}
    \item $\id_x \colon x \to x$
    \item $f^{-1} \colon y \to x$
    \item $g \circ f \colon x \to z$
    \item $\id_y \circ f = f = f \circ \id_x$
    \item $f^{-1} \circ f = \id_x$ and $f \circ f^{-1} = \id_y$
    \item $(h \circ g) \circ f = h \circ (g \circ f)$
\end{enumerate}
\end{remark}

When multiple groupoids are present, we distinguish their structure maps by subscripts.

\begin{definition}\label{def:groupoid_morphism_structure_maps}
Let $H$ and $G$ be groupoids acting on $T$ and $S$, respectively. A \emph{morphism of groupoids} from $H$ to $G$ is a pair $(\varphi, u)$, where $\varphi \colon H \to G$ and $u \colon T \to S$ are maps satisfying the following.
\begin{enumerate}
    \item $s_G \circ \varphi = u \circ s_H$ and $t_G \circ \varphi = u \circ t_H$
    \item $\varphi \circ e_H = e_G \circ u$
    \item $\varphi \circ i_H = i_G \circ \varphi$
    \item $\varphi \circ c_H = c_G \circ (\varphi \times_u \varphi)$
\end{enumerate}
\end{definition}

\begin{remark}\label{rem:groupoid_morphism_categorical}
For morphisms $f \colon x \to y$ and $g \colon y \to z$ in $H$, the conditions above translate as follows.
\begin{enumerate}
    \item $\varphi(f) \colon u(x) \to u(y)$
    \item $\varphi(\id_x) = \id_{u(x)}$
    \item $\varphi(f^{-1}) = \varphi(f)^{-1}$
    \item $\varphi(g \circ f) = \varphi(g) \circ \varphi(f)$
\end{enumerate}
\end{remark}

\begin{definition}\label{def:gset_structure_maps}
A \emph{$G$-set} is a pair $(M, \rho)$ with a map $M \to S$, where
\[ \rho \colon s^* M \xrightarrow{\sim} t^* M \]
is a bijection of sets over $G$ satisfying
\[ e^* \rho = \id_M \quad\text{and}\quad c^* \rho = \pr_0^* \rho \circ \pr_1^* \rho. \]
A morphism $(M, \rho) \to (N, \sigma)$ is a map $\eta \colon M \to N$ of sets over $S$ satisfying
\[ t^* \eta \circ \rho = \sigma \circ s^* \eta. \]
\end{definition}

\begin{remark}\label{rem:gset_categorical}
For composable morphisms $f \colon x \to y$ and $g \colon y \to z$ in $G$, the conditions above translate as follows. The map
\[ \rho(f) \colon M_x \xrightarrow{\sim} M_y \]
is a bijection satisfying
\[ \rho(\id_x) = \id_{M_x} \quad\text{and}\quad \rho(g \circ f) = \rho(g) \circ \rho(f). \]
A morphism $\eta \colon M \to N$ satisfies
\[ \eta_y \circ \rho(f) = \sigma(f) \circ \eta_x. \]
\end{remark}

\begin{definition}\label{def:restriction_structure_maps}
Given a morphism of groupoids $(\varphi, u) \colon (H, T) \to (G, S)$ and a $G$-set $(M, \rho)$, the \emph{restriction} of $(M, \rho)$ along $(\varphi, u)$ is the $H$-set $\Res_H^G(M, \rho) \coloneqq (u^* M, \varphi^* \rho)$, where
\[ \varphi^* \rho \colon s_H^* u^* M \xrightarrow{\sim} t_H^* u^* M \]
is induced by $\rho$ via $u \circ s_H = s_G \circ \varphi$ and $u \circ t_H = t_G \circ \varphi$. 
\end{definition}

\begin{remark}\label{rem:restriction_categorical}
For a morphism $f \colon x \to y$ in $H$, we have an isomorphism
\[ (\varphi^* \rho)(f) = \rho(\varphi(f)) \colon M_{u(x)} \xrightarrow{\sim} M_{u(y)}. \]
\end{remark}

\subsection{Profinite Groupoids}\label{sec:profinite_groupoids}

The main reference for this subsection is \cite[Exp.~V]{SGA1}. We define a \emph{transitive profinite groupoid} as a groupoid arising from a Galois category $\mathcal{C}$ with a nonempty collection of chosen fiber functors. Its objects are the chosen fiber functors, its morphisms are isomorphisms of fiber functors, and its vertex groups carry the profinite topology.

Let $G$ and $H$ be transitive profinite groupoids arising from Galois categories $\mathcal{C}$ and $\mathcal{D}$, respectively. A morphism $\varphi \colon H \to G$ is an exact functor
\[ \Res_H^G \colon \mathcal{C} \to \mathcal{D} \]
such that for every object $F$ of $H$, the fiber functor $F \circ \Res_H^G$ is an object of $G$. The induced groupoid morphism is given on objects by $F \mapsto F \circ \Res_H^G$. Transitive profinite groupoids then form a category.

For a transitive profinite groupoid $G$ arising from $\mathcal{C}$, we define the categories of \emph{finite discrete $G$-sets}, \emph{discrete $G$-sets}, and \emph{discrete $G$-modules} by
\[ \FSet(G) \coloneqq \mathcal{C}, \quad \Set(G) \coloneqq \Ind(\mathcal{C}), \quad\text{and}\quad \Mod(G) \coloneqq \AbObj(\Ind(\mathcal{C})), \]
respectively. For every object $x$ of $G$, restriction to the vertex group $G_x$, regarded as a profinite group, induces equivalences
\[ \FSet(G) \xrightarrow{\sim} \FSet(G_x), \quad \Set(G) \xrightarrow{\sim} \Set(G_x), \quad\text{and}\quad \Mod(G) \xrightarrow{\sim} \Mod(G_x). \]

The restriction functor
\[ \Res_H^G \colon \FSet(G) \to \FSet(H) \]
naturally extends to restriction functors
\[ \Res_H^G \colon \Set(G) \to \Set(H) \quad\text{and}\quad \Res_H^G \colon \Mod(G) \to \Mod(H). \]
Reduction to vertex groups shows that both restriction functors are exact and admit right adjoints, called the \emph{coinduction functors}. This gives adjunctions
\begin{align*}
\Res_H^G : \Set(G) &\rightleftarrows \Set(H) : \Coind_H^G \quad\text{and} \\
\Res_H^G : \Mod(G) &\rightleftarrows \Mod(H) : \Coind_H^G.
\end{align*}

\begin{example}
For a nonempty connected scheme $X$, the \emph{\'etale fundamental groupoid} $\Pi_1^\et(X)$ \cite[Exp.~V, \S7]{SGA1} is a central example in this paper.\footnote{SGA~1 requires $X$ to be locally Noetherian, but this condition is unnecessary by \cite[\href{https://stacks.math.columbia.edu/tag/0BNB}{Tag~0BNB}]{StacksProject}.} The category $\EtCov(X)$ of finite \'etale covers of $X$ is a Galois category, and every geometric point $x$ of $X$ defines a fiber functor $E \mapsto E_x$. The \'etale fundamental groupoid $\Pi_1^\et(X)$ is the transitive profinite groupoid arising from $\EtCov(X)$ and the collection of these fiber functors. We may therefore regard its objects as the geometric points of $X$ and its morphisms as isomorphisms between the corresponding fiber functors. The vertex group at a geometric point $x$ is the \'etale fundamental group $\pi_1^\et(X, x)$. Although this makes $\Pi_1^\et(X)$ a large category, the essential smallness of $\EtCov(X)$ ensures that no set-theoretic difficulties arise in this paper.
\end{example}

A \emph{profinite groupoid} is a formal finite coproduct
\[ G = \coprod_{C \in \pi_0(G)} C \]
of transitive profinite groupoids. A morphism $\varphi \colon H \to G$ consists of a map
\[ \sigma \colon \pi_0(H) \to \pi_0(G) \]
together with a morphism $D \to \sigma(D)$ of transitive profinite groupoids for every $D \in \pi_0(H)$. These objects and morphisms form the category $\PfGpd$ of profinite groupoids. Given $G \in \PfGpd$, we define
\[ \FSet(G) \coloneqq \prod_{C \in \pi_0(G)} \FSet(C). \]
We then set $\Set(G) \coloneqq \Ind(\FSet(G))$ and $\Mod(G) \coloneqq \AbObj(\Ind(\FSet(G)))$.

Let $X$ be a scheme with finitely many connected components. We define its \emph{\'etale fundamental groupoid} by
\[ \Pi_1^\et(X) \coloneqq \coprod_{Z \in \pi_0(X)} \Pi_1^\et(Z). \]

\subsection{Affine Groupoid Schemes}\label{sec:affine_groupoids}

The main reference for this subsection is \cite{Deligne}; see also \cite[\href{https://stacks.math.columbia.edu/tag/022L}{Tag~022L}]{StacksProject}. A \emph{groupoid scheme} is a groupoid object in the category of schemes. Equivalently, it is a pair of schemes $(G, S)$ equipped with structure maps as in \Cref{def:groupoid_structure_maps}. Morphisms of groupoid schemes are defined as in \Cref{def:groupoid_morphism_structure_maps}. We say that $G$ is a \emph{transitive affine groupoid scheme} acting on $S$ if $S$ is nonempty and $(s, t) \colon G \to S \times S$ is faithfully flat and affine. These objects form a category. They are precisely the groupoids arising from Tannaka duality \cite[Th\'eor\`eme~1.12]{Deligne}.

\begin{definition}[see {\cite[1.6]{Deligne}}]\label{def:rep_groupoid_scheme}
A \emph{representation} of $G$ is a pair $(M, \rho)$, where $M$ is a quasi-coherent sheaf on $S$ and
\[ \rho \colon s^* M \xrightarrow{\sim} t^* M \]
is an isomorphism of quasi-coherent sheaves on $G$ satisfying
\[ e^* \rho = \id_M \quad\text{and}\quad c^* \rho = \pr_0^* \rho \circ \pr_1^* \rho. \]
A morphism $(M, \rho) \to (N, \sigma)$ is a morphism $\eta \colon M \to N$ of quasi-coherent sheaves on $S$ satisfying
\[ t^* \eta \circ \rho = \sigma \circ s^* \eta. \]
We denote the category of representations of $G$ by $\Rep(G)$, and its full subcategory consisting of representations that are locally free of finite rank on $S$ by $\FRep(G)$.
\end{definition}

Let $G$ be the Tannaka groupoid of a Tannakian category $\mathcal{S}$ associated with a fiber functor over a scheme $S$. Then
\[ \FRep(G) \simeq \mathcal{S} \quad\text{and}\quad \Rep(G) \simeq \Ind(\mathcal{S}). \]
The first equivalence follows from \cite[Th\'eor\`eme~1.12]{Deligne}, and the second from the lemma below. The lemma is a folklore fact, but we do not know a convenient reference.

\begin{lemma}\label{lem:rep_ind}
Let $G$ be a transitive affine groupoid scheme acting on a scheme $S$. Then the natural functor
\[ \Ind(\FRep(G)) \to \Rep(G) \]
is an equivalence of categories.
\end{lemma}
\begin{proof}
By \cite[Remarque~1.8, (3.5.1)]{Deligne}, we may assume that $S$ is a nonempty affine scheme. Since $\FRep(G)$ is essentially small, we may apply \cite[Theorem~A.4]{ACMU}. It suffices to verify the following three conditions. First, $\Rep(G)$ has filtered colimits. These are computed on the underlying quasi-coherent sheaves because $s^*$ and $t^*$ commute with filtered colimits. Second, every $N \in \Rep(G)$ is a filtered union of objects in $\FRep(G)$ by \cite[Corollaire~3.9]{Deligne}. Third, every $(M, \rho) \in \FRep(G)$ is finitely presentable in $\Rep(G)$. For $(N, \sigma) \in \Rep(G)$,
\[ \Hom_{\Rep(G)}(M, N) \cong \Eq\!\left(\Hom_{\QCoh(S)}(M, N) \rightrightarrows \Hom_{\QCoh(G)}(s^* M, t^* N)\right), \]
where the two maps send $f \colon M \to N$ to $t^* f \circ \rho$ and $\sigma \circ s^* f$, respectively. Since $S$ and $G$ are affine, and $M$ and $s^* M$ are locally free of finite rank, \cite[\href{https://stacks.math.columbia.edu/tag/07U7}{Tag~07U7}]{StacksProject} shows that $\Hom_{\QCoh(S)}(M, -)$ and $\Hom_{\QCoh(G)}(s^* M, -)$ commute with filtered colimits. Moreover, $t^*$ commutes with filtered colimits, and finite limits commute with filtered colimits in $\Set$. Hence, $\Hom_{\Rep(G)}(M, -)$ commutes with filtered colimits.
\end{proof}

Pullback along a morphism
\[ (\varphi, u) \colon (H, T) \to (G, S) \]
of transitive affine groupoid schemes defines restriction functors
\[ \Res_H^G \colon \FRep(G) \to \FRep(H) \quad\text{and}\quad \Res_H^G \colon \Rep(G) \to \Rep(H). \]
Under Tannaka duality, the first restriction functor is an exact tensor functor, and the second is its extension to the ind-categories. Therefore, $\Res_H^G \colon \Rep(G) \to \Rep(H)$ is exact by \cite[Corollary~8.6.8]{KashiwaraSchapira}.

Let $u \colon T \to S$ be a morphism with $T$ nonempty. The fiber product $G_T \coloneqq T \times_{u, S, s} G \times_{t, S, u} T$ carries the natural structure of a transitive affine groupoid scheme acting on $T$ \cite[1.6]{Deligne}. The projection $\varphi \colon G_T \to G$ onto the middle factor, together with $u$, defines a morphism
\[ (\varphi, u) \colon (G_T, T) \to (G, S) \]
of transitive affine groupoid schemes. The induced restriction functors
\[ (-)_T \colon \FRep(G) \xrightarrow{\sim} \FRep(G_T) \quad\text{and}\quad (-)_T \colon \Rep(G) \xrightarrow{\sim} \Rep(G_T) \]
are equivalences \cite[Remarque~1.8, (3.5.1)]{Deligne}.

If $x \in S$ is a rational point, the vertex group scheme $G_x \coloneqq G_{\{x\}}$ is an affine group scheme over $k$. The categories $\FRep(G_x)$ and $\Rep(G_x)$ are the usual categories of finite-dimensional and arbitrary representations of $G_x$, respectively. Although $S$ need not have a rational point, one may always choose $T$ to be affine.

\begin{example}
The \emph{pro-algebraic fundamental groupoid} $\Pi_1^\alg(X)$ of a smooth geometrically connected variety $X$ over a field $k$ of characteristic $0$ is a central example in this paper. It is the Tannaka groupoid of the Tannakian category
\[ \MIC^\rs(X) \coloneqq \{ \text{vector bundles with regular singular integrable connections on } X \} \]
with respect to the fiber functor over $X$ that forgets the connection \cite[10.26 and 10.27(i)]{Deligne1}. For a rational point $x \in X(k)$, its vertex group
\[ \pi_1^\alg(X, x) \coloneqq \Pi_1^\alg(X)_x \]
is the \emph{pro-algebraic fundamental group} of $X$ at $x$.
\end{example}

\begin{definition}\label{def:proalg_completion}
The \emph{pro-algebraic completion} $G^\alg$ of an abstract group $G$ is the Tannaka group of $\FRep(G)$ with respect to the forgetful fiber functor
\[ \FRep(G) \to \Vect(k). \]
\end{definition}

Over $\mathbb{C}$, there are equivalences of categories \cite[10.25]{Deligne1}, given by
\begin{align*}
\MIC^\rs(X) &\simeq \{ \text{analytic vector bundles with flat connections on } X^\an \} \\
&\simeq \{ \text{finite-dimensional representations of } \Pi_1(X) \}.
\end{align*}
Therefore, for every $x \in X(\mathbb{C})$, the pro-algebraic fundamental group $\pi_1^\alg(X, x)$ is the pro-algebraic completion of $\pi_1(X, x)$.

\begin{example}
The \emph{Nori fundamental groupoid} $\Pi_1^\N(X)$ of a geometrically connected and geometrically reduced variety $X$ over a field $k$ is another central example in this paper. For a rational point $x \in X(k)$, its vertex group
\[ \pi_1^\N(X, x) \coloneqq \Pi_1^\N(X)_x \]
is the \emph{Nori fundamental group scheme} of $X$ at $x$. It may be described as the inverse limit of the structure groups of finite torsors on $X$ based at $x$, as first suggested by Nori \cite{Nori}. When $\charac k = p > 0$, it also detects finite non-\'etale torsors whose degrees are divisible by $p$, which are invisible to the \'etale fundamental group.

The groupoid $\Pi_1^\N(X)$ arises as the Tannaka groupoid of a suitable Tannakian category. For proper $X$, this category was constructed by Nori \cite[p.~82]{Nori}. For smooth, not necessarily proper, $X$, it was constructed by Esnault--Hai in characteristic $0$ \cite[Lemma~4.3]{EsnaultHai2} and by Esnault--Hogadi in positive characteristic \cite[Corollary~4.9]{EsnaultHogadi}.
\end{example}

An \emph{affine groupoid scheme} is a finite disjoint union of transitive affine groupoid schemes. We write $\pi_0(G)$ for the set of its connected components, so
\[ G = \coprod_{C \in \pi_0(G)} C. \]
Affine groupoid schemes form the category $\AfGpd$. For $G \in \AfGpd$, restriction to the connected components induces natural equivalences
\[ \FRep(G) \xrightarrow{\sim} \prod_{C \in \pi_0(G)} \FRep(C) \quad\text{and}\quad \Rep(G) \xrightarrow{\sim} \prod_{C \in \pi_0(G)} \Rep(C). \]

Let $X$ be a smooth variety over a field $k$ of characteristic $0$ with geometrically connected components. We define its pro-algebraic fundamental groupoid by
\[ \Pi_1^\alg(X) \coloneqq \coprod_{Z \in \pi_0(X)} \Pi_1^\alg(Z). \]
Similarly, for a variety $X$ over $k$ with geometrically connected and geometrically reduced components, we define its Nori fundamental groupoid by
\[ \Pi_1^\N(X) \coloneqq \coprod_{Z \in \pi_0(X)} \Pi_1^\N(Z). \]

\raggedbottom
\bibliographystyle{plain}
\bibliography{references}

\end{document}